\documentclass{amsart}
\usepackage[utf8]{inputenc}
\usepackage{amsmath,amsthm}
\usepackage{comment}
\usepackage{amssymb}
\usepackage{mathrsfs}
\usepackage{stmaryrd}
\usepackage{MnSymbol}
\usepackage{hyperref}
\usepackage{color}
\usepackage[left=3cm,right=3cm,top=3cm,bottom=4.5cm]{geometry}
\usepackage{pdfsync}
\usepackage{tikz-cd}
\usepackage{url}
\usepackage{enumitem}
\usepackage{multicol}
\usepackage{listofitems}
\hypersetup{linktoc=all}
\usepackage[capitalise]{cleveref}

\newcommand{\problematic}[1]{\textcolor{red}{#1}}

\newcommand{\defterm}[1]{\emph{#1}}

\theoremstyle{plain}
	\newtheorem{thm}{Theorem}
	\numberwithin{thm}{section}
	\newtheorem*{thm*}{Theorem}
	\newtheorem{cor}[thm]{Corollary}
	\newtheorem*{cor*}{Corollary}
	\newtheorem{prop}[thm]{Proposition}
	\newtheorem*{prop*}{Proposition}
	\newtheorem{lem}[thm]{Lemma}
	\crefname{lem}{Lemma}{Lemmas}
	\newtheorem*{lem*}{Lemma}
	
	\newtheorem*{ex*}{Exercise}
	
	\newtheorem*{claim*}{Claim}
	
	\newtheorem*{question*}{Question}
	
	\newtheorem*{fact*}{Fact}
	\newtheorem{assume}{Assumption}
	\crefname{assume}{Assumption}{Assumptions}
\theoremstyle{definition}
	\newtheorem{Def}[thm]{Definition}
	\newtheorem*{Def*}{Definition}
	
	\newtheorem*{obs*}{Observation}
	\newtheorem{rmk}[thm]{Remark}
	\newtheorem*{rmk*}{Remark}
	
	\newtheorem{soln*}{Solution}
	
	\newtheorem*{note*}{Note}
	\newtheorem{eg}[thm]{Example}
	\newtheorem*{eg*}{Example}	
	
	\newtheorem*{construction*}{Construction}
	
	\newtheorem*{warning*}{Warning}
	
	\newtheorem*{conj*}{Conjecture}
	\newtheorem{notation}[thm]{Notation}
	\newtheorem*{notation*}{Notation}
	
\crefname{prop}{Proposition}{Propositions}
\crefname{thm}{Theorem}{Theorems}
	
\newcommand{\op}{\mathrm{op}}
\newcommand{\id}{\mathrm{id}}
\newcommand{\Id}{\mathrm{Id}}

\DeclareMathOperator{\dom}{\mathrm{dom}}
\DeclareMathOperator{\cod}{\mathrm{cod}}

\newcommand{\Mor}{\operatorname{Mor}}

\newcommand{\Set}{\mathsf{Set}}

\newcommand{\Cat}{\mathsf{Cat}}

\newcommand{\sSet}{\mathsf{sSet}}

\newcommand{\calC}{\mathcal{C}}
\newcommand{\calD}{\mathcal{D}}

\newcommand{\calK}{\mathcal{K}}

\newcommand{\calM}{\mathcal{M}}
\newcommand{\calO}{\mathcal{O}}

\newcommand{\mSet}{\mathsf{sSet}^+}

\newcommand{\trunc}[1]{{\tau_{#1}}}
\newcommand{\PreComp}{\mathsf{PreComp}}

\newcommand{\precomp}{\mathrm{pre}}
\newcommand{\join}{\star}
\renewcommand{\Im}{\operatorname{Im}}
\newcommand{\co}{\mathrm{co}}
\newcommand{\coop}{\mathrm{co-op}}
\newcommand{\core}[1]{\mathrm{core}_{#1}}

\newcommand{\sDelta}[1]{\Delta^{#1}} 
\newcommand{\mDelta}[1]{\widetilde{\Delta^{#1}}} 
\newcommand{\cDelta}[2]{\Delta_{#1}^{#2}} 
\newcommand{\horn}[2]{\Lambda_{#1}^{#2}} 
\newcommand{\cDeltap}[2]{{\Delta_{#1}^{#2}}'} 
\newcommand{\cDeltapp}[2]{{\Delta_{#1}^{#2}}''} 
\newcommand{\presat}{\Delta^3_{\mathrm{eq}}} 

\newcommand{\frontface}[2]{\upmodels_1^{#1,#2}}
\newcommand{\backface}[2]{\upmodels_2^{#1,#2}}

\newcommand{\im}{\operatorname{im}}
\newcommand{\scrI}{\mathscr{I}}
\newcommand{\colim}{\operatorname{colim}}
\newcommand{\face}[2]{\partial_{#1,#2}} 
\newcommand{\admdelta}[2]{\Delta^{#1}_{#2}}
\newcommand{\admdeltap}[2]{{\Delta^{#1}_{#2}}'}
\newcommand{\admdeltapp}[2]{{\Delta^{#1}_{#2}}''}
\newcommand{\faces}[1]{\partial_{#1}} 
\newcommand{\pseudo}{\varoast} 
\newcommand{\rezkd}[2]{L_{#1#2}}
\newcommand{\rezkc}[2]{L_{#1#2}'}
\newcommand{\incl}{\hookrightarrow}
\newcommand{\pretensor}{\otimes^\precomp}

\newcommand{\Cube}{\square}
\newcommand{\cSet}{\mathsf{cSet}}
\newcommand{\mcSet}{\mathsf{cSet}^+}

\newcommand{\mcube}[1]{\widetilde{\Cube}^{#1}}
\newcommand{\cube}[1]{\Cube^{#1}}
\newcommand{\cbdy}[1]{\partial \Cube^{#1}}
\newcommand{\admcube}[3]{\Cube^{#1}_{#2,#3}}
\newcommand{\admcubep}[3]{{\Cube^{#1}_{#2,#3}}'}
\newcommand{\admcubepp}[3]{{\Cube^{#1}_{#2,#3}}''}

\newcommand{\obox}[3]{\sqcap^{#1}_{#2, #3}}

\renewcommand{\epsilon}{\varepsilon}

\title{A cubical model for $(\infty, n)$-categories}
\author[T.~Campion, K.~Kapulkin, Y.~Maehara]{Timothy Campion \and Krzysztof Kapulkin \and Yuki Maehara}
\date{June 1, 2020}

\begin{document}
\begin{abstract}
  We propose a new model for the theory of $(\infty,n)$-categories (including the case $n=\infty$) in the category of marked cubical sets with connections, similar in flavor to complicial sets of Verity.
The model structure characterizing our model is shown to be monoidal with respect to suitably defined (lax and pseudo) Gray tensor products; in particular, these tensor products are both associative and biclosed.
Furthermore, we show that the triangulation functor to pre-complicial sets is a left Quillen functor and is strong monoidal with respect to both Gray tensor products.
\end{abstract}
\maketitle


\section*{Introduction}

The theory of $(\infty,n)$-categories is becoming an important tool in a number of areas of mathematics, including manifold topology, where it is used in the definition and classification of extended topological quantum field theories \cite{lurie:tqft}, and in derived algebraic geometry, where it is used to capture certain properties of the ``category'' of correspondences \cite{gaitsgory-rozenblyum:dag-1,gaitsgory-rozenblyum:dag-2}.
There are several equivalent models for this theory, including: $n$-trivial saturated complicial sets (\cite{verity:weak-complicial-1,verity:weak-complicial-2,riehl-verity:more-elements,loubaton:complicial}), $n$-quasicategories \cite{ara:higher-quasicat}, $\Theta_n$-spaces \cite{rezk:cartesian-presentation}, and $n$-fold complete Segal spaces \cite{barwick:thesis}.

In this paper, we propose a new model for the theory of $(\infty,n)$-categories, using \emph{comical sets} (\emph{com}position + cub\emph{ical sets}).
Comical sets are certain marked cubical sets (having marked $n$-cubes for all values $n \geq 1$), just like complicial sets are certain marked simplicial sets.
Our model allows for a particularly elegant and simple treatment of the (lax and pseudo) Gray tensor products since they are inherently cubical in nature.
One can find drawings of cubes in Gray's book \cite{gray:book}, and the simplest ways of defining the lax Gray tensor product of strict $\omega$-categories are via cubical sets \cite{crans:thesis,al-agl-brown-steiner}.

Because of the obvious similarities with complicial sets, there is a natural comparison functor to marked simplicial sets.
To obtain it, we extend the usual triangulation functor $T \colon \cSet \to \sSet$ from cubical sets to simplicial sets to a marked version $T \colon \mcSet \to \PreComp$.
Here $T$ is valued not in the whole category $\mSet$ of marked simplicial sets but in the reflective subcategory $\PreComp$ of pre-complicial sets so that our results hold up to isomorphism rather than homotopy.
$\PreComp$ supports a model structure that is Quillen equivalent to the complicial model structure on $\mSet$, and the lax Gray tensor product on $\mSet$ is more well-behaved when restricted to $\PreComp$.

With that, our main results (cf.~\cref{comical-model-structure,GrayFundamentalProp,T-strong-monoidal-lax,T-strong-monoidal-pseudo,left-Quillen}) can be summarized as follows:

\begin{thm*}
  The category $\mcSet$ of marked cubical sets carries a model structure whose cofibrations are the monomorphisms and whose fibrant objects are the comical sets.
  This model structure is monoidal with respect to both the lax and pseudo Gray tensor products, which are simultaneously associative and biclosed.
  
  Furthermore, the triangulation functor $T \colon \mcSet \to \PreComp$ is left Quillen and strong monoidal with respect to both Gray tensor products.
\end{thm*}

Since this paper was first made available in 2020, a slight adaptation of our model was proven to be Quillen equivalent via the triangulation functor to $n$-trivial saturated complicial sets in \cite{doherty-kapulkin-maehara}.
The ``special cases'' of this result had previously been known for $(\infty,0)$-categories (i.e., $\infty$-groupoids) \cite{CisinskiAsterisque} and $(\infty,1)$-categories \cite{doherty-kapulkin-lindsey-sattler}, although these papers consider slightly different versions of the cubical site from us.

In particular, our model validates the assertions \cite[Props.~10.3.2.6 and 10.3.2.9]{gaitsgory-rozenblyum:dag-1}, given there without a proof or a reference.
They are essentially the desiderata of a convenient model of $(\infty, 2)$-categories used throughout \cite{gaitsgory-rozenblyum:dag-1,gaitsgory-rozenblyum:dag-2}, and in that sense our model in $\cSet^+$ is a convenient such model.
We should note however that these assertions were previously proven in \cite{verity:weak-complicial-1} and \cite{maehara:gray} in the contexts of complicial sets and $2$-quasicategories respectively.

Finally, our work owes a great deal to \cite{steiner:cubical-nerve}, where the (semi-)cubical nerves of strict $\omega$-categories are analyzed.
In particular, our definition of comical open boxes in \cref{sec:model} follows \cite[Ex.~2.9]{steiner:cubical-nerve}.

\textbf{Organization of the paper.}
We begin in \cref{sec:background} by reviewing the necessary background on model categories, cubical sets, and complicial sets.
In \cref{sec:cubical}, we introduce marked cubical sets, study their basic properties, and construct both the lax and the pseudo Gray tensor products.
In \cref{sec:model}, we define comical sets and construct the model structure for them.
As a proof of concept, we define in \cref{homotopy-1-categories} the homotopy $1$-category of a comical set and show that it has expected properties.
We then turn our attention to the comparison between the cubical and the simplicial approaches.
We extend the triangulation functor to marked cubical sets in \cref{triangulation}, show that it is strong monoidal with respect to both the lax and the pseudo Gray tensor products in \cref{sec:triangulating-gray}, and that it is a left Quillen functor in \cref{sec:triangulating-quillen}.

\textbf{Acknowledgement.}
The authors benefited greatly from conversations about related matters with Alexander Campbell, Emily Riehl, and Dominic Verity.
The paper was greatly improved by the comments of the anonymous referee.
This material is based upon work supported by the National Science Foundation under Grant No.~1440140, while the authors were in residence at the Mathematical Sciences Research Institute in Berkeley, California, in the program ``Higher Categories and Categorification'' in Spring 2020.
We would like to thank the MSRI for its hospitality, and the program organizers for giving us the opportunity to participate.
Above all, we thank (again) Emily Riehl for continued support and encouragement.



\section{Background} \label{sec:background}

In this section we introduce the notation and collect preliminary results to be used later in the paper.

\subsection{Model categories}\label{subsec:background}

In this subsection, we review (a special case of) the theory of Olschok \cite{olschok:thesis} for constructing combinatorial model structures with all objects cofibrant, which generalizes the theory of Cisinski \cite{CisinskiAsterisque} for constructing combinatorial model structures on presheaf categories with cofibrations the monomorphisms. This theory will be used to construct the model structures for comical sets.

\begin{Def}[\cite{simpson:homotopy-of-higher-cats}]
	 We say a set $\Lambda$ of trivial cofibrations in a model category $\calM$ is \defterm{pseudo-generating} if and only if any map $f$ that has a fibrant codomain and the right lifting property with respect to $\Lambda$ is a fibration.
\end{Def}

Now fix a locally presentable category $\calK$.

Given a bifunctor $\odot \colon \calK \times \calK \to \calK$ and maps $f \colon A \to A'$, $g \colon B \to B'$ in $\calK$, we denote
\[
f \hat \odot g \colon (A' \odot B) \coprod_{A \odot B} (A \odot B') \to A' \odot B'
\]
the \defterm{Leibniz $\odot$-product} of $f$ and $g$.
Similarly, for any natural transformation $\phi \colon F \Rightarrow G$ between endofunctors $F,G\colon  \calK \to \calK$ and for any $f \colon A \to A'$ in $\calK$, we denote
\[
\hat \phi_f \colon G(A) \coprod_{F(A)} F(A') \to G(A')
\]
the \defterm{Leibniz product} of $\phi$ and $f$.

By the \defterm{cellular closure} of a set $S$ of maps in $\calK$, we mean the closure of $S$ under pushouts along arbitrary maps and transfinite composition.
In the rest of this subsection, assume that we are given a small set $I$ of maps in $\calK$ whose cellular closure is precisely the monomorphisms.

\begin{Def}
A \defterm{functorial cylinder} on $\calK$ is a functor $C\colon  \calK \to \calK$ equipped with natural transformations $\partial^0, \partial^1\colon  \Id \rightrightarrows C$ and $\sigma\colon  C \to \Id$ such that $\sigma\partial^0 = \sigma \partial^1 = \id$. We also write $\partial_X = [\partial^0_X,\partial^1_X]\colon  X + X \to CX$.
We say that $C$ is a \defterm{cartesian} cylinder if the functor $C$ preserves colimits and moreover $\partial_X$ is a monomorphism for all $X$.
\end{Def}

\begin{Def}
Suppose that $\calK$ admits a cartesian functorial cylinder $C = (C,\partial^0,\partial^1,\sigma)$.
Let $S$ be a set of morphisms in $\calK$.
We define $\Lambda(\calK, I , C, S) \subseteq \Mor \calK$ to be the smallest class of morphisms containing
\[
S \cup \{\hat \partial^0_i \mid i \in I\} \cup \{\hat \partial^1_i \mid i \in I\},
\]
and closed under the operation $f \mapsto \hat \partial_f$.
\end{Def}

\begin{thm}[{\cite[Thm.~2.2.5, Lem.~2.2.20]{olschok:thesis}}] \label{olschok}
Let $\calK$ and $I$ be as above.
Suppose we are given a cartesian functorial cylinder $C$ on $\calK$ and a set $S$ of monomorphisms in $\calK$.
Then there exists a model structure on $\calK$ uniquely characterized by the following properties:
\begin{itemize}
	\item The cofibrations are the monomorphisms.
	\item The set $\Lambda(\calK,I,C,S)$ is a pseudo-generating set of trivial cofibrations.
\end{itemize}
This model structure is combinatorial and left proper.
\end{thm}

\begin{prop}\label{left-Quillen-bifunctor}
	Let $\calK$ and $I$ be as above.
	Suppose $\calK$ admits a model structure whose cofibrations are the monomorphisms, and a pseudo-generating set $\Lambda$ of trivial cofibrations.
	Suppose further that $\calK$ is equipped with a tensor product $\odot \colon \calK \times \calK \to \calK$ that forms part of a biclosed monoidal structure.
	Then these data form a monoidal model structure if and only if:
	\begin{itemize}
		\item $f \hat \odot g$ is a cofibration whenever $f,g \in I$;
		\item $f \hat \odot g$ is a trivial cofibration whenever $f \in \Lambda$ and $g \in I$; and
		\item $f \hat \odot g$ is a trivial cofibration whenever $f \in I$ and $g \in \Lambda$. 
	\end{itemize}
\end{prop}
\begin{proof}
	This is an instance of \cite[Prop.~A.4]{maehara:gray}.
	See also \cite[App.~B]{Henry:weak}.
\end{proof}

\subsection{Cubical sets}\label{subsec:cubical}
We will define cubical sets as presheaves on the \emph{box category}, denoted $\Cube$.
The category $\Cube$ is the (non-full) subcategory of the category of posets whose objects are the posets of the form $[1]^n := \{0 \leq 1\}^n$ and whose maps are generated by the cubical operators:
\begin{itemize}
	\item \emph{faces} $\partial^n_{i,\varepsilon} \colon [1]^{n-1} \to [1]^n$ for $i = 1, \ldots , n$ and $\varepsilon = 0, 1$ given by:
	\[ \partial^n_{i,\varepsilon} (x_1, x_2, \ldots, x_{n-1}) = (x_1, x_2, \ldots, x_{i-1}, \varepsilon, x_i, \ldots, x_{n-1})\text{;}  \]
	\item \emph{degeneracies} $\sigma^n_i \colon [1]^n \to [1]^{n-1}$ for $i = 1, 2, \ldots, n$ given by:
	\[ \sigma^n_i ( x_1, x_2, \ldots, x_n) = (x_1, x_2, \ldots, x_{i-1}, x_{i+1}, \ldots, x_n)\text{;}  \]
	\item \emph{max-connections} $\gamma^{n}_{i,0} \colon [1]^n \to [1]^{n-1}$ for $i = 1, 2, \ldots, n-1$ given by:
	\[ \gamma^n_{i,0} (x_1, x_2, \ldots, x_n) = (x_1, x_2, \ldots, x_{i-1}, \max\{ x_i , x_{i+1}\}, x_{i+2}, \ldots, x_n) \text{;}
	 \]
	\item \emph{min-connections} $\gamma^{n}_{i,1} \colon [1]^n \to [1]^{n-1}$ for $i = 1, 2, \ldots, n-1$ given by:
	\[ \gamma^n_{i,1} (x_1, x_2, \ldots, x_n) = (x_1, x_2, \ldots, x_{i-1}, \min\{ x_i , x_{i+1}\}, x_{i+2}, \ldots, x_n) \text{;}
	 \]
\end{itemize}

We will omit the superscript $n$ when no confusion is possible.

A straightforward computation shows that cubical operators satisfy the following \emph{cubical identities}.
These maps obey the following \emph{cubical identities}:

\begin{multicols}{2}
$\partial_{j, \varepsilon'} \partial_{i, \varepsilon} = \partial_{i+1, \varepsilon} \partial_{j, \varepsilon'}$ for $j \leq i$;

$\sigma_i \sigma_j = \sigma_j \sigma_{i+1} \quad \text{for } j \leq i$;

$\sigma_j \partial_{i, \varepsilon} = \left\{ \begin{array}{ll}
\partial_{i-1, \varepsilon} \sigma_j   & \text{for } j < i \text{;} \\
\id                                                       & \text{for } j = i \text{;} \\
\partial_{i, \varepsilon} \sigma_{j-1} & \text{for } j > i \text{;}
\end{array}\right.$

$\gamma_{j,\varepsilon'} \gamma_{i,\varepsilon} = \left\{ \begin{array}{ll} \gamma_{i,\varepsilon} \gamma_{j+1,\varepsilon'} & \text{for } j > i \text{;} \\
\gamma_{i,\varepsilon}\gamma_{i+1,\varepsilon} & \text{for } j = i, \varepsilon' = \varepsilon \text{;}\\
\end{array}\right.$

$\gamma_{j,\varepsilon'} \partial_{i, \varepsilon} =  \left\{ \begin{array}{ll}
\partial_{i-1, \varepsilon} \gamma_{j,\varepsilon'}   & \text{for } j < i-1 \text{;} \\
\id                                                         & \text{for } j \in \{i-1,\, i\}, \, \varepsilon = \varepsilon' \text{;} \\
\partial_{i, \varepsilon} \sigma_i         & \text{for } j \in \{i-1,\, i\}, \, \varepsilon = 1-\varepsilon' \text{;} \\
\partial_{i, \varepsilon} \gamma_{j-1,\varepsilon'} & \text{for } j > i \text{;} 
\end{array}\right.$

$\sigma_j \gamma_{i,\varepsilon} =  \left\{ \begin{array}{ll}
\gamma_{i-1,\varepsilon} \sigma_j  & \text{for } j < i \text{;} \\
\sigma_i \sigma_i           & \text{for } j = i \text{;} \\
\gamma_{i,\varepsilon} \sigma_{j+1} & \text{for } j > i \text{.} 
\end{array}\right.$
\end{multicols}

Let us point out that this is only one of the many choices of the box category that appear in the literature.
References such as \cite{maltsiniotis:connections-strict-test-cat,kapulkin-lindsey-wong} consider a box category that is spanned by faces, degeneracies, and one of the connections, specifically the max-connection (although dual arguments can be used to work with min-connections as well).
In \cite{CisinskiAsterisque,jardine:categorical-homotopy-theory}, a subcategory of our $\Box$ is considered that is generated by the face and degeneracy maps, but no connections; and in \cite{steiner:cubical-nerve}, an even smaller subcategory is considered, as it is spanned by the face maps alone.
On the other extreme, \cite{kapulkin-voevodsky} works with $\Box$ as the full subcategory of $\Cat$.

Our choice is intentional.
Since our (marked) cubical sets will be used to model $(\infty, n)$-categories, all of our cubes need to have an orientation, and hence the symmetry and diagonal maps appearing in the choices strictly larger than ours are undesirable.
On the other hand, the box category with at least one connection is known to have better categorical properties than the ones without connections, cf.~\cite{tonks:cubical-groups,maltsiniotis:connections-strict-test-cat}.
Finally, allowing for at least one connection, we choose to work with both to allow for a convenient description of the opposite $(\infty, n)$-category.

Given our choice of the box category, we have the following \emph{normal form} of morphisms in $\Cube$.

\begin{thm}[{\cite[Thm.~5.1]{grandis-mauri}}] \label{normal-form}
  Every map in the category $\Box$ can be factored uniquely as a composite
  \[ (\partial_{c_1, \varepsilon'_1} \ldots \partial_{c_r, \varepsilon'_r})
     (\gamma_{b_1,\varepsilon_1} \ldots \gamma_{b_q,\varepsilon_q})
     (\sigma_{a_1} \ldots \sigma_{a_p})\text{,} \]
  where $1 \leq a_1 < \ldots < a_p$, $1 \leq b_1 \leq \ldots \leq b_q$, $b_i < b_{i+1}$ if $\varepsilon_{i} = \varepsilon_{i+1}$, and $c_1 > \ldots > c_r \geq 1$.   \qed
\end{thm}

With this, one can describe $\Cube$ as the category generated by the cubical operators, subject to the cubical identities.

\begin{rmk}
	In particular, any composite face map can be written uniquely as $	\alpha = \face{k_1}{\epsilon_1}\dots\face{k_t}{\epsilon_t}$	with $k_1 > \dots > k_t$.
	Geometrically, such $\alpha$ is the intersection of all $\face{k_s}{\epsilon_s}$'s.
\end{rmk}

This theorem allows us to prove the following key property of $\Cube$:

\begin{thm}\label{Cube-EZ-Reedy}
	The category $\Cube$ is an EZ Reedy category with the structure defined as:
	\begin{itemize}
	  \item $\deg [1]^n = n$;
	  \item $\Cube_-$ is generated under composition by degeneracies and (both kinds of) connections;
	  \item $\Cube_+$ is generated under composition by face maps.
	\end{itemize}
\end{thm}

The key difficulty in proving this theorem lies in showing that each map in $\Cube_-$ is determined by its sections.
This is done by induction on the length of the decomposition of such a map given in \cref{normal-form}.
Before proceeding with the proof, we state two lemmas.
The first of these contains the base case of induction, whereas the second contains the technical heart of the proof---a case analysis allowing us to complete the inductive step.

\begin{lem} \label{sections-of-generating-degeneracies} \leavevmode
  \begin{enumerate}
    \item The sections of $\sigma_i$ are $\partial_{i,0}$ and $\partial_{i,1}$.
    \item The sections of $\gamma_{i,0}$ are $\partial_{i,0}$ and $\partial_{i+1,0}$.
    \item The sections of $\gamma_{i,1}$ are $\partial_{i,1}$ and $\partial_{i+1,1}$. \qed
  \end{enumerate}
\end{lem}

\begin{lem} \label{Cube-EZ-Reedy-technical}
  Given two distinct maps $p, p' \colon [1]^n \to [1]^{n-k}$ in $\Cube_-$, there is a face map $\partial_{i, \varepsilon} \colon [1]^{n-k} \to [1]^{n-(k-1)}$ such that $p \partial_{i, \varepsilon} \neq p' \partial_{i, \varepsilon}$ and at least one of $p \partial_{i, \varepsilon}$ and $p' \partial_{i, \varepsilon}$ is in $\Cube_-$.
\end{lem}

\begin{proof}
  We proceed by induction with respect to $k$ with the base case of $k=1$ handled by \cref{sections-of-generating-degeneracies}.

  For the inductive step, we may use \cref{normal-form} to write:
  \begin{align*}
    p  & = \gamma_{i_1, \varepsilon_1} \ldots \gamma_{i_l, \varepsilon_l}
        \sigma_{j_1} \ldots \sigma_{j_m} \\
    p' & = \gamma_{i_1', \varepsilon_1'} \ldots \gamma_{i'_{l'}, \varepsilon'_{l'}}
        \sigma_{j_1'} \ldots \sigma_{j'_{m'}}
  \end{align*}
  and without loss of generality we may assume that $m \leq m'$.
  
  We first suppose that there is an index $j_i$ that does not appear in the set $j'_1, \ldots, j'_{m'}$, i.e., there is a degeneracy in the decomposition of $p$ that is not present in the decomposition of $p'$.
  Then we may note that the normal form of $p \partial_{j_i, 0}$ is obtained by removing $\sigma_{j_i}$ from the normal form of $p$, and hence the resulting map is in $\Cube_-$.
  On the other hand, the normal form $p' \partial_{j_i, 0}$ will contain more degeneracy maps than that of $p \partial_{j_i, 0}$, since $\partial_{j_i, 0}$ will not cancel with any of the degeneracy maps present in the normal form of $p'$ and we assumed $m \leq m'$.
  
  If there is no such $j_i$, then the string $\sigma_{j_1} \ldots \sigma_{j_m}$ is a substring of $\sigma_{j_1'} \ldots \sigma_{j'_{m'}}$.
  By precomposing with different face maps, we may assume that $m = 0$.
  We proceed by case analysis, addressing $m' \geq 2$, $m' = 1$, and $m' = 0$ in order.
  
  For $m' \geq 2$, we can write:
    \begin{align*}
    p  & = \gamma_{i_1, \varepsilon_1} \ldots \gamma_{i_\ell, \varepsilon_\ell} \\
    p' & = \gamma_{i_1', \varepsilon_1'} \ldots \gamma_{i'_{\ell'}, \varepsilon'_{\ell'}}
        \sigma_{j_1'} \ldots \sigma_{j'_{m'}}\text{.}
  \end{align*}
  Now observe that $p\partial_{i_\ell, \varepsilon_\ell} = \gamma_{i_1, \varepsilon_1} \ldots \gamma_{i_{\ell-1}, \varepsilon_{\ell-1}}$ is in the normal form (and belongs to $\Cube_-$), but the normal form of $p\partial_{i_l, \varepsilon_l}$ must end with at least one degeneracy.
  
  For $m' = 1$, we can write:
  \begin{align*}
    p  & = \gamma_{i_1, \varepsilon_1} \ldots \gamma_{i_\ell, \varepsilon_\ell} \\
    p' & = \gamma_{i_1', \varepsilon_1'} \ldots \gamma_{i'_{\ell-1}, \varepsilon'_{\ell-1}}
        \sigma_{j'}\text{.}
  \end{align*}
  To treat this case, we will precompose both $p$ and $p'$ with $\partial_{j', \varepsilon}$ to cancel the degeneracy appearing in the normal form of $p'$, yielding a normal form of $p'\partial_{j', \varepsilon}$, which then clearly belongs to $\Cube_-$.
  However, some care needed to choose the correct $\varepsilon$ in order to ensure that the normal form of $p \partial_{j', \varepsilon}$ is different from that of $p' \partial_{j', \varepsilon}$.
  If $j'$ appears in the sequence: $i_1$, \ldots, $i_\ell$, then we pick $\varepsilon = 1 - \varepsilon_{j'}$.
  With this choice, the normal form of $p \partial_{j', \varepsilon}$ will end with a degeneracy, making it distinct from $p \partial_{j', \varepsilon}$.
  If on the other hand $j'$ does not appear on the list of indices: $i_1, \ldots, i_l$, we first need to determine whether when using cubical identities to write $p \partial_{j', \varepsilon}$ in normal form, we will encounter a connection with which our face map will cancel: if not, then we can pick either $\varepsilon$; otherwise, we pick $\varepsilon$ in such a way as to ensure that as a result of commuting the face and the connection, we obtain a degeneracy map.
  
  At this point, it remains to treat the case when $m' = 0$, i.e., the normal forms of $p$ and $p'$ consist solely in connections:
  \begin{align*}
    p  & = \gamma_{i_1, \varepsilon_1} \ldots \gamma_{i_k, \varepsilon_k} \\
    p' & = \gamma_{i_1', \varepsilon_1'} \ldots \gamma_{i'_k, \varepsilon'_k}\text{.}
  \end{align*}
  Note that the two decompositions have the same length, since both $p$ and $p'$ are maps $[1]^n \to [1]^{n-k}$.
  Without loss of generality, we may assume that $i_k \leq i'_k$, and we proceed by case analysis based on $i'_k - i_k$, considering three cases: $i'_k = i_k$, $i'_k = i_k + 1$, and $i'_k \geq i_k +2$.
  
  If $i_k = i'_k$, we precompose $p$ and $p'$ with $\partial_{i_k, \varepsilon_k}$.
  In the case of $\varepsilon_k = \varepsilon'_k$, this reduces us to the inductive hypothesis.
  If however $\varepsilon_k \neq \varepsilon'_k$, then the normal form of $p \partial_{i_k, \varepsilon_k}$ will be $\gamma_{i_1, \varepsilon_1} \ldots \gamma_{i_{k-1}, \varepsilon_{k-1}}$, making it an element of $\Cube_-$, whereas the normal form of $p' \partial_{i_k, \varepsilon_k}$ will end in a degeneracy.
  
  Next, suppose that $i'_k = i_k + 1$.
  Then the cases of $\varepsilon_k = \varepsilon'_k$ and $\varepsilon_k \neq \varepsilon'_k$ need to be treated separately.
  In the former, we precompose with $\partial_{i_k, \varepsilon_k}$.
  Then $p \partial_{i_k, \varepsilon_k} = \gamma_{i_1, \varepsilon_1} \ldots \gamma_{i_{k-1}, \varepsilon_{k-1}}$ is the normal form, making it an element of $\Cube_-$, but the normal form of $p' \partial_{i_k, \varepsilon_k}$ ends with one of: a degeneracy, a connection of first index greater than $i_k$, or a connection $\gamma_{i_k, \varepsilon'_k}$, making it distinct from $p \partial_{i_k, \varepsilon_k}$.
  In the latter case, we see that the normal form of $p \partial_{i_{k+1}, \varepsilon'_k}$ ends with a degeneracy, but the normal form of $p' \partial_{i_{k+1}, \varepsilon'_k}$ ends with $\gamma_{i'_{k-1}, \varepsilon'_{k-1}}$ and this element belongs to $\Cube_-$.
  
  Finally, if $i'_k \geq i_k + 2$, then we precompose both $p$ and $p'$ with $\partial_{i_k, \varepsilon_k}$.
  This gives the normal form of $p\partial_{i_k, \varepsilon_k}$ as $\gamma_{i_1, \varepsilon_1} \ldots \gamma_{i_{k-1}, \varepsilon_{k-1}}$, making it an element of $\Cube_-$.
  But the normal form of $p' \partial_{i_k, \varepsilon_k}$ ends with either a degeneracy or the connection $\gamma_{i'_k, \varepsilon'_k}$, making it distinct.
\end{proof}

\begin{proof}[Proof of \cref{Cube-EZ-Reedy}]
  The Reedy part follows immediately from \cref{normal-form}.
  Every morphism in $\Cube_-$ is a split epimorphism by \cref{sections-of-generating-degeneracies}.
  
  It remains to show that maps in $\Cube_-$ are determined by their sections.
  To do this, we pick two such maps $p, p' \colon [1]^n \to [1]^{n-k}$ and proceed by induction with respect to $k$.
  The base case of $k = 1$ is handled by \cref{sections-of-generating-degeneracies}, whereas the inductive step is handled by \cref{Cube-EZ-Reedy-technical}.
%
\end{proof}

The category $\Cube$ carries a canonical strict monoidal product $\otimes$ given by $[1]^m \otimes [1]^n := [1]^{m+n}$ with unit given by $[1]^0$.
Note that this product is not cartesian since, for instance, there is no `diagonal' map $[1]^1 \to [1]^2$ in $\Cube$.
This monoidal structure leads to another characterization of our box category, due to Grandis and Mauri \cite[\S5]{grandis-mauri}, as a certain kind of a free monoidal category.

A \emph{cubical monoid} in a monoidal category $(\calC, \otimes, I)$ is an object $X$ equipped with maps:
\[ \partial_0, \partial_1 \colon I \to X\text{,} \quad \sigma \colon X \to I\text{,} \quad \gamma_0, \gamma_1 \colon X \otimes X \to X\text{,}\]
subject to the axioms:
\begin{multicols}{2}
	$\sigma \partial_\varepsilon = \id$ for $\varepsilon=0,1$;
	
	$\sigma \gamma_\varepsilon = \sigma (\sigma \otimes \id_X) = \sigma (\id_X \otimes \sigma)$ for $\varepsilon = 0, 1$;
	
	$\gamma_\varepsilon (\gamma_\varepsilon \otimes \id_X) = \gamma_\varepsilon (\id_X \otimes \gamma_\varepsilon)$  for $\varepsilon = 0, 1$;

	$\gamma_\varepsilon (\partial_\varepsilon \otimes \id_X) = \id_X = \gamma_\varepsilon (\id_X \otimes \partial_\varepsilon)$ for $\varepsilon = 0, 1$;
	
	$\gamma_\varepsilon (\partial_\delta \otimes \id_X) = \partial_\delta \sigma = \gamma_\varepsilon (\id_X \otimes \partial_\delta)$ for $\delta \neq \varepsilon$.
		\end{multicols}

\begin{thm}[{\cite[Thm.~5.2.(d)]{grandis-mauri}}] \label{GrandisMauriMonoid}
	The box category $\Cube$ is the free strict monoidal category equipped with a cubical monoid. \qed
\end{thm}

Having established basic properties of the box category, we can now define cubical sets and fundamental constructions on them.

\begin{Def}
  A \emph{cubical set} is a presheaf $X \colon \Cube^\op \to \Set$.
  A \emph{cubical map} is a natural transformation of such presheaves.
  The category of cubical sets and cubical maps will be denoted $\cSet$.
\end{Def}

We write $\Cube^n$ for the cubical set represented by $[1]^n$ and call it the (\emph{generic}) $n$-\emph{cube}.
The \emph{boundary} of the $n$-cube, denoted $\partial \Cube^n \to \Cube^n$, is the maximal proper subobject of the representable $\Cube^n$, i.e., the union of all of its faces.
The subobject of $\Cube^n$ given by the union of all faces except the $(i,\varepsilon)$-th one is called the $(i, \varepsilon)$-\emph{open box} and denoted $\obox{n}{i}{\varepsilon} \to \Cube^n$.

\begin{prop}
  The monomorphisms of $\cSet$ are the cellular closure of the set
  \[ \{ \partial \Cube^n \hookrightarrow \Cube^n \ | \ n \geq 0 \}\text{.} \]
\end{prop}

\begin{proof}
  This follows from \cref{normal-form}.
\end{proof}

The monoidal product $\otimes$ can be extended via Day convolution from $\Cube$ to $\cSet$, making $(\cSet, \otimes, \Cube^0)$ a biclosed monoidal category.
We refer to this monoidal product as the \emph{geometric product} of cubical sets.

We adopt the convention of writing the action of cubical operators on the right, e.g., the $(1, 0)$-face of an $n$-cube $x \colon \Cube^n \to X$ will be denoted $x \partial_{1, 0}$.

\begin{prop}\label{geometric-product-description}
  The geometric product $X \otimes Y$ of cubical sets $X$ and $Y$ admits the following description.
   \begin{itemize}
 	\item For $n \geq 0$, the $n$-cubes in $X \otimes Y$ are the formal products $x \otimes y$ of pairs $x \in X_k$ and $y \in Y_\ell$ such that $k+\ell = n$, subject to the identification $(x\sigma_{k+1})\otimes y = x\otimes(y\sigma_{1})$.
 	\item For $x \in X_k$ and $y \in Y_\ell$, the faces, degeneracies, and connections of the $(k+\ell)$-cube $x \otimes y$ are computed as follows:
 	\begin{itemize}
 		\item $(x\otimes y)\partial_{i,\epsilon} = 
 		\begin{cases} (x\partial_{i,\varepsilon})\otimes y & 1 \leq i \leq k \\ 
 		x\otimes( y\partial_{i-k,\epsilon})  & k + 1 \leq i \leq k+\ell
 		\end{cases}$
 		\item $(x\otimes y)\sigma_{i} = 
 		\begin{cases} (x\sigma_{i})\otimes y     & 1 \leq i \leq k + 1 \\
 		x\otimes (y\sigma_{i-k}) & k + 1 \leq i \leq k+\ell+ 1 
 		\end{cases}$  
 		\item $(x\otimes y)\gamma_{i,\varepsilon} = 
 		\begin{cases} (x\gamma_{i, \varepsilon})\otimes y     & 1 \leq i \leq k \\ 
 		x\otimes (y\gamma_{i-k,\varepsilon}) & k + 1 \leq i \leq k+\ell
 		\end{cases}$
 	\end{itemize}
 \end{itemize}
\end{prop}

\begin{proof}
  This is proven in \cite[Prop.~1.20]{doherty-kapulkin-lindsey-sattler} in the case of cubical sets with one kind of connections.
  The proof given there works almost verbatim in our case.
\end{proof}

Given cubes $x \in X_k$ and $y \in Y_\ell$, we may regard them as cubical maps $x \colon \Cube^k \to X$ and $y \colon \Cube^\ell \to Y$.
Then applying the geometric product to these maps yields a map $x \otimes y \colon \Cube^{k+\ell} \to X \otimes Y$ which corresponds precisely to the $(k+\ell)$-cube with the same name.
Moreover, every $n$-cube of $X \otimes Y$ arises via this construction for some, perhaps non-unique, pair of cubes $(x \colon \Cube^k \to X, y \colon \Cube^\ell \to Y)$ for $k +\ell =n$.

Since the identification in \cref{geometric-product-description} only concerns degenerate cubes, we obtain the following corollary.

\begin{cor} \label{non-deg-cubes-in-Gray}
  A pair of non-degenerate cubes $x \in X_k$, $y \in Y_\ell$ yields a non-degenerate $(k+\ell)$-cube $x \otimes y$ in $X \otimes Y$.
Conversely, every non-degenerate cube in $X \otimes Y$ arises this way from a unique pair of non-degenerate cubes. \qed
\end{cor}

\begin{rmk}
	In particular, when $X = \cube m$ and $Y = \cube n$ are representable this pairing is given by the formula
	\[
	(\face{k_1}{\epsilon_1}\dots\face{k_t}{\epsilon_t}) \otimes (\face{\ell_1}{\eta_1} \dots \face{\ell_s}{\eta_s}) = \face{m+\ell_1}{\eta_1}\dots\face{m+\ell_s}{\eta_s}\face{k_1}{\epsilon_1}\dots\face{k_t}{\epsilon_t}
	\]
	where all strings of $\partial$'s are in the normal form specified by \cref{normal-form}.
	The factors are permuted because the geometric product lists cubes in order (in the sense that $x$ in $x \otimes y$ corresponds to smaller values of $i$) whereas the normal form lists faces in reverse order.
\end{rmk}

\begin{prop}\label{geom-prod-of-bdry}
For natural numbers $k$, $m$, and $n$, and $\varepsilon = 0, 1$, we have natural isomorphisms:
	\[
	\begin{gathered}
	\bigl(\partial \cube m \hookrightarrow \cube m\bigr) \hat \otimes \bigl(\partial \cube n \hookrightarrow \cube n \bigr) \cong \bigl(\partial \cube {m+n} \hookrightarrow \cube {m+n} \bigr)\\
	\bigl(\obox m k \epsilon \hookrightarrow \cube m\bigr) \hat \otimes \bigl(\partial \cube n \hookrightarrow \cube n \bigr) \cong \bigl(\obox {m+n} k \epsilon \hookrightarrow \cube {m+n} \bigr)\\
	\bigl(\partial \cube m \hookrightarrow \cube m\bigr) \hat \otimes \bigl(\obox n k \epsilon \hookrightarrow \cube n \bigr) \cong \bigl(\obox {m+n} {m+k} \epsilon \hookrightarrow \cube {m+n} \bigr)
	\end{gathered}
	\]
\end{prop}
\begin{proof}
	This follows from \cite[Lem.~1.26]{doherty-kapulkin-lindsey-sattler} and the associativity of $\hat\otimes$.
\end{proof}

Using the above proposition and the fact that $\Cube$ is an elegant Reedy category, we obtain:

\begin{cor}\label{geom-prod-of-monos}
	If $f$ and $g$ are monomorphisms in $\cSet$, then $f \hat \otimes g$ is again a monomorphism. \qed
\end{cor}

The category $\Cube$ admits two canonical identity-on-objects automorphisms $(-)^\co, (-)^\coop \colon \Cube \to \Cube$.
The first one takes $\partial^n_{i, \varepsilon}$ to $\partial^n_{n+1-i, \varepsilon}$, $\sigma^n_i$ to $\sigma^n_{n+1-i}$, and $\gamma^n_{i,\varepsilon}$ to $\gamma^n_{(n-1)+1-i, \varepsilon}$.
The second one takes $\partial^n_{i, \varepsilon}$ to $\partial^n_{i, 1- \varepsilon}$, $\sigma^n_i$ to $\sigma^n_{i}$, and $\gamma^n_{i,\varepsilon}$ to $\gamma^n_{i, 1- \varepsilon}$.
(Their names are motivated by the fact that, according to the source/target distinction described in \cref{sec:model} below, $(-)^\co$ reverses the direction of even-dimensional cubes and $(-)^\coop$ reverses the direction of all cubes.)
Precomposition with these automorphisms induces functors also denoted $(-)^\co , (-)^\coop\colon \cSet \to \cSet$.
Moreover, $(-)^\co \circ (-)^\coop = (-)^\coop \circ (-)^\co$, yielding a third automorphism $(-)^\op := (-)^\co \circ (-)^\coop$.

The ``contravariant'' behavior of these automorphisms with respect to the cubical structure can be seen via their interaction with the geometric product.

\begin{prop} \leavevmode
\begin{enumerate}
  \item   The functor $(-)^\co \colon \cSet \to \cSet$ is strong anti-monoidal, i.e., $(X \otimes Y)^\co \cong Y^\co \otimes X^\co$, naturally in $X$ and $Y$.
  \item The functor $(-)^\coop  \colon \cSet \to \cSet$ is strong monoidal i.e., $(X \otimes Y)^\coop \cong X^\coop \otimes Y^\coop $, naturally in $X$ and $Y$.
  \item   The functor $(-)^\op  \colon \cSet \to \cSet$ is strong anti-monoidal, i.e., $(X \otimes Y)^\op \cong Y^\op \otimes X^\op$, naturally in $X$ and $Y$. \qed
\end{enumerate}
\end{prop}

Finally, the composite $\Cube \to \Cat \to \sSet$ given by $\Cube^n \mapsto (\Delta^1)^n$ defines a co-cubical object in the category of simplicial sets. Taking the Yoneda extension, we obtain an adjoint pair
\[ T :  \cSet \rightleftarrows \sSet : U\text{.} \]
We will call $T \colon \cSet \to \sSet$ the \emph{triangulation} functor.

The triangulation functor can also be seen through the lenses of \cref{GrandisMauriMonoid}.
The simplicial faces $\partial_1, \partial_0 \colon [0] \to [1]$ and degeneracy $\sigma_0 \colon [1] \to [0]$ maps, along with $\max, \min \colon [1]^2 \to [1]$ equip $[1]$ with the structure of a cubical monoid.
Since the nerve functor preserves products, this gives a structure of a cubical monoid on $\Delta^1$ in $\sSet$.
The triangulation functor $T \colon \cSet \to \sSet$ arises from this cubical monoid via \cref{GrandisMauriMonoid}.

We conclude this section by recording some properties of the triangulation functor.

\begin{prop}\leavevmode
  \begin{enumerate}
    \item $T$ is strong monoidal.
    \item $T$ preserves monomorphisms.
  \end{enumerate}    
\end{prop}

\begin{proof}
  The first statement follows by the fact that $T$ preserves colimits and $\sSet$ is cartesian closed.
  
  The second statement follows from first, since $T$ takes boundary inclusions, i.e., the elements of the cellular model, to monomorphisms.
\end{proof}



\subsection{Complicial sets}
In this section, we recall marked simplicial sets and model structures for ($n$-trivial) complicial sets from \cite{verity:complicial,verity:weak-complicial-1}.
The reader is referred to those papers for more detail on the subject.
The theory developed in \cref{sec:model} draws great insight from this simplicial precursor.

Just as in the case of cubical sets, when working with simplicial sets, we will write the action of simplicial operators on the right.

\begin{Def}
A \defterm{marked simplicial set} is a simplicial set $X$ equipped with a subset $eX$ of its simplices called the \defterm{marked} simplices such that
\begin{itemize}
	\item no 0-simplex is marked, and
	\item every degenerate simplex is marked.
\end{itemize}
A \defterm{map} of marked simplicial sets $f\colon X \to Y$ is a map of simplicial sets which carries marked simplices to marked simplices. We denote $\mSet$ for the category of marked simplicial sets with maps for morphisms.
\end{Def}

Marked simplicial sets used to be called \emph{stratified simplicial sets} (cf.~e.g., \cite{verity:complicial}), but the name `marked' is more descriptive and has since become more popular.

There is a natural forgetful functor $\mSet \to \sSet$, which has both left and right adjoints.
The left adjoint $X \mapsto X^\flat$ endows a simplicial set $X$ with the \defterm{minimal marking}, marking only the degenerate simplices. The right adjoint $X \mapsto X^\sharp$ endows a simplicial set $X$ with the \defterm{maximal marking}, marking all simplices.

If $X$ is a simplicial set, we will by default consider it as a marked simplicial set \emph{with its minimal marking} $X^\flat$.

\begin{Def}
We say that $X \in \mSet$ is \defterm{$n$-trivial} if every simplex of dimension $\geq n+1$ is marked.
\end{Def}

Given a marked simplicial set $X$, we will write $\core n X$ for its maximal $n$-trivial subset.
In other words, the $k$-simplices of $\core n X$ are precisely those $k$-simplices $x$ in $X$ such that $x \alpha$ is marked in $X$ for any $\alpha \colon [m] \to [k]$ with $m >n$.
This assignment extends to a functor $\core n \colon \mSet \to \mSet$, which admits a left adjoint $\trunc n \colon \mSet \to \mSet$.
Explicitly, $\trunc n$ acts as the identity on the underlying simplicial set and a $k$-simplex is marked in $\trunc n X$ if either $k \leq n$ and $x$ is marked in $X$ or $k \geq n+1$.

\begin{Def}
  A map $X \to Y$ of marked simplicial sets is:
  \begin{itemize}
    \item \emph{regular} if it creates markings, i.e., for an $n$-simplex $x$ of $X$ we have: $x \in eX_n$ if and only if $f(x) \in eY_n$;
    \item \emph{entire} if the induced map between the underlying simplicial sets is invertible.
  \end{itemize}
\end{Def}


We now define several distinguished objects and maps in $\mSet$. These will be essential to the description of various model structures we will be considering.

We denote $\mDelta n = \trunc {n-1} (\sDelta n)$ the $n$-simplex with the non-degenerate $n$-simplex marked and no other non-degenerate simplices marked. We call the canonical map $\sDelta n \to \mDelta n$ the \defterm{$n$-marker}.

For $n \geq 1$ and $0 \le k  \le n$, we denote $\cDelta k n$ the $n$-simplex with the following marking: a non-degenerate simplex is marked iff it contains all of the points $\{k-1,k,k+1\} \cap [n]$ among its vertices.
We call $\cDelta k n$ the \defterm{$k$-complicial $n$-simplex}. We denote $\horn k n \subset \cDelta k n$ the $k$-horn of dimension $n$ (i.e.~the simplicial subset missing the non-degenerate $n$-simplex and the $k$th $(n-1)$-face) endowed with the marking making it a regular subset of $\cDelta k n$. We call $\horn k n$ the \defterm{complicial $k$-horn of dimension $n$}.
We call the inclusion $\horn k n \to \cDelta k n$ the \defterm{$k$-complicial horn inclusion of dimension $n$}. 
We denote $\cDeltapp k n = \trunc {n-2} \cDelta k n$, and we denote $\cDeltap k n = \cDelta k n \amalg_{\horn k n} \trunc {n-2} \horn k n$.
The canonical inclusion $\cDeltap k n \to \cDeltapp k n$ is called the \defterm{elementary $k$-complicial marking extension of dimension $n$}.

Let $\presat$ denote the marked simplicial set obtained from $\Delta^3$ by marking the $1$-simplices $\{0,2\}$, $\{1,3\}$ and all $2$- and $3$-simplices.
By a \emph{saturation map}, we mean a map of the form $\Delta^m \join \presat \to \Delta^m \join (\Delta^3)^\sharp$ for $m \ge -1$ (where $\Delta^{-1}$ is interpreted as $\varnothing$).

There are two standard model structures on marked simplicial sets:

\begin{thm}\label{complicial-model-structure}
  The category $\mSet$ carries two model structures:
  \begin{enumerate}
  	\item the \defterm{complicial model structure} characterized by the following properties:
  	\begin{itemize}
  		\item The cofibrations are the monomorphisms.
  		\item The set of
  		\begin{itemize}
  			\item complicial horn inclusions, and 
  			\item elementary complicial marking extensions
  		\end{itemize}
  		forms a pseudo-generating set of trivial cofibrations.
  	\end{itemize}
  	\item the \defterm{saturated complicial model structure} characterized by the following properties:
  	\begin{itemize}
  		\item The cofibrations are the monomorphisms.
  		\item The set of
  		\begin{itemize}
  			\item complicial horn inclusions,
  			\item elementary complicial marking extensions, and
  			\item saturation maps
  		\end{itemize}
  		forms a pseudo-generating set of trivial cofibrations.
  		\end{itemize}
  \end{enumerate}
Both of these model structures are cartesian.
\end{thm}

\begin{proof}
	This is a combination of \cite[Lem.~72, Thm.~100~\&~Lem.~105]{verity:weak-complicial-1} and \cite[App.~B]{ozornova-rovelli:model-structure}.
\end{proof}

Note that since the terminal object is always fibrant, this includes a characterization of the fibrant objects of the model structure,s which are called (\defterm{saturated}) \defterm{complicial sets}.
\begin{Def}
  A map of marked simplicial sets $X \to Y$ is 
  \begin{itemize}
    \item a \defterm{complicial marking extension} if it is in the cellular closure of the elementary complicial marking extensions;
    \item \defterm{complicial} if it is in the cellular closure of the complicial horn inclusions and the elementary complicial marking extensions.
  \end{itemize}  
\end{Def}

There is also the $n$-trivial version of the (saturated) complicial model structure.

\begin{thm}\label{n-complicial-model-structure}
  The category $\mSet$ carries two model structures:
  \begin{enumerate}
  	\item the \defterm{$n$-trivial complicial model structure} characterized by the following properties:
  	\begin{itemize}
  		\item The cofibrations are the monomorphisms.
  		\item The set of
  		\begin{itemize}
  			\item complicial horn inclusions,
  			\item elementary complicial marking extensions of dimension $\le n+1$, and
  			\item markers of dimension $>n$
  		\end{itemize}
  		forms a pseudo-generating set of trivial cofibrations.
  	\end{itemize}
  	\item the \defterm{saturated $n$-trivial complicial model structure} characterized by the following properties:
  	\begin{itemize}
  		\item The cofibrations are the monomorphisms.
  		\item The set of
  		\begin{itemize}
  			\item complicial horn inclusions,
  			\item elementary complicial marking extensions of dimension $\le n+1$,
  			\item markers of dimension $>n$, and
  			\item saturation maps
  		\end{itemize}
  		forms a pseudo-generating set of trivial cofibrations.
  	\end{itemize}
  \end{enumerate}

\end{thm}

\begin{proof}
	Essentially the same as the proof of \cref{complicial-model-structure}, but combined with \cite[Ex.~104]{verity:weak-complicial-1}.
\end{proof}

In $\mSet$, the pseudo Gray tensor product is modelled by the cartesian product.
We will adopt the following notation from \cite{verity:weak-complicial-1,verity:weak-complicial-1} which emphasizes this view.

\begin{notation}
	The cartesian product on $\mSet$ (and its reflective subcategory $\PreComp$ described below) is denoted by $\pseudo$.
\end{notation}

Thus \cref{complicial-model-structure} in particular says that Verity's model structure is monoidal with respect to the pseudo Gray tensor product.

The following proposition will be useful later.

\begin{prop}\label{pseudo-Gray-of-entire-maps}
	Let $f\colon A \to X$ and $g\colon B \to Y$ be entire maps in $\mSet$.
	Then their Leibniz pseudo Gray tensor $f \hat \pseudo g$ is a complicial marking extension.
\end{prop}
\begin{proof}
	Since the forgetful functor $\mSet \to \sSet$ preserves colimits and products, $f \hat \pseudo g$ is entire.
	We assume for the sake of simplicity that $f \hat \pseudo g$ is an inclusion.
	Let $(x,y)$ be an $n$-simplex that is marked in $X \pseudo Y$ but not in $\dom(f \hat \pseudo g) = (A \pseudo Y) \cup (X \pseudo B)$.
	Equivalently, $x$ is marked in $X$ but not in $A$, and $y$ is marked in $Y$ but not in $B$.
	Then we must have $n \ge 1$, so the $(n+1)$-simplex $z = (x   \sigma_0,y   \sigma_1)$ is well-defined.
	We claim that this simplex $z$ extends as indicated below:
	\[
	\begin{tikzcd}
	\Delta^{n+1}
	\arrow [r, "z"]
	\arrow [d] &
	(A \pseudo Y) \cup (X \pseudo B) \\
	\admdeltap{n+1}{1}
	\arrow [ur, dashed, "\exists", swap] &
	\end{tikzcd}
	\]
	To see that $z$ at least extends to $\admdelta{n+1}{1}$, let $\alpha \colon [m] \to [n+1]$ be a simplicial operator with 
	$0,1,2 \in \im \alpha$.
	Then both $x   (\sigma_0 \circ \alpha)$ and $y   (\sigma_1 \circ \alpha)$ are degenerate, so $z   \alpha$ is marked in $\dom(f \hat \pseudo g)$.
	Since the face $z   \faces 0 = \bigl(x, y   (\faces 0 \circ \sigma_0)\bigr)$ is marked in $X \pseudo B$ and the face $z   \faces 2 = \bigl(x  (\faces 1 \circ \sigma_0\bigr),y)$ is marked in $A \pseudo Y$, we indeed have an extension as indicated.
	Therefore we have a pushout square
	\[
	\begin{tikzcd}
	\coprod \admdeltap{n+1}{1}
	\arrow [r]
	\arrow [d] &
	(A \pseudo Y) \cup (X \pseudo B)
	\arrow [d] \\
	\coprod \admdeltapp{n+1}{1}
	\arrow [r] &
	X \pseudo Y
	\end{tikzcd}
	\]
	where the coproducts are taken over all $n$-simplices $(x,y)$ that are marked in $X \pseudo Y$ but not in $\dom(f \hat \pseudo g)$, and both horizontal maps are induced by the simplices of the form $(x   \sigma_0,y   \sigma_1)$.
	This completes the proof.
\end{proof}

\begin{Def}
	Let $[n] \in \Delta$ and let $0 \leq p,q \leq n$ be such that $p+q = n$.
	Then we write $\frontface p q \colon [p] \to [n]$ for the simplicial operator $i \mapsto i$, and $\backface p q \colon [q] \to [n]$ for the operator $i \mapsto p+i$.
\end{Def}

%

In the following definition, we use slightly different terminology from Verity's original one \cite[Def.~127 \& Def.~128]{verity:complicial}.

\begin{Def}
	Let $X,Y \in \mSet$, let $(x,y) \in X_n \times Y_n$ be a simplex of $X \times Y$, and let $0 \leq i \leq n$. We say that $(x,y)$ is \defterm{$i$-cloven} if \emph{either} $x   \frontface i {n-i}$ is marked in $X$ or $y   \backface i {n-i}$ is marked in $Y$. We say that $(x,y)$ is \defterm{fully cloven} if it is $i$-cloven for all $0 \leq i \leq n$.
	
	The \defterm{Gray tensor product} of $X$ and $Y$, denoted $X \otimes Y$, is defined to be the marked simplicial set with underlying simplicial set $X \times Y$, where a simplex $(x,y) \in X_n \times Y_n$ is marked if and only if it is fully cloven.
\end{Def}

\begin{thm}[{\cite[Lem.~131]{verity:complicial}}]\label{thm:gray}
	The Gray tensor product endows $\mSet$ with a (nonsymmetric) monoidal structure, such that the forgetful functor $(\mSet, \otimes) \to (\sSet, \times)$ is strict monoidal.
\end{thm}

\begin{Def}
A \defterm{pre-complicial set} is a marked simplicial set $X$ with the right lifting property with respect to the complicial marking extensions. These form a reflective subcategory of $\mSet$ which we will denote $\PreComp$. We will denote the localization functor $X \mapsto X^\precomp$.
\end{Def}

\begin{prop}\label{unit}
	The unit of the reflection $X \to X^\precomp$ is a complicial marking extension for any $X \in \mSet$.
\end{prop}
\begin{proof}
	Obtain a complicial marking extension $f \colon X \to Y$ into a pre-complicial set $Y$ by applying the small object argument to the unique map $X \to 1$ with respect to the elementary complicial marking extensions.
	Then any map $X \to Z$ into a pre-complicial set $Z$ factors through $f$.
	Moreover, since $f$ is an epimorphism, such a factorisation is necessarily unique.
	In other words, $f$ has the universal property of the unit $X \to X^\precomp$.
\end{proof}

\begin{thm}\label{precomplicial-model-structure}
	For each of the four model structures in \cref{complicial-model-structure,n-complicial-model-structure}, the category $\PreComp$ carries an analogous model structure characterized by the following conditions:
	\begin{itemize}
		\item The cofibrations are the monomorphisms.
		\item A pseudo-generating set of trivial cofibrations can be obtained by taking that for the corresponding model structure on $\mSet$ (described in \cref{comical-model-structure} or \cref{n-complicial-model-structure}), removing the complicial marking extensions, and then applying the pre-complicial reflection.
	\end{itemize}
The localization $(-)^\precomp$ is a left Quillen equivalence between the complicial model structures (resp. the $n$-trivial complicial model structures).
\end{thm}

\begin{proof}
Fix one of the four model structures on $\mSet$.
Observe that, if we factor a map between pre-complicial sets into a cofibration followed by a fibration (one of which is trivial) with respect to that model structure, then the middle object must also be pre-complicial.
Thus we obtain a restricted model structure on $\PreComp$.

We obtain the characterisation of its cofibrations by observing that the reflective inclusion $\PreComp \incl \mSet$ preserves and reflects monomorphisms.
It is straightforward to check that the pre-complicial reflection preserves pseudo-generating sets, and moreover it inverts all (elementary) complicial marking extensions.

%
%

It follows that the adjunction $\mSet \rightleftarrows \PreComp$ is a Quillen adjunction.
Finally, since the unit is a natural weak equivalence this is in fact a Quillen equivalence.
\end{proof}

\begin{rmk}
	We believe that the pre-complicial reflection does not actually affect the remaining members of the pseudo-generating set.
	However we do not provide a proof as it is not essential.
\end{rmk}

Now we analyse the pre-complicial reflection of the Gray tensor product on $\PreComp$.

\begin{Def}
	We write $\pretensor \colon \PreComp \times \PreComp \to \PreComp$ for the pre-complicial Gray tensor product functor $(X,Y) \mapsto (X\otimes Y)^\precomp$.
\end{Def}

\begin{thm}\label{PreComp-monoidal-model}
	The bifunctor $\pretensor$ is part of a biclosed monoidal structure on $\PreComp$.
Moreover each model structure on $\PreComp$ described in \cref{precomplicial-model-structure} is monoidal with respect $\pretensor$.
\end{thm}
\begin{proof}
The first assertion is \cite[Thm.~148]{verity:complicial}.
It is straightforward to check that the Leibniz Gray tensor product preserves monomorphisms.
Since the complicial model structure on $\mSet$ is monoidal with respect to the Gray tensor product \cite[Thm.~109]{verity:weak-complicial-1} (although it is not biclosed on $\mSet$) and the unit of the pre-complicial reflection is a levelwise complicial marking extension, it follows that the complicial model structure on $\PreComp$ is monoidal.
The $n$-trivial and saturated versions follow from \cite[\textsection 2]{ozornova-rovelli:model-structure}.
\end{proof}



\section{Marked cubical sets and Gray tensor products} \label{sec:cubical}

In this section, we introduce marked cubical sets and define their Gray tensor product.

\subsection{Marked cubical sets}
In order to define marked cubical sets, we need to introduce certain enlargement $\Cube^+$ of the box category.
The objects of $\Cube^+$ consist of: $[1]^n$ for every $n \geq 0$ and $[1]^n_e$ for every $n \geq 1$.
The morphisms of $\Cube^+$ are generated by the maps
\begin{itemize}
	\item[-] $\partial^n_{i, \varepsilon} \colon [1]^{n-1} \to [1]^n$ for every $n \geq 1$, $i = 1, \ldots, n$, and $\varepsilon = 0, 1$,
	\item[-] $\sigma^n_i \colon [1]^n \to [1]^{n-1}$ for $n \geq 1$ and $i = 1, \ldots, n$,
	\item[-] $\gamma^n_{i,\varepsilon} \colon [1]^n \to [1]^{n-1}$ for $n \geq 2$, $i = 1, \ldots, n-1$, and $\varepsilon = 0, 1$,
	\item[-] $\varphi^n \colon [1]^n \to [1]^n_e$ for $n \geq 1$,
	\item[-] $\zeta^n_i \colon [1]^n_e \to [1]^{n-1}$ for $n \geq 1$ and $i = 1, \ldots, n$,
	\item[-] $\xi^n_{i, \varepsilon} \colon [1]^n_e \to [1]^{n-1}$ for $n \geq 1$, $i = 1, \ldots, n$, and $\varepsilon = 0, 1$,
\end{itemize}
subject to the usual cubical identities and the following additional relations:
\begin{multicols}{2}
  $\zeta_i \varphi = \sigma_i$;

  $\xi_{i, \varepsilon} \varphi = \gamma_{i, \varepsilon}$;

  $\sigma_i \zeta_j = \sigma_j \zeta_{i+1}$ for $j \leq i$;

  $\gamma_{j, \varepsilon} \xi_{i, \delta} = \left\{ \begin{array}{ll}
	\gamma_{i, \delta} \xi_{j+1, \varepsilon} & \text{for } j > i\text{;} \\
	\gamma_{i, \delta} \xi_{i+1, \delta} & \text{for } j=i \text{ and } \delta = \varepsilon\text{;}
	\end{array}\right.$	

 $\sigma_j \xi_{i, \delta} =  \left\{ \begin{array}{ll}
	\gamma_{i-1, \delta} \zeta_j    & \text{for } j < i \text{;} \\
	\sigma_i \zeta_i                        & \text{for } j = i \text{;} \\
	\gamma_{i, \delta} \zeta_{j+1} & \text{for } j > i \text{.} 
	\end{array}\right.$
\end{multicols}

\begin{prop}\label{Cube+isReedy}
  The category $\Cube^+$ is an EZ Reedy category with the following Reedy structure:
  \begin{itemize}
    \item $\deg [1]^0 = 0$; $\deg [1]^n = 2n-1$ for $n \geq 1$; $\deg [1]^n_e = 2n$ for $n \geq 1$;
    \item $\Cube^+_-$ is generated by the maps: $\sigma^n_i$, $\gamma^n_{i, \varepsilon}$, $\zeta^n_i$, and $\xi^n_{i, \varepsilon}$;
    \item $\Cube^+_+$ is generated by the maps: $\partial^n_{i, \varepsilon}$ and $\varphi^n$.
  \end{itemize}
\end{prop}

The proof of this fact follows the one in \cite[App.~C]{ozornova-rovelli:model-structure}.
We begin by noting the following simple lemma:

\begin{lem} \label{Cube+ReedyHelper}\leavevmode
  \begin{enumerate}
    \item\label{Cube+ReedyHelperMinus} The are no non-identity maps in $\Cube^+_-$ whose target is in $\Cube^+ \setminus \Cube$.
    \item\label{Cube+ReedyHelperPlus} The are no non-identity maps in $\Cube^+_+$ whose source is in $\Cube^+ \setminus \Cube$. \qed
  \end{enumerate}
\end{lem}

\begin{proof}[Proof of \cref{Cube+isReedy}]
  We first note that the sections of $\zeta_i$ are $\varphi \partial_{i,0}$ and $\varphi \partial_{i,1}$; the sections of $\xi_{i, 1}$ are $\varphi \partial_{i,0}$ and $\varphi \partial_{i+1,0}$; and the sections of $\xi_{i, 0}$ are $\varphi \partial_{i,1}$ and $\varphi \partial_{i+1,1}$.
  Thus all maps in $\Cube^+_-$ have sections.
  
  Using the techniques of \cite[Thm.~5.1]{grandis-mauri}, we can then extend \cref{normal-form} to write normal forms for maps in $\Cube^+$.
  These are established separately for the four cases:
  \begin{enumerate}
    \item The normal form of a map of the form $[1]^m \to [1]^n$ is given by its normal form in $\Cube$.
    If such a form were to be non-unique, we would need to have $[1]^m \to [1]^k_e \to [1]^n$ with $[1]^k_e \to [1]^n \in \Cube^+_+$, which is impossible by \cref{Cube+ReedyHelper}(\ref{Cube+ReedyHelperPlus}).
    \item The normal form of a map $[1]^m \to [1]^n_e$ is obtained by observing that it is necessarily a composite $[1]^m \to [1]^n \overset\varphi\to [1]^n_e$ and taking the normal form of the first map in $\Cube$.
    Again, \cref{Cube+ReedyHelper}(\ref{Cube+ReedyHelperPlus}) implies uniqueness.
    \item The normal form of a map of the form $[1]^m_e \to [1]^n$ is obtained by factoring it as $[1]^m_e \to [1]^{m-1} \to [1]^n$, where the first map is either $\zeta_i$ or $\xi_{i, \varepsilon}$, and taking the normal form of \cref{normal-form} of $[1]^{m-1} \to [1]^n$ in $\Cube$.
    Note that the choice of $\zeta_i$ or $\xi_{i, \varepsilon}$ as the first map may not be unique, but it can be made so by imposing the additional compatibility requirement with the factorization of \cref{normal-form} --- this is because of the additional relations relating $\zeta_i$'s to $\sigma_i$'s and $\xi_{i, \varepsilon}$'s to $\gamma_{i, \varepsilon}$'s.
    Put differently, we may precompose $[1]^m_e \to [1]^n$ with $\varphi$, use the normal form in $\Cube$, and replace the last element by $\zeta_i$ or $\xi_{i, \varepsilon}$ as appropriate.
    \item $[1]^m_e \to [1]^n_e$.
    In this case, we obtain the normal form is obtained by combining the techniques from the previous two cases, namely factoring
    \[ [1]^m_e \to [1]^{m-1} \to [1]^n \overset\varphi\to [1]^n_e\text{,} \]
    where again the first map is one of $\zeta_i$ or $\xi_{i, \varepsilon}$ and the composite $[1]^{m-1} \to [1]^k \to [1]^n$ is obtained in $\Cube$.
  \end{enumerate}
  Having established the normal forms, we then proceed in a manner analogous to the proof of \cref{Cube-EZ-Reedy}.   
\end{proof}

\begin{Def}
	A \emph{structurally marked cubical set} is a presheaf $X \colon (\Cube^+)^\op \to \Set$.
	A \emph{map of structurally marked cubical sets} is a natural transformation of such presheaves.
\end{Def}

Given a structurally marked cubical set $X$, we will write $X_n$ for $X([1]^n)$ and $eX_n$ for $X([1]^n_e)$.
Just as in the case of cubical sets, we adopt the convention of writing cubical operators on the right, e.g., for $x \in e X_1$, we write $x \varphi$ for the resulting element of $X_1$.

\begin{Def}
A \emph{marked cubical set} is a structurally marked cubical set $X \colon (\Cube^+)^\op \to \Set$ for which the map $X\varphi \colon eX_n \to X_n$ is a monomorphism for all $n \geq 1$.
We write $\mcSet$ for the full subcategory of $\Set^{(\Cube^+)^\op}$ spanned by the marked cubical sets.
\end{Def}

We think of a marked cubical set $X$ as a cubical set in which certain $n$-cubes have been designated as \emph{equivalences}, i.e., those in $eX_n \subseteq X_n$.
The maps $\zeta_i$ and $\xi_{i, \varepsilon}$ ensure that every degenerate cube is marked.

We may apply the same intuition to structurally marked cubical sets.
However, failure of $X \varphi$'s to be monomorphisms (in an arbitrary structurally marked cubical set $X$) means that being an equivalence is not a property of an $n$-cube of $X$, but a structure on it, as there can be multiple markings on a single cube.

Every (structurally) marked cubical set has an underlying cubical set, which defines a functor $\upsilon \colon \mcSet \to \cSet$.
Given a cubical set $X$, we can form a marked cubical set in two ways:
\begin{itemize}
  \item the  \emph{minimal marking} functor takes a cubical set $X$ to a marked cubical set $X^\flat$, where only degenerate $n$-cubes are marked;
  \item the \emph{maximal marking functor}, assigns to $X$ the marked cubical set $X^\sharp$ in which all cubes marked (i.e., all $X \varphi^n$'s are identities).
\end{itemize}
This gives two functors $(-)^\flat, (-)^\sharp \colon \cSet \to \mcSet$.
A striaghtforward verification shows:

\begin{prop}
  We have the following string of adjoint functors $(-)^\flat \dashv \upsilon \dashv (-)^\sharp$. \qed
\end{prop}

\begin{rmk}[Limits and colimits of marked cubical sets]\label{colim-in-mcSet}
   The proposition above gives a recipe for computing limits and colimits of diagrams $F \colon \scrI \to \mcSet$.
   In both cases, we first compute the underlying cubical set by taking the (co)limit of $\upsilon F$ in $\cSet$, and then equipping it with the minimal marking making the colimit inclusions maps of marked cubical sets, or the maximal marking making the limit projections maps of marked cubical sets.
   It follows, e.g., that a cube in a colimit is marked if and only if it is in the image of a marked cube under one of the colimit inclusions.
\end{rmk}

Furthermore, the canonical embedding $\mcSet \hookrightarrow \Set^{({\Cube^+})^\op}$ of marked cubical sets into structurally marked cubical sets admits a left adjoint, denoted $\Im \colon \Set^{({\Cube^+})^\op} \to \mcSet$.
Explicitly, $\Im X$ is obtained by factoring all $\varphi_n$'s via their image $eX_n \to (eX_n)\varphi_n \to X_n$ and taking the resulting object as a new set of marked $n$-cubes.
We may summarize it with the following statement:

\begin{prop}
  Marked cubical sets form a reflective subcategory of the structurally marked cubical sets with the reflector given by $\Im \colon \Set^{({\Cube^+})^\op} \to \mcSet$. \qed
\end{prop}

\begin{cor}
  The category $\mcSet$ of marked cubical sets is locally presentable. \qed
\end{cor}

\begin{Def}
  A map $f \colon X \to Y$ of marked cubical sets is:
  \begin{itemize}
    \item \emph{regular} if it creates markings, i.e., for an $n$-cube $x$ of $X$ we have: $x \in eX_n$ if and only if $f(x) \in eY_n$;
    \item \emph{entire} if the induced map between the underlying cubical sets is invertible.
  \end{itemize}
\end{Def}

\begin{Def}
We say that $X \in \mcSet$ is \defterm{$n$-trivial} if every cube of dimension $\geq n+1$ is marked.
\end{Def}

Given a marked cubical set $X$, we will write $\core n X $ for its maximal $n$-trivial subset.
In other words, the $k$-cubes of $\core n X$ are precisely those $k$-cubes $x$ such that $x\alpha$ is marked for all $\alpha \colon [1]^m \to [1]^k$ with $m > n$.
This assignment extends to a functor ${\core n} \colon \mcSet \to \mcSet$, which admits a left adjoint $\trunc n \colon \mcSet \to \mcSet$.
Explicitly, $\trunc n$ acts as the identity on the underlying cubical set and a $k$-cube is marked in $\trunc n X$ if either $k \leq n$ and $x$ is marked in $X$ or $k \geq n+1$.

When a cubical set is considered as a marked cubical set, it will \emph{almost} always be considered with its minimal marking.
The only exception is the open boxes; see \cref{sec:model}.
We denote $\cube n = (\cube n)^\flat$ the $n$-cube regarded as a marked cubical set and likewise $\cbdy n = (\cbdy n)^\flat$.
Just as in the case of cubical sets, we call the inclusion map $\cbdy n \to \cube n$ the \defterm{boundary inclusion}. 
We denote $\mcube n = \trunc {n-1} (\cube n)$ the $n$-cube with the non-degenerate $n$-cube marked and no other non-degenerate cubes marked.
We call the canonical map $\cube n \to \mcube n$ the \defterm{$n$-marker}.

\begin{prop}
  The monomorphisms of $\mcSet$ (and $\Set^{(\Cube^+)^\op}$) are the cellular closure of the set
  \[ \{ \partial \Cube^n \hookrightarrow \Cube^n \ | \ n \geq 0 \}
  \cup \{ \Cube^n \hookrightarrow \mcube n \ | \ n \geq 1 \}\text{.} \]
\end{prop}

The functors $(-)^\co, (-)^\coop, (-)^\op \colon \cSet \to \cSet$ generalize to the marked setting in the straightforward manner.
For $(-)^\co$ we send $\varphi^n$ to itself, $\zeta^n_i$ to $\zeta^n_{n+1-i}$, and $\xi^{n}_{i, \varepsilon}$ to $\xi^n_{n+1-i , \varepsilon}$.
For $(-)^\coop$, we send $\varphi^n$ and $\zeta^n_i$ to themselves and $\xi^{n}_{i, \varepsilon}$ to $\xi^n_{i , 1-\varepsilon}$.
These then induce functors by precomposition $(-)^\co, (-)^\coop, (-)^\op \colon \mcSet \to \mcSet$.


\subsection{Gray tensor products}

The following definition makes use of \cref{non-deg-cubes-in-Gray}.

\begin{Def}\label{lax-def}
	The (\emph{lax}) \emph{Gray tensor product} $X \otimes Y$ of two marked cubical sets $X,Y \in \mcSet$ is the geometric product $\upsilon X \otimes \upsilon Y$ wherein a non-degenerate cube $x \otimes y$ is marked if and only if either $x$ is marked in $X$ or $y$ is marked in $Y$.
	This extends to a functor $\otimes \colon \mcSet \times \mcSet \to \mcSet$ in the obvious way.
\end{Def}
\begin{Def}\label{pseudo-def}
	The \emph{pseudo Gray tensor product} $X \pseudo Y$ is the geometric product $\upsilon X \otimes \upsilon Y$ wherein a non-degenerate cube $x \otimes y$ is unmarked if and only if:
	\begin{itemize}
		\item $x$ is a 0-cube and $y$ is unmarked in $Y$; or
		\item $x$ is unmarked in $X$ and $y$ is a 0-cube.
	\end{itemize}
	This extends to a functor $\pseudo \colon \mcSet \times \mcSet \to \mcSet$ in the obvious way.
\end{Def}
\begin{rmk}\label{comparing-Gray-tensors}
	Since no 0-cubes are marked, one can easily check that $X \pseudo Y$ may be obtained from $X \otimes Y$ by marking those non-degenerate $x \otimes y$ such that $x \in X_m$, $y \in Y_n$ with $m,n \ge 1$.
	Thus the identity at $\upsilon X \otimes \upsilon Y$ lifts to an entire map $\mu_{X,Y} \colon X \otimes Y \to X \pseudo Y$.
	This map is clearly natural in $X$ and $Y$, and moreover $\mu_{X,Y}$ is invertible if either $X$ or $Y$ is 0-trivial.
\end{rmk}

\begin{rmk}
	The Gray tensor products $\otimes$ and $\pseudo$ share many properties and often a statement or a proof applies equally well to both tensor products.
	In such situations, we write $\odot$ to mean either.
	Of course the interpretation of $\odot$ is to be kept consistent within each statement and its proof.
\end{rmk}

\begin{thm} \label{GrayFundamentalProp}\leavevmode
	\begin{enumerate}
		\item The Gray tensor product $\odot$ forms part of a biclosed monoidal structure on $\mcSet$ such that the forgetful functor $\upsilon \colon (\mcSet,\odot) \to (\cSet,\otimes)$ is strict monoidal.
		\item The entire inclusions $\mu_{X,Y} \colon X \otimes Y \to X \pseudo Y$ together with $\mu_0 = \id_{\cube 0}$ equip the identity functor with a monoidal structure $(\id_{\mcSet},\mu) \colon (\mcSet, \pseudo) \to (\mcSet,\otimes)$.
		\item The minimal marking functor $(-)^\flat \colon (\cSet,\otimes) \to (\mcSet, \otimes)$ is strict monoidal.
		\item The maximal marking functor $(-)^\sharp \colon (\cSet,\otimes) \to (\mcSet,\odot)$ is strict monoidal.
	\end{enumerate}
\end{thm}

\begin{proof}
	We first check the associativity of the tensor product.
	Suppose we are given non-degenerate cubes $x \in X_m$, $y \in Y_n$, $z \in Z_k$ in $X,Y,Z \in \mcSet$.
	Then the $(m+n+k)$-cube $(x \otimes y) \otimes z$ in $(X \odot Y) \odot Z$ is unmarked if and only if:
	\begin{itemize}
		\item [($\otimes$)] none of $x,y,z$ is marked; or
		\item [($\pseudo$)] (at least) two of $x,y,z$ are 0-cubes and the last is unmarked.
	\end{itemize}
	One can give a similar characterization of when $x \otimes (y \otimes z)$ is unmarked, and it follows that the associativity isomorphism $(\upsilon X \otimes \upsilon Y) \otimes \upsilon Z \cong \upsilon X \otimes (\upsilon Y \otimes \upsilon Z)$ in $\cSet$ lifts to an isomorphism $(X \odot Y) \odot Z \cong X \odot (Y \odot Z)$ in $\mcSet$.
	The unit isomorphisms can be lifted similarly, and moreover these lifted isomorphisms are suitably natural and coherent.
	Thus we indeed obtain a monoidal structure on $\mcSet$ such that $\upsilon$ is strict monoidal.
	The clauses (2-4) are then obvious from the definitions of the tensor products.
	
	It remains to show that this monoidal structure is biclosed.
	Equivalently, we must show that $\odot$ preserves colimits in each variable separately.
	So let $F \colon \scrI \to \mcSet$ and $X \in \mcSet$.
	Since the geometric product is cocontinuous in each variable and $\upsilon$ is cocontinuous and strict monoidal, the canonical comparison map
	\[
	\colim(X \odot F) \to X \odot \colim F
	\]
	is $\upsilon$-invertible.
	Moreover one can check using \cref{colim-in-mcSet} that a non-degenerate cube in either side is marked if and only if it is the image of some marked cube under the canonical map from $X \odot Fi$ for some $i \in \scrI$.
	It follows that this comparison map is invertible.
	Dually, $(-)\odot X$ preserves colimits.
	Since $\mcSet$ is locally finitely presentable, the existence of the desired biclosed structure now follows.
\end{proof}

\begin{lem}\label{Gray-tensor-of-monos}
	Let $f\colon A \to X$ and $g \colon B \to Y$ be monomorphisms in $\mcSet$.
	Then $f \hat \odot g$ is again a monomorphism.
	Moreover:
	\begin{enumerate}
		\item if both $f$ and $g$ are regular then so is $f \hat \odot g$;
		\item if either $f$ or $g$ is entire then so is $f \hat \odot g$;
		\item if both $f$ and $g$ are entire then $f \hat \odot g$ is invertible; and
		\item if either $f$ or $g$ is entire then the square
		\[
		\begin{tikzcd}[row sep = large]
		(X \otimes B) \cup (A \otimes Y)
		\arrow [r]
		\arrow [d, swap, "f \hat \otimes g"] &
		(X \pseudo B) \cup (A \pseudo Y)
		\arrow [d, "f \hat \pseudo g"] \\
		X \otimes Y
		\arrow [r, "\mu_{X,Y}"] &
		X \pseudo Y
		\end{tikzcd}
		\]
		is a pushout in $\mcSet$ where the upper horizontal map is induced by $\mu$.
	\end{enumerate}
\end{lem}
\begin{proof}
	Since a map in $\mcSet$ is a monomorphism if and only if its underlying map in $\cSet$ is a monomorphism, the first (un-numbered) assertion follows from \cref{geom-prod-of-monos}, \cref{GrayFundamentalProp}(1) and the cocontinuity of $\upsilon$.
	We will assume for the sake of simplicity that $f \hat \odot g$ is an inclusion.
	
	(1$\otimes$)
	Let $x \otimes y$ be a non-degenerate cube in $\dom(f \hat \otimes g)$.
	By duality, we may assume that $x$ is in $A$.
	If $x \otimes y$ is marked in $X \otimes Y$, then either $x$ is marked in $X$ or $y$ is marked in $Y$.
	It follows (by the regularity of $f$ in the former case) that $x \otimes y$ is marked in $A \otimes Y$.
	This shows that $f \hat \otimes g$ is regular.
	
	(1$\pseudo$)
	Let $x \otimes y$ be a marked non-degenerate cube in $X \pseudo Y$.
	Suppose that $x \otimes y$ is in the image of $f \hat \pseudo g$.
	The case (1$\otimes$) combined with the commutativity of the square in (4) imply that if $x \otimes y$ is marked in $X \otimes Y$ then it is also marked in $\dom(f \hat \pseudo g)$.
	Thus by \cref{comparing-Gray-tensors}, it suffices to consider the case where $x \in X_m$ and $y \in Y_n$ for some $m,n \ge 1$.
	But in this case $x \otimes y$ is marked in $\dom(f \hat \pseudo g)$ by the definition of $\pseudo$.
	
	(2)
	Since $\upsilon$ preserves colimits, we have $\upsilon(f \hat \odot g) \cong \upsilon f \hat \odot \upsilon g$.
	Thus this assertion follows from the fact that the pushout of an isomorphism along any map is itself an isomorphism.
	
	(3$\otimes$)
	We know from (2) that $f \hat \otimes g$ is entire, so it suffices to show that this map is also regular.
	Let $x \otimes y$ be a marked non-degenerate cube in $X \otimes Y$.
	Then either $x$ is marked in $X$ or $y$ is marked in $Y$.
	The cube $x \otimes y$ is then marked in $X \otimes B$ in the first subcase and it is marked in $A \otimes Y$ in the second subcase.
	Thus $f \hat \otimes g$ is indeed regular.
	
	(4)
	By (2), each map in this square is entire.
	Thus its image under $\upsilon$ is trivially a pushout in $\cSet$.
	Moreover, for each of the horizontal maps, \cref{comparing-Gray-tensors} implies that the codomain is obtained from the domain by marking those cubes $x \otimes y$ such that $x \in X_m$ and $y \in Y_n$ with $m,n \ge 1$.
	Now the assertion follows by \cref{colim-in-mcSet}.
	
	(3$\pseudo$)
	This case follows from (3$\otimes$) and (4).
\end{proof}

\begin{prop}\label{Gray-tensor-of-bdry}
	For any $m,n \ge 0$, the Leibniz Gray tensor product $(\partial \cube m \incl \cube m) \hat \otimes (\partial \cube n \incl \cube n)$ of boundary inclusions in $\mcSet$ is isomorphic to $\partial \cube{m+n} \incl \cube{m+n}$.
\end{prop}
\begin{proof}
	This is a straightforward consequence of \cref{geom-prod-of-bdry}, \cref{GrayFundamentalProp}(3) and the cocontinuity of $(-)^\flat$.
\end{proof}



\section{Model structure for comical sets} \label{sec:model}

In this section, we construct two families of model structures on the category $\mcSet$ of marked cubical sets.
The former of those has as its fibrant objects (\emph{saturated}) \emph{comical sets}, which we will define, and it is our tentative model for the theory of weak $\omega$-categories.
The fibrant objects of the latter are the $n$-trivial comical sets, and it is our tentative model for the theory of $(\infty,n)$-categories.

	A comical set is to be thought of as a kind of weak $\omega$-category, and an $n$-cube therein represents an $n$-dimensional morphism.
	The $(n-1)$-source of such an $n$-cube is the ``composite'' of the faces $\face{k}{\varepsilon}$ with $k+\varepsilon$ odd, and similarly the $(n-1)$-target is given by the even faces.
	(This idea of parity-based decomposition into source and target goes back to \cite{street:algebra-of-oriented-simplexes} where Street considers the free $\omega$-categories on simplices.
	In the case of cubes, see e.g., \cite{aitchison:cubes,street:parity-complexes,steiner:algebra-of-directed-complexes,al-agl-brown-steiner}).
	For instance, a $2$-cube can be seen as a morphism of the form: 
	\[
	\begin{tikzpicture}[baseline = 15]
	\filldraw
	(0,0) circle [radius = 1pt]
	(1.2,0)  circle [radius = 1pt]
	(0,1.2)  circle [radius = 1pt]
	(1.2,1.2)  circle [radius = 1pt];
	\draw[double, ->] (0.4,0.4) -- (0.8,0.8);
	
	\draw[->] (0.2,0) -- (1,0);
	\draw[->] (0.2,1.2) -- (1,1.2);
	\draw[->] (0,1) -- (0,0.2);
	\draw[->] (1.2,1) -- (1.2,0.2);
	\node[scale = 0.8] at (-0.3,0.6) {$\partial_{1,0}$};
	\node[scale = 0.8] at (0.6,1.4) {$\partial_{2,0}$};
	\node[scale = 0.8] at (1.5,0.6) {$\partial_{1,1}$};
	\node[scale = 0.8] at (0.6,-0.2) {$\partial_{2,1}$};
	\end{tikzpicture}
	\]
	and a $3$-cube represents a morphism between the following composites:
	
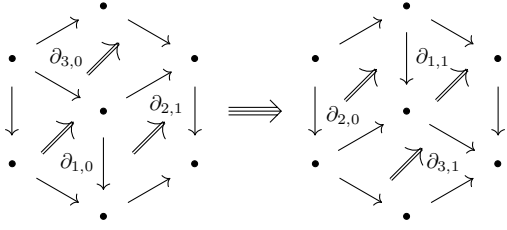
\[
	\begin{tikzpicture}[baseline = -2, scale = 0.7]
	\filldraw
	(150:2) circle [radius = 1.6pt]
	(90:2) circle [radius = 1.6pt]
	(30:2) circle [radius = 1.6pt]
	(-30:2) circle [radius = 1.6pt]
	(-90:2) circle [radius = 1.6pt]
	(-150:2) circle [radius = 1.6pt]
	(0,0) circle [radius = 1.6pt];
	
	\draw[{->}] (135:1.8) -- (105:1.8);
	\draw[{->}] (75:1.8) -- (45:1.8);
	\draw[->] (15:1.8) -- (-15:1.8);
	\draw[{->}] (165:1.8) -- (-165:1.8);	
	\draw[{->}] (-135:1.8) --(-105:1.8);
	\draw[{->}] (-75:1.8) -- (-45:1.8);
	
	\draw[->] (150:1.5) -- (150:0.5);
	\draw[->] (30:0.5) -- (30:1.5);
	\draw[->] (-90:0.5) -- (-90:1.5);
	
	\node[scale=0.8] at (-0.5,-1) {$\partial_{1,0}$};
	\node[scale=0.8] at (-0.7,1) {$\partial_{3,0}$};
	\node[scale=0.8] at (1.2,0.1) {$\partial_{2,1}$};
	
	\draw[->,double] (-150:1) + (-0.3,-0.3) --+ (0.3,0.3);
	\draw[->,double] (90:1) + (-0.3,-0.3) --+ (0.3,0.3);
	\draw[->,double] (-30:1) + (-0.3,-0.3) --+ (0.3,0.3);
	\end{tikzpicture}
	\hspace{10pt}
	\begin{tikzpicture}[baseline = -2]
	\draw
	(0,0) -- (0.7,0)
	(0,0.05)--(0.65,0.05)
	(0,-0.05)--(0.65,-0.05)
	(0.55,0.15)--(0.7,0)--(0.55,-0.15);
	\end{tikzpicture}
	\hspace{10pt}
	\begin{tikzpicture}[baseline = -2, scale = 0.7]
	\filldraw
	(150:2) circle [radius = 1.6pt]
	(90:2) circle [radius = 1.6pt]
	(30:2) circle [radius = 1.6pt]
	(-30:2) circle [radius = 1.6pt]
	(-90:2) circle [radius = 1.6pt]
	(-150:2) circle [radius = 1.6pt]
	(0,0) circle [radius = 1.6pt];
	
	\draw[{->}] (135:1.8) -- (105:1.8);
	\draw[{->}] (75:1.8) -- (45:1.8);
	\draw[->] (15:1.8) -- (-15:1.8);
	\draw[{->}] (165:1.8) -- (-165:1.8);	
	\draw[{->}] (-135:1.8) --(-105:1.8);
	\draw[{->}] (-75:1.8) -- (-45:1.8);

	\draw[->] (90:1.5) -- (90:0.5);
	\draw[->] (-150:1.5) -- (-150:0.5);
	\draw[->] (-30:0.5) -- (-30:1.5);
	
	\node[scale=0.8] at (0.5,1) {$\partial_{1,1}$};
	\node[scale=0.8] at (0.7,-1) {$\partial_{3,1}$};
	\node[scale=0.8] at (-1.2,-0.1) {$\partial_{2,0}$};
	
	\draw[->,double] (150:1) + (-0.3,-0.3) --+ (0.3,0.3);
	\draw[->,double] (-90:1) + (-0.3,-0.3) --+ (0.3,0.3);
	\draw[->,double] (30:1) + (-0.3,-0.3) --+ (0.3,0.3);
	\end{tikzpicture}
	\]
	Marked $n$-cubes are to be thought of as being (weakly) invertible, although not every invertible cube is marked unless the comical set is saturated.

Before defining comical sets, we will need a few auxiliary definitions.

For $n \geq 1$, $1 \leq k  \leq n$, and $\varepsilon \in \{0,1\}$, we denote $\admcube n k \varepsilon$ the $n$-cube with the following marking: a non-degenerate cube $\face{k_1}{\varepsilon_1}\dots\face{k_t}{\varepsilon_t}$, written in the form specified by \cref{normal-form}, is marked whenever this string does not contain $\face{k-1}{\varepsilon}$, $\face{k}{\varepsilon}$, $\face{k}{1-\varepsilon}$, or $\face{k+1}{\varepsilon}$.
(This is exactly the marking described in \cite[Ex.~2.9]{steiner:cubical-nerve}.)
We call this the \defterm{$(k,\varepsilon)$-comical $n$-cube}.
We denote $\obox n k \varepsilon \subset \admcube n k \varepsilon$ the $(k,\varepsilon)$-open box of dimension $n$ (i.e., the cubical subset missing the non-degenerate $n$-cube and the $(k,\varepsilon)$th $(n-1)$-face) endowed with the marking making it a regular subset of $\admcube n k \varepsilon$.
We call $\obox n k \varepsilon$ the \defterm{comical $(k,\varepsilon)$-open box of dimension $n$}.
We call the inclusion $\obox n k \varepsilon \to \admcube n k \varepsilon$ the \defterm{$(k,\varepsilon)$-comical open box inclusion of dimension $n$}.

The \defterm{elementary $(k,\varepsilon)$-comical marking extension of dimension $n$}, denoted by $\admcubep n k \varepsilon \to \admcubepp n k \varepsilon$, is the Leibniz product of the unit $\id \to \trunc{n-2}$ and the comical box inclusion $\obox n k \varepsilon \to \admcube n k \varepsilon$, i.e., the dashed map in:
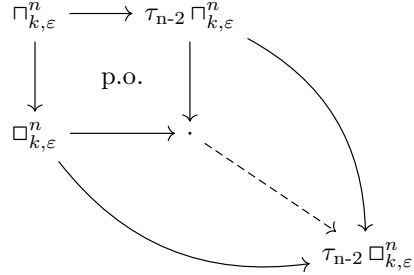
\[
\begin{tikzcd}[row sep = large]
\obox n k \varepsilon
\arrow [d]
\arrow [r]
\arrow [dr, phantom, "\mathrm{p.o.}"] &
\trunc{n-2}\obox n k \varepsilon
\arrow [d]
\arrow [ddr, bend left] & \\
\admcube n k \varepsilon
\arrow [r]
\arrow [drr, bend right]&
\cdot
\arrow [dr, dashed] & \\
& & \trunc{n-2}\admcube n k \varepsilon
\end{tikzcd}
\]

For each $x,y \in \{\nearrow, \swarrow\}$, we define the \defterm{basic Rezk map} $\rezkd x y \incl \rezkc x y$ as the entire inclusion depicted below:
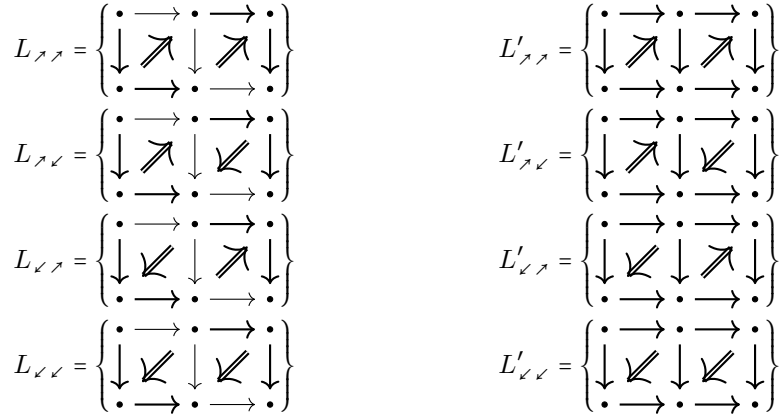
\begin{align*}
\rezkd{\nearrow}{\nearrow} &= 
\left\{\begin{tikzpicture}[baseline = 12]
\foreach \x in {0,1,2}
\filldraw (\x,0) circle [radius = 1pt] (\x,1) circle [radius = 1pt];
\draw[->] (0.2,1) -- (0.8,1);
\draw[->] (1,0.8) -- (1,0.2);
\draw[->] (1.2,0) -- (1.8,0);
\draw[->, line width = 1.3pt] (0,0.8) -- (0,0.2);
\draw[->, line width = 1.3pt] (0.2,0) -- (0.8,0);
\draw[->, line width = 1.3pt] (1.2,1) -- (1.8,1);
\draw[->, line width = 1.3pt] (2,0.8) -- (2,0.2);
\draw[->, double, line width = 1.3pt] (0.3,0.3) -- (0.7,0.7);
\draw[->, double, line width = 1.3pt] (1.3,0.3) -- (1.7,0.7);
\end{tikzpicture}\right\} &
\rezkc{\nearrow}{\nearrow} &= 
\left\{\begin{tikzpicture}[baseline = 12]
\foreach \x in {0,1,2}
\filldraw (\x,0) circle [radius = 1pt] (\x,1) circle [radius = 1pt];
\draw[->, line width = 1.3pt] (0.2,1) -- (0.8,1);
\draw[->, line width = 1.3pt] (1,0.8) -- (1,0.2);
\draw[->, line width = 1.3pt] (1.2,0) -- (1.8,0);
\draw[->, line width = 1.3pt] (0,0.8) -- (0,0.2);
\draw[->, line width = 1.3pt] (0.2,0) -- (0.8,0);
\draw[->, line width = 1.3pt] (1.2,1) -- (1.8,1);
\draw[->, line width = 1.3pt] (2,0.8) -- (2,0.2);
\draw[->, double, line width = 1.3pt] (0.3,0.3) -- (0.7,0.7);
\draw[->, double, line width = 1.3pt] (1.3,0.3) -- (1.7,0.7);
\end{tikzpicture}\right\}\\
\rezkd{\nearrow}{\swarrow} &= 
\left\{\begin{tikzpicture}[baseline = 12]
\foreach \x in {0,1,2}
\filldraw (\x,0) circle [radius = 1pt] (\x,1) circle [radius = 1pt];
\draw[->] (0.2,1) -- (0.8,1);
\draw[->] (1,0.8) -- (1,0.2);
\draw[->] (1.2,0) -- (1.8,0);
\draw[->, line width = 1.3pt] (0,0.8) -- (0,0.2);
\draw[->, line width = 1.3pt] (0.2,0) -- (0.8,0);
\draw[->, line width = 1.3pt] (1.2,1) -- (1.8,1);
\draw[->, line width = 1.3pt] (2,0.8) -- (2,0.2);
\draw[->, double, line width = 1.3pt] (0.3,0.3) -- (0.7,0.7);
\draw[<-, double, line width = 1.3pt] (1.3,0.3) -- (1.7,0.7);
\end{tikzpicture}\right\} &
\rezkc{\nearrow}{\swarrow} &= 
\left\{\begin{tikzpicture}[baseline = 12]
\foreach \x in {0,1,2}
\filldraw (\x,0) circle [radius = 1pt] (\x,1) circle [radius = 1pt];
\draw[->, line width = 1.3pt] (0.2,1) -- (0.8,1);
\draw[->, line width = 1.3pt] (1,0.8) -- (1,0.2);
\draw[->, line width = 1.3pt] (1.2,0) -- (1.8,0);
\draw[->, line width = 1.3pt] (0,0.8) -- (0,0.2);
\draw[->, line width = 1.3pt] (0.2,0) -- (0.8,0);
\draw[->, line width = 1.3pt] (1.2,1) -- (1.8,1);
\draw[->, line width = 1.3pt] (2,0.8) -- (2,0.2);
\draw[->, double, line width = 1.3pt] (0.3,0.3) -- (0.7,0.7);
\draw[<-, double, line width = 1.3pt] (1.3,0.3) -- (1.7,0.7);
\end{tikzpicture}\right\}\\
\rezkd{\swarrow}{\nearrow} &= 
\left\{\begin{tikzpicture}[baseline = 12]
\foreach \x in {0,1,2}
\filldraw (\x,0) circle [radius = 1pt] (\x,1) circle [radius = 1pt];
\draw[->] (0.2,1) -- (0.8,1);
\draw[->] (1,0.8) -- (1,0.2);
\draw[->] (1.2,0) -- (1.8,0);
\draw[->, line width = 1.3pt] (0,0.8) -- (0,0.2);
\draw[->, line width = 1.3pt] (0.2,0) -- (0.8,0);
\draw[->, line width = 1.3pt] (1.2,1) -- (1.8,1);
\draw[->, line width = 1.3pt] (2,0.8) -- (2,0.2);
\draw[<-, double, line width = 1.3pt] (0.3,0.3) -- (0.7,0.7);
\draw[->, double, line width = 1.3pt] (1.3,0.3) -- (1.7,0.7);
\end{tikzpicture}\right\} &
\rezkc{\swarrow}{\nearrow} &= 
\left\{\begin{tikzpicture}[baseline = 12]
\foreach \x in {0,1,2}
\filldraw (\x,0) circle [radius = 1pt] (\x,1) circle [radius = 1pt];
\draw[->, line width = 1.3pt] (0.2,1) -- (0.8,1);
\draw[->, line width = 1.3pt] (1,0.8) -- (1,0.2);
\draw[->, line width = 1.3pt] (1.2,0) -- (1.8,0);
\draw[->, line width = 1.3pt] (0,0.8) -- (0,0.2);
\draw[->, line width = 1.3pt] (0.2,0) -- (0.8,0);
\draw[->, line width = 1.3pt] (1.2,1) -- (1.8,1);
\draw[->, line width = 1.3pt] (2,0.8) -- (2,0.2);
\draw[<-, double, line width = 1.3pt] (0.3,0.3) -- (0.7,0.7);
\draw[->, double, line width = 1.3pt] (1.3,0.3) -- (1.7,0.7);
\end{tikzpicture}\right\}\\
\rezkd{\swarrow}{\swarrow} &= 
\left\{\begin{tikzpicture}[baseline = 12]
\foreach \x in {0,1,2}
\filldraw (\x,0) circle [radius = 1pt] (\x,1) circle [radius = 1pt];
\draw[->] (0.2,1) -- (0.8,1);
\draw[->] (1,0.8) -- (1,0.2);
\draw[->] (1.2,0) -- (1.8,0);
\draw[->, line width = 1.3pt] (0,0.8) -- (0,0.2);
\draw[->, line width = 1.3pt] (0.2,0) -- (0.8,0);
\draw[->, line width = 1.3pt] (1.2,1) -- (1.8,1);
\draw[->, line width = 1.3pt] (2,0.8) -- (2,0.2);
\draw[<-, double, line width = 1.3pt] (0.3,0.3) -- (0.7,0.7);
\draw[<-, double, line width = 1.3pt] (1.3,0.3) -- (1.7,0.7);
\end{tikzpicture}\right\} &
\rezkc{\swarrow}{\swarrow} &= 
\left\{\begin{tikzpicture}[baseline = 12]
\foreach \x in {0,1,2}
\filldraw (\x,0) circle [radius = 1pt] (\x,1) circle [radius = 1pt];
\draw[->, line width = 1.3pt] (0.2,1) -- (0.8,1);
\draw[->, line width = 1.3pt] (1,0.8) -- (1,0.2);
\draw[->, line width = 1.3pt] (1.2,0) -- (1.8,0);
\draw[->, line width = 1.3pt] (0,0.8) -- (0,0.2);
\draw[->, line width = 1.3pt] (0.2,0) -- (0.8,0);
\draw[->, line width = 1.3pt] (1.2,1) -- (1.8,1);
\draw[->, line width = 1.3pt] (2,0.8) -- (2,0.2);
\draw[<-, double, line width = 1.3pt] (0.3,0.3) -- (0.7,0.7);
\draw[<-, double, line width = 1.3pt] (1.3,0.3) -- (1.7,0.7);
\end{tikzpicture}\right\}
\end{align*}
Here thick arrows indicate marked cubes.
More precisely, $\rezkd{\nearrow}{\nearrow}$ is the pushout of the span
\[
\begin{tikzcd}
X &
\cube 1
\arrow [l, swap, "\face 1 1"]
\arrow [r, "\face 1 0"] &
Y
\end{tikzcd}
\]
where $X$ is obtained from $\mcube 2$ by marking $\face 1 0$ and $\face 2 1$, and $Y$ is obtained from $\mcube 2$ by marking $\face 1 1$ and $\face 2 0$.
The codomain $\rezkc{\nearrow}{\nearrow}$ is the $0$-trivialization $\trunc 0(\rezkd{\nearrow}{\nearrow})$.
The marked cubical sets $\rezkd x y, \rezkc x y$ are defined similarly for other choices of $x,y \in \{\nearrow, \swarrow\}$.
By a \defterm{Rezk map} we mean any map of the form
\[
\bigl(\partial \cube m \incl \cube m\bigr) \hat \otimes \bigl(\rezkd x y \incl \rezkc x y\bigr) \hat \otimes \bigl(\partial \cube n \incl \cube n \bigr).
\]

\begin{Def}\leavevmode
\begin{enumerate}
  \item A \defterm{comical set} is a marked cubical set with the right lifting property with respect to the comical open box inclusions and the elementary comical marking extensions.
  \item A \defterm{saturated comical set} is a marked cubical set with the right lifting property with respect to the comical open box inclusions, the elementary comical marking extensions, and the Rezk maps.
\end{enumerate}
\end{Def}
	
\begin{rmk}\label{comical-remark}
	We briefly explain how the definition of (saturated) comical set should be interpreted.
	In the comical $n$-cube $\admcube n k \varepsilon$, any sub-cube not contained in $\face{k-1}{\varepsilon}$, $\face{k}{\varepsilon}$, $\face{k}{1-\varepsilon}$, or $\face{k+1}{\varepsilon}$ is marked.
	In particular the unique non-degenerate $n$-cube is marked, so it can be thought of as an equivalence between the composite of its odd faces and the composite of even faces.
	In other words, the comical $n$-cube $\admcube n k \varepsilon$ exhibits $\face k \varepsilon$ as a composite of $\face{k-1}{\varepsilon}$, $\face{k}{1-\varepsilon}$, and $\face{k+1}{\varepsilon}$. e.g., $\admcube 3 2 0$ looks like:
		
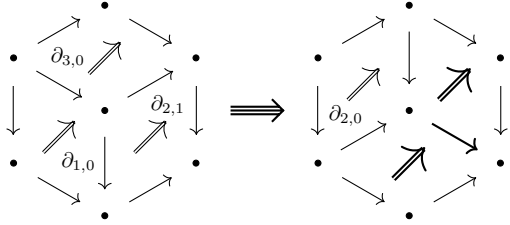
\[
	\begin{tikzpicture}[baseline = -2, scale = 0.7]
	\filldraw
	(150:2) circle [radius = 1.6pt]
	(90:2) circle [radius = 1.6pt]
	(30:2) circle [radius = 1.6pt]
	(-30:2) circle [radius = 1.6pt]
	(-90:2) circle [radius = 1.6pt]
	(-150:2) circle [radius = 1.6pt]
	(0,0) circle [radius = 1.6pt];
	
	\draw[{->}] (135:1.8) -- (105:1.8);
	\draw[{->}] (75:1.8) -- (45:1.8);
	\draw[->] (15:1.8) -- (-15:1.8);
	\draw[{->}] (165:1.8) -- (-165:1.8);	
	\draw[{->}] (-135:1.8) --(-105:1.8);
	\draw[{->}] (-75:1.8) -- (-45:1.8);
	
	\draw[->] (150:1.5) -- (150:0.5);
	\draw[->] (30:0.5) -- (30:1.5);
	\draw[->] (-90:0.5) -- (-90:1.5);
	
	\node[scale=0.8] at (-0.5,-1) {$\partial_{1,0}$};
	\node[scale=0.8] at (-0.7,1) {$\partial_{3,0}$};
	\node[scale=0.8] at (1.2,0.1) {$\partial_{2,1}$};
	
	\draw[->,double] (-150:1) + (-0.3,-0.3) --+ (0.3,0.3);
	\draw[->,double] (90:1) + (-0.3,-0.3) --+ (0.3,0.3);
	\draw[->,double] (-30:1) + (-0.3,-0.3) --+ (0.3,0.3);
	\end{tikzpicture}
	\hspace{10pt}
	\begin{tikzpicture}[baseline = -2]
	\draw[line width = 1.3pt]
	(0,0) -- (0.7,0)
	(0,0.1)--(0.6,0.1)
	(0,-0.1)--(0.6,-0.1)
	(0.5,0.2)--(0.7,0)--(0.5,-0.2);
	\end{tikzpicture}
	\hspace{10pt}
	\begin{tikzpicture}[baseline = -2, scale = 0.7]
	\filldraw
	(150:2) circle [radius = 1.6pt]
	(90:2) circle [radius = 1.6pt]
	(30:2) circle [radius = 1.6pt]
	(-30:2) circle [radius = 1.6pt]
	(-90:2) circle [radius = 1.6pt]
	(-150:2) circle [radius = 1.6pt]
	(0,0) circle [radius = 1.6pt];
	
	\draw[{->}] (135:1.8) -- (105:1.8);
	\draw[{->}] (75:1.8) -- (45:1.8);
	\draw[->] (15:1.8) -- (-15:1.8);
	\draw[{->}] (165:1.8) -- (-165:1.8);	
	\draw[{->}] (-135:1.8) --(-105:1.8);
	\draw[{->}] (-75:1.8) -- (-45:1.8);

	\draw[->] (90:1.5) -- (90:0.5);
	\draw[->] (-150:1.5) -- (-150:0.5);
	\draw[->, line width = 1.3pt] (-30:0.5) -- (-30:1.5);
	
	\node[scale=0.8] at (-1.2,-0.1) {$\partial_{2,0}$};
	
	\draw[->,double] (150:1) + (-0.3,-0.3) --+ (0.3,0.3);
	\draw[->,double, line width = 1.3pt] (-90:1) + (-0.3,-0.3) --+ (0.3,0.3);
	\draw[->,double, line width = 1.3pt] (30:1) + (-0.3,-0.3) --+ (0.3,0.3);
	\end{tikzpicture}
	\]
	One can thus interpret the right lifting property with respect to the comical box inclusions and the comical marking extensions respectively as the existence of composites and the closure of marked cubes under composition.
	In \cref{homotopy-1-categories}, we show how these conditions additionally encode such expected properties of composition as the unit and associative laws, at least for 1-cubes.
\end{rmk}

There are two standard model structures on marked cubical sets:

\begin{thm}[Model structure for comical sets] \label{comical-model-structure}
  The category $\mcSet$ carries two model structures:
  \begin{enumerate}
    \item the \defterm{comical model structure} characterized by the following properties:
    \begin{itemize}
		\item The cofibrations are the monomorphisms.
		\item The set of\begin{itemize}
			\item comical open box inclusions, and
			\item elementary comical marking extensions
		\end{itemize}
	forms a pseudo-generating set of trivial cofibrations.
    \end{itemize}
    \item the \defterm{saturated comical model structure} characterized by the following properties:
    \begin{itemize}
		\item The cofibrations are the monomorphisms.
		\item The set of
		\begin{itemize}
			\item comical open box inclusions,
			\item elementary comical marking extensions, and
			\item Rezk maps
		\end{itemize}
	forms a pseudo-generating set of trivial cofibrations.      
    \end{itemize}
    Both of these model structures are combinatorial, left proper, monoidal with respect to either of the Gray tensor products, and have all objects cofibrant.
  \end{enumerate}
\end{thm}

The proof of this theorem is an application of the Cisinski--Olschok theory and verification of the closure of anodyne maps under pushout-product.
The latter part is contained in \cref{monoidal-model} below.

\begin{Def}
  We say that a map of marked cubical sets $X \to Y$ is 
  \begin{enumerate}
    \item a \defterm{comical marking extension} if it is in the cellular closure of the elementary comical marking extensions.
    \item \defterm{comical} if it is in the cellular closure of the comical open box inclusions and the elementary comical marking extensions.
    \end{enumerate}
\end{Def}

\begin{lem}\label{monoidal-model}
	For any $1 \le k \le m$, $\varepsilon \in \{0,1\}$ and $n \ge 0$ (or $n \ge 1$ for $g$), the Leibniz Gray tensor products
	\[
	\begin{gathered}
	f = \bigl(\obox m k \varepsilon \hookrightarrow \admcube m k \varepsilon \bigl) \hat \odot \bigl(\partial \cube n \hookrightarrow \cube n \bigr)\\
	g = \bigl(\obox m k \varepsilon \hookrightarrow \admcube m k \varepsilon \bigl) \hat \odot \bigl(\cube n \hookrightarrow \mcube n \bigr)\\
	h = \bigl(\admcubep{m}{k}{\varepsilon} \hookrightarrow \admcubepp{m}{k}{\varepsilon} \bigl) \hat \odot \bigl(\partial \cube n \hookrightarrow \cube n \bigr)
	\end{gathered}
	\]
	are all comical.
\end{lem}
\begin{proof}
	Since the case $n = 0$ is trivial, we will assume otherwise.
	
	Consider a face of $\cube {m+n}$ whose normal form $\face{k_1}{\varepsilon_1}\dots\face{k_c}{\varepsilon_c}$ does not involve $\face{k-1}{\varepsilon}$, $\face{k}{0}$, $\face{k}{1}$ or $\face{k+1}{\varepsilon}$.
	Then clearly any terminal segment of this normal form does not involve any of these four $\partial$'s.
	This observation implies that the second isomorphism of \cref{geom-prod-of-bdry} may be lifted to the following commutative square:
	\[
	\begin{tikzcd}[row sep = large]
	\obox {m+n} k \varepsilon
	\arrow [r] 
	\arrow [d, hook] &
	\bigl(\admcube{m}{k}{\varepsilon} \odot \partial \cube n\bigr) \cup \bigl(\obox{m}{k}{\varepsilon} \odot \cube n\bigr)
	\arrow [d, hook, "f"] \\
	\admcube{m+n}{k}{\varepsilon}
	\arrow [r] &
	\admcube{m}{k}{\varepsilon}\odot\cube n 
	\end{tikzcd}
	\]
	Observe that this is a pushout square on the underlying cubical set level.
	
	In the case $\odot = \otimes$, it is in fact a pushout square in $\mcSet$.
	To see this, it suffices to check that the marking on $\admcube{m}{k}{\varepsilon}\otimes\cube n$ agrees with that described in \cref{colim-in-mcSet}.
	This is indeed the case since $f$ is regular by \cref{Gray-tensor-of-monos}(1) and the only marked non-degenerate cube in $\cod(f) \setminus \dom(f)$ is the $(m+n)$-cube, which is the image of a marked cube under the lower horizontal map.
	Thus $f$ is indeed comical.
	
	Now we consider the case $\odot = \pseudo$.
	If $m=1$ then we can simply repeat the above argument since we have natural isomorphisms $\obox 1 1 \epsilon \pseudo (-) \cong \obox 1 1 \epsilon \otimes (-)$ and $\admcube 1 1 \epsilon \pseudo (-) \cong \admcube 1 1 \epsilon \otimes (-)$ by \cref{comparing-Gray-tensors}.
	So assume $m \ge 2$.
	Then there is an extra marked non-degenerate cube in $\cod(f)\setminus\dom(f)$, namely $\face k \epsilon$.
	But then it is straightforward to check that $\face{k-1}{\epsilon}$, $\face{k}{1-\epsilon}$ and $\face{k+1}{\epsilon}$ are also marked (whenever they exist).
	So $f$ can be written as a pushout of the open box inclusion $\obox{m+n}{k}{\epsilon} \incl \admcube{m+n}{k}{\epsilon}$ followed by a pushout of the comical marking extension $\admcubep{m+n}{k}{\epsilon} \incl \admcubepp{m+n}{k}{\epsilon}$.
	
	The map $g$ is entire by \cref{Gray-tensor-of-monos}(2).
	Similarly to the above argument, one can deduce the existence of the following commutative square of entire monomorphisms:
	\[
	\begin{tikzcd}[row sep = large]
	\admcubep{m+n}{k}{\varepsilon}
	\arrow [r]
	\arrow [d, hook] &
	\bigl(\admcube{m}{k}{\varepsilon}\odot\cube n\bigr) \cup \bigl(\obox{m}{k}{\varepsilon}\odot\mcube n\bigr)
	\arrow [d, hook, "g"] \\
	\admcubepp{m+n}{k}{\varepsilon}
	\arrow [r] &
	\admcube{m}{k}{\varepsilon}\odot\mcube n
	\end{tikzcd}
	\]
	One can check that, in the case where $\odot = \pseudo$ and $m \ge 2$, the map $g$ is in fact invertible.
	Otherwise, the only cube in $\cod(g)$ that is not marked in $\dom(g)$ is $\face{k}{\varepsilon}$ and it is the image of a marked cube under the lower horizontal map.
	This shows that the above square is a pushout.
	Hence $g$ is a comical marking extension.
	
	Similarly, one can check that the following square is a pushout:
	\[
	\begin{tikzcd}[row sep = large]
	\admcubep{m+n}{k}{\varepsilon}
	\arrow [r]
	\arrow [d, hook] &
	\bigl(\admcubepp{m}{k}{\varepsilon}\odot\partial\cube n\bigr) \cup \bigl(\admcubep{m}{k}{\varepsilon}\odot\cube n\bigr)
	\arrow [d, hook, "h"] \\
	\admcubepp{m+n}{k}{\varepsilon}
	\arrow [r] &
	\admcubepp{m}{k}{\varepsilon}\odot\cube n
	\end{tikzcd}
	\]
	Therefore $h$ is a comical marking extension.
\end{proof}

\begin{proof}[Proof of \cref{comical-model-structure}]
    We apply the Cisinski--Olschok theory, i.e., \cref{olschok} with $\calK = \mcSet$ and $I$ the set of boundary inclusions and markers.
    The set $S$ consists of the comical open box inclusions and the comical marking extensions in (1), and it additionally contains all Rezk maps in (2).
    For our cylinder functor $C$, we can use either $\mcube 1 \otimes (-)$ or $\mcube 1 \varoast (-)$ as they are equal by \cref{comparing-Gray-tensors}.
    This produces a model structure on $\mcSet$ in which the cofibrations are the monomorphisms and $\Lambda(\mcSet,I,C,S)$ is a pseudo-generating set of trivial cofibrations.
    
    It remains to prove that the set $S$ is in fact pseudo-generating, and moreover the model structure is monoidal with respect to either of the Gray tensor products.
    By duality and \cref{left-Quillen-bifunctor} it suffices to show that:
    \begin{itemize}
    	\item $f \hat \odot g$ is in the cellular closure of $I$ whenever $f,g \in I$; and
    	\item $f \hat \odot g$ is in the cellular closure of $S$ whenever $f \in S$ and $g \in I$.
    \end{itemize}
	The first clause essentially follows from the unmarked version (\cref{geom-prod-of-monos}).
	We now treat the second clause.
    
	There are three kinds of maps in $S$, namely:
	\begin{itemize}
		\item [(A)] comical box inclusions;
		\item [(B)] elementary comical marking extensions; and
		\item [(C)] Rezk maps,
	\end{itemize}
	and two kinds of maps in $I$, namely:
	\begin{itemize}
		\item [(a)] boundary inclusions; and
		\item [(b)] markers.
	\end{itemize}
	The case (Ca$\otimes$) is a straightforward consequence of the associativity of $\hat \otimes$ and \cref{Gray-tensor-of-bdry}.
	The case (Ca$\varoast$) then follows by \cref{Gray-tensor-of-monos}(4).
	In the cases (Bb) and (Cb), the map $f \hat \odot g$ is invertible by \cref{Gray-tensor-of-monos}(3).
	The remaining cases are treated in \cref{monoidal-model}.
\end{proof}

There are also $n$-trivial versions of these model structures.

\begin{thm}[Model structure for $n$-trivial comical sets] \label{n-comical-model-structure}
The category $\mcSet$ carries two families of model structures:
\begin{enumerate}
  \item the \defterm{$n$-trivial comical model structure} characterized by the following properties:
  \begin{itemize}
	\item The cofibrations are the monomorphisms.
	\item The set of
	\begin{itemize}
		\item comical open box inclusions,
		\item elementary comical marking extensions of dimension $\le n+1$, and
		\item markers of dimension $>n$
	\end{itemize}
forms a pseudo-generating set of trivial cofibrations
\end{itemize}
  \item the \defterm{saturated $n$-trivial comical model structure} characterized by the following properties:
  \begin{itemize}
	\item The cofibrations are the monomorphisms.
	\item The set of
	\begin{itemize}
		\item comical open box inclusions,
		\item elementary comical marking extensions of dimension $\le n+1$,
		\item markers of dimension $>n$, and
		\item Rezk maps
	\end{itemize}
form a pseudo-generating set of trivial cofibrations.
\end{itemize}
\end{enumerate}
\end{thm}

\begin{proof}
  Analogous to the proof of \cref{comical-model-structure}.
  Note that the Leibniz Gray tensor product of the $m$-marker with any monomorphism is in the cellular closure of the $m'$-markers with $m' \ge m$.
\end{proof}

\begin{prop} \label{trunc-left-Quillen-cubical}
The functor $\trunc n \colon \mcSet \to \mcSet$ is a left Quillen functor from the $n$-trivial comical model structure (resp.~saturated $n$-trivial comical model structure) to the comical model structure (resp.~saturated comical model structure).
\end{prop}

\begin{lem}\label{fill-and-mark}
	For any $n \ge 0$, $1 \le k \le n+2$ and $\varepsilon \in \{0,1\}$, the map $\trunc{n}\obox {n+2} k \varepsilon \hookrightarrow \trunc{n} \admcube {n+2} k \varepsilon = \admcubepp {n+2} k \varepsilon$ is comical.
\end{lem}
\begin{proof}
	Recall the defining pushout square of the comical marking extension:
	\[
	\begin{tikzcd}[row sep = large]
	\obox {n+2} k \varepsilon
	\arrow [d]
	\arrow [r]
	\arrow [dr, phantom, "\mathrm{p.o.}"] &
	\trunc{n}\obox {n+2} k \varepsilon
	\arrow [d]
	\arrow [ddr, bend left] & \\
	\admcube {n+2} k \varepsilon
	\arrow [r]
	\arrow [drr, bend right]&
	\cdot
	\arrow [dr, dashed] & \\
	& & \trunc{n}\admcube {n+2} k \varepsilon
	\end{tikzcd}
	\]
	This diagram exhibits the desired result.
\end{proof}

\begin{proof}[Proof of \cref{trunc-left-Quillen-cubical}]
	That $\trunc n$ preserves cofibrations is obvious.
	Thus it suffices to check (by \cite[Lem.~7.14]{joyal-tierney:qcat-vs-segal}) that $\trunc n$ sends each member $f \colon X \to Y$ of the pseudo-generating set to a trivial cofibration.
	
	Observe that, unless $f$ is the open box inclusion $\obox {m} k \epsilon \incl \admcube {m} k \epsilon$ with $m \ge n+2$, any marked cube in $Y$ of dimension $>n$ admits a (not necessarily marked) preimage in $X$.
	In these cases, the naturality square for the unit
	\[
	\begin{tikzcd}
	X
	\arrow [r]
	\arrow [d, "f", swap] &
	\trunc n X
	\arrow [d, "\trunc n f"]\\
	Y
	\arrow [r] &
	\trunc n Y
	\end{tikzcd}
	\]
	is a pushout in $\mcSet$ by \cref{colim-in-mcSet}, so $\trunc n f$ is a trivial cofibration.
	
	So assume that $f$ is the open box inclusion $\obox {m} k \epsilon \incl \admcube {m} k \epsilon$ with $m \ge n+2$.
	Then we have the following commutative square:
	\[
	\begin{tikzcd}[row sep = large]
		\trunc{m-2} \obox m k \epsilon
		\arrow [d, "\trunc{m-2}f", swap]
		\arrow [r] &
		\trunc n \obox m k \epsilon
		\arrow [d, "\trunc n f = \trunc n(\trunc{m-2}f)"]\\
		\trunc{m-2}\admcube{m}{k}{\epsilon}
		\arrow [r] &
		\trunc n \admcube m k \epsilon
	\end{tikzcd}
	\]
	Observe that the left vertical map is comical by \cref{fill-and-mark} (with suitable substitution), and moreover it satisfies the condition on $f$ described in the previous paragraph.
	Thus this square is a pushout, exhibiting $\trunc n f$ as a comical map.
	This completes the proof.
\end{proof}

As we mentioned earlier, our definition of comical box inclusion uses the marking described in \cite{steiner:cubical-nerve} where Steiner characterizes the nerves of strict $\omega$-categories.
Phrased in the language of comical sets, his characterization implies the following result:

\begin{thm}[{cf.~\cite[Thm.~3.16]{steiner:cubical-nerve}}]
	The cubical nerve of a globular $\omega$-category, or equivalently the underlying cubical set of a cubical $\omega$-category with connections, is a comical set. \qed
\end{thm}

This is analogous to the statement that the simplicial nerve of a strict $\omega$-category is a complicial set \cite{verity:complicial}.

We conclude this section with the following observation which will be useful in \cref{sec:triangulating-quillen}.

\begin{prop}\label{elementary-boxes}
	For any $n \ge 2$, $1 \le k \le n$ and $\varepsilon \in \{0,1\}$, the comical box inclusion $\obox{n}{k}{\varepsilon} \hookrightarrow \admcube{n}{k}{\varepsilon}$ may be written as:
	\begin{align*}
	& \bigl(\obox{2}{1}{\varepsilon} \hookrightarrow \admcube{2}{1}{\varepsilon}\bigr) \hat\otimes \bigl(\partial\cube{n-2} \hookrightarrow \cube{n-2}\bigr), & & k = 1,\\
	& \bigl(\partial \cube{k-2} \hookrightarrow \cube{k-2} \bigr) \hat \otimes \bigl(\obox{3}{2}{\varepsilon} \hookrightarrow \admcube{3}{2}{\varepsilon}\bigr) \hat \otimes \bigl(\partial \cube{n-k-1} \hookrightarrow \cube{n-k-1}\bigr), & & 1 < k < n,\\
	& \bigl(\partial\cube{n-2} \hookrightarrow \cube{n-2}\bigr) \hat \otimes \bigl(\obox{2}{2}{\varepsilon} \hookrightarrow \admcube{2}{2}{\varepsilon}\bigr), & & k = n.
	\end{align*}
\end{prop}
\begin{proof}
	It is easy to check that the underlying cubical maps match, and also the markings on the codomains match.
	Now observe that these Leibniz Gray tensor products are regular by \cref{Gray-tensor-of-monos}(1).
\end{proof}



\newcommand{\tikzsquare}[1]
{
\begin{tikzpicture}[baseline = 25]
\node at (0,0) {$#1[3]$};
\node at (0,2) {$#1[1]$};
\node at (2.3,0) {$#1[4]$};
\node at (2.3,2) {$#1[2]$};

\draw[{#1[5]}] (0.4,2) -- (1.9,2);
\node[scale = 0.8] at (1.15,2.3) {$#1[6]$};
\draw[{#1[7]}] (2.3,1.7) -- (2.3,0.3);
\node[scale = 0.8] at (2.6,1) {$#1[8]$};
\draw[{#1[9]}] (0,1.7) -- (0,0.3);
\node[scale = 0.8] at (-0.3,1) {$#1[10]$};
\draw[{#1[11]}] (0.4,0) -- (1.9,0);
\node[scale = 0.8] at (1.15,-0.3) {$#1[12]$};
\draw[->, double, line width = 1.3pt] (0.4,0.4) -- (1.9,1.6);
\node[fill = white] at (1.15,1) {$#1[13]$};
\end{tikzpicture}
}

\newcommand{\tikzcube}[1]
{
	\begin{tikzpicture}[baseline = -2, scale = 0.8]
	\node at (150:2) {$#1[1]$};
	\node at (90:2) {$#1[2]$};
	\node at (30:2) {$#1[3]$};
	\node at (-30:2) {$#1[4]$};
	\node at (-90:2) {$#1[6]$};
	\node at (-150:2) {$#1[5]$};
	\node at (0,0) {$#1[7]$};
	
	\draw[{#1[9]}] (135:1.8) -- (105:1.8);
	\node[scale = 0.8] at (120:1.9) {$#1[10]$};
	\draw[{#1[11]}] (75:1.8) -- (45:1.8);
	\node[scale = 0.8] at (60:1.9) {$#1[12]$};
	\draw[{#1[13]}] (15:1.8) -- (-15:1.8);
	\node[scale=0.8] at (1.9,0) {$#1[14]$};
	\draw[{#1[15]}] (165:1.8) -- (-165:1.8);
	\node[scale=0.8] at (-1.9,0) {$#1[16]$};	
	\draw[{#1[17]}] (-135:1.8) --(-105:1.8);
	\node[scale = 0.8] at (-120:1.9) {$#1[18]$};
	\draw[{#1[19]}] (-75:1.8) -- (-45:1.8);
	\node[scale = 0.8] at (-60:1.9) {$#1[20]$};
	
	\draw[{#1[21]}] (150:1.5) -- (150:0.5);
	\node[scale=0.8] at (140:1) {$#1[22]$};
	\draw[{#1[23]}] (30:0.5) -- (30:1.5);
	\node[scale=0.8] at (40:1) {$#1[24]$};
	\draw[{#1[25]}] (-90:0.5) -- (-90:1.5);
	\node[scale=0.8] at (0.2,-0.9) {$#1[26]$};
	
	\node at (90:1) {$#1[33]$};
	\node at (-30:1) {$#1[34]$};
	\node at (-150:1) {$#1[35]$};
	\end{tikzpicture}
	\hspace{10pt} #1[39] \hspace{10pt}
	\begin{tikzpicture}[baseline = -2, scale = 0.8]
	\node at (150:2) {$#1[1]$};
	\node at (90:2) {$#1[2]$};
	\node at (30:2) {$#1[3]$};
	\node at (-30:2) {$#1[4]$};
	\node at (-90:2) {$#1[6]$};
	\node at (-150:2) {$#1[5]$};
	
	\draw[{#1[9]}] (135:1.8) -- (105:1.8);
	\node[scale = 0.8] at (120:1.9) {$#1[10]$};
	\draw[{#1[11]}] (75:1.8) -- (45:1.8);
	\node[scale = 0.8] at (60:1.9) {$#1[12]$};
	\draw[{#1[13]}] (15:1.8) -- (-15:1.8);
	\node[scale=0.8] at (1.9,0) {$#1[14]$};
	\draw[{#1[15]}] (165:1.8) -- (-165:1.8);
	\node[scale=0.8] at (-1.9,0) {$#1[16]$};	
	\draw[{#1[17]}] (-135:1.8) --(-105:1.8);
	\node[scale = 0.8] at (-120:1.9) {$#1[18]$};
	\draw[{#1[19]}] (-75:1.8) -- (-45:1.8);
	\node[scale = 0.8] at (-60:1.9) {$#1[20]$};
	
	\node at (0,0) {$#1[8]$};
	
	\draw[{#1[27]}] (90:1.5) -- (90:0.5);
	\node[scale=0.8] at (0.2,0.9) {$#1[28]$};
	\draw[{#1[29]}] (-150:1.5) -- (-150:0.5);
	\node[scale=0.8] at (-160:1) {$#1[30]$};
	\draw[{#1[31]}] (-30:0.5) -- (-30:1.5);
	\node[scale=0.8] at (-20:1) {$#1[32]$};
	
	\node at (30:1) {$#1[36]$};
	\node at (150:1) {$#1[37]$};
	\node at (-90:1) {$#1[38]$};
	\end{tikzpicture}}

\newcommand{\tikzcubedblarr}[1]
{
	\begin{tikzpicture}[baseline = -2]
	\node at (150:2) {$#1[1]$};
	\node at (90:2) {$#1[2]$};
	\node at (30:2) {$#1[3]$};
	\node at (-30:2) {$#1[4]$};
	\node at (-90:2) {$#1[6]$};
	\node at (-150:2) {$#1[5]$};
	\node at (0,0) {$#1[7]$};
	
	\draw[{#1[9]}] (135:1.8) -- (105:1.8);
	\node[scale = 0.8] at (120:1.9) {$#1[10]$};
	\draw[{#1[11]}] (75:1.8) -- (45:1.8);
	\node[scale = 0.8] at (60:1.9) {$#1[12]$};
	\draw[{#1[13]}] (15:1.8) -- (-15:1.8);
	\node[scale=0.8] at (1.9,0) {$#1[14]$};
	\draw[{#1[15]}] (165:1.8) -- (-165:1.8);
	\node[scale=0.8] at (-1.9,0) {$#1[16]$};	
	\draw[{#1[17]}] (-135:1.8) --(-105:1.8);
	\node[scale = 0.8] at (-120:1.9) {$#1[18]$};
	\draw[{#1[19]}] (-75:1.8) -- (-45:1.8);
	\node[scale = 0.8] at (-60:1.9) {$#1[20]$};
	
	\draw[{#1[21]}] (150:1.5) -- (150:0.5);
	\node[scale=0.8] at (140:1) {$#1[22]$};
	\draw[{#1[23]}] (30:0.5) -- (30:1.5);
	\node[scale=0.8] at (40:1) {$#1[24]$};
	\draw[{#1[25]}] (-90:0.5) -- (-90:1.5);
	\node[scale=0.8] at (0.2,-0.9) {$#1[26]$};
	
	\draw[->,double] (-150:1) + (-0.2,-0.2) --+ (0.2,0.2);
	\draw[->,double] (90:1) + (-0.2,-0.2) --+ (0.2,0.2);
	\draw[->,double] (-30:1) + (-0.2,-0.2) --+ (0.2,0.2);
	\end{tikzpicture}
	\hspace{10pt} = \hspace{10pt}
	\begin{tikzpicture}[baseline = -2]
	\node at (150:2) {$#1[1]$};
	\node at (90:2) {$#1[2]$};
	\node at (30:2) {$#1[3]$};
	\node at (-30:2) {$#1[4]$};
	\node at (-90:2) {$#1[6]$};
	\node at (-150:2) {$#1[5]$};
	
	\draw[{#1[9]}] (135:1.8) -- (105:1.8);
	\node[scale = 0.8] at (120:1.9) {$#1[10]$};
	\draw[{#1[11]}] (75:1.8) -- (45:1.8);
	\node[scale = 0.8] at (60:1.9) {$#1[12]$};
	\draw[{#1[13]}] (15:1.8) -- (-15:1.8);
	\node[scale=0.8] at (1.9,0) {$#1[14]$};
	\draw[{#1[15]}] (165:1.8) -- (-165:1.8);
	\node[scale=0.8] at (-1.9,0) {$#1[16]$};	
	\draw[{#1[17]}] (-135:1.8) --(-105:1.8);
	\node[scale = 0.8] at (-120:1.9) {$#1[18]$};
	\draw[{#1[19]}] (-75:1.8) -- (-45:1.8);
	\node[scale = 0.8] at (-60:1.9) {$#1[20]$};
	
	\node at (0,0) {$#1[8]$};
	
	\draw[{#1[27]}] (90:1.5) -- (90:0.5);
	\node[scale=0.8] at (0.2,0.9) {$#1[28]$};
	\draw[{#1[29]}] (-150:1.5) -- (-150:0.5);
	\node[scale=0.8] at (-160:1) {$#1[30]$};
	\draw[{#1[31]}] (-30:0.5) -- (-30:1.5);
	\node[scale=0.8] at (-20:1) {$#1[32]$};
	
	\draw[->,double] (150:1) + (-0.2,-0.2) --+ (0.2,0.2);
	\draw[->,double] (-90:1) + (-0.2,-0.2) --+ (0.2,0.2);
	\draw[->,double] (30:1) + (-0.2,-0.2) --+ (0.2,0.2);
	\end{tikzpicture}}


\section{Homotopy 1-categories of comical sets} \label{homotopy-1-categories}

Suppose we are given two 1-cubes $f,g \colon x \to y$ in a comical set $X$.
Then a marked 2-cube satisfying any one of the following boundary conditions may be reasonably regarded as a homotopy $f \sim g$:
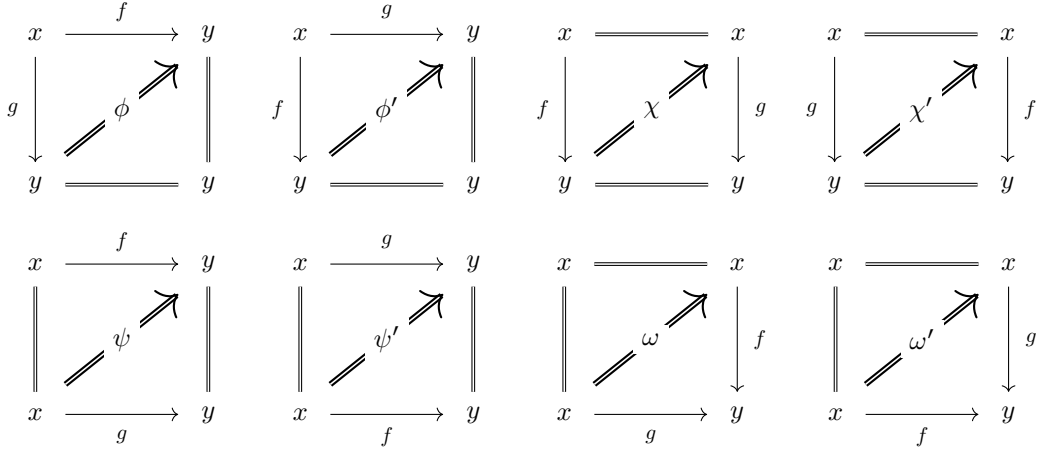
\[
\begin{aligned}
\readlist\htpyone{x,y,y,y,->,f,double,,->,g,double,,\phi}
\tikzsquare{\htpyone} & & 
\readlist\htpyonep{x,y,y,y,->,g,double,,->,f,double,,\phi'}
\tikzsquare{\htpyonep} & & 
\readlist\htpytwo{x,x,y,y,double,,->,g,->,f,double,,\chi}
\tikzsquare{\htpytwo} & & 
\readlist\htpytwop{x,x,y,y,double,,->,f,->,g,double,,\chi'}
\tikzsquare{\htpytwop} \\
\readlist\htpythree{x,y,x,y,->,f,double,,double,,->,g,\psi}
\tikzsquare{\htpythree} & &
\readlist\htpythreep{x,y,x,y,->,g,double,,double,,->,f,\psi'}
\tikzsquare{\htpythreep} & &
\readlist\htpyfour{x,x,x,y,double,,->,f,double,,->,g,\omega}
\tikzsquare{\htpyfour} & &
\readlist\htpyfourp{x,x,x,y,double,,->,g,double,,->,f,\omega'}
\tikzsquare{\htpyfourp}
\end{aligned}
\]
Here equalities indicate degenerate (and hence marked) 1-cubes.
\begin{prop}\label{homotopy-relation}
	If any one of the above boundary conditions admits a marked solution in the comical set $X$ then so do the others.
\end{prop}

\begin{proof}
	Consider the following picture:
	\[
	\readlist\huone{x,x,x,y,x,x,x,x,double,,double,,->,f,double,,double,,->,f,double,,double,,double,,double,,double,,->,g,deg,min,deg,\omega,deg,\omega',}
	\tikzcube{\huone}
	\]
	(Here the faces labelled ``$deg$'' are fully degenerate on the $0$-cube $x$, and the face labelled ``$min$'' is the min-connection on $f$.)
	If we have a marked 2-cube $\omega$ satisfying the boundary condition specified above, then this picture may be interpreted as a map $\trunc 1 \obox{3}{3}{1} \to X$ which may be extended to $\admcubepp 3 3 1$ by \cref{fill-and-mark}, yielding a marked 2-cube $\omega'$.
	Conversely, if we are given $\omega'$ then this picture specifies a map $\trunc 1 \obox 3 1 1 \to X$ and extending it to $\admcubepp 3 1 1$ yields a marked 2-cube $\omega$.
	
	Similarly, the following picture shows that $\chi$ exists if and only if $\omega$ does:
	\[
	\readlist\hutwo{x,x,x,y,x,y,x,x,double,,double,,->,g,double,,->,g,double,,double,,double,,->,f,double,,double,,->,g,deg,\chi,\omega,min,deg,deg,}
	\tikzcube{\hutwo}
	\]
	Dually, $\chi'$ exists if and only if $\omega'$ does.
	
	The following picture shows that $\psi$ exists if and only if $\omega'$ does (and dually $\psi'$ exists if and only if $\omega$ does):
	\[
	\readlist\huthree{x,x,y,y,x,x,x,x,double,,->,g,double,,double,,double,,->,g,double,,->,f,double,,double,,double,,->,g,\omega',\psi,deg,deg,deg,min,}
	\tikzcube{\huthree}
	\]
	
	Finally the following picture shows that $\phi$ exists if and only if $\chi'$ does (and dually $\phi'$ exists if and only if $\chi$ does):
	\[
	\readlist\hufour{x,x,y,y,x,y,x,x,double,,->,f,double,,double,,->,g,double,,double,,->,f,->,g,double,,double,,->,f,min,\phi,min,deg,deg,\chi',}
	\tikzcube{\hufour}
	\]
	This completes the proof.
\end{proof}

\begin{Def}
	We say two 1-cubes $f,g$ in a comical set $X$ are \emph{homotopic} and write $f \sim g$ if any one of the above marked 2-cubes exists in $X$.
\end{Def}

\begin{prop}\label{homotopy-is-eq-rel}
	For any pair of $0$-cubes $x,y$ in a comical set $X$, the homotopy relation is an equivalence relation on the set of all 1-cubes $x \to y$.
\end{prop}
\begin{proof}
	The reflexivity of $\sim$ is obvious, and its symmetry follows from \cref{homotopy-relation}.
	For transitivity, suppose we are given two homotopies:
	\[
	\begin{aligned}
	\readlist\eqone{x,y,y,y,->,f,double,,->,g,double,,\phi}
	\tikzsquare{\eqone} & & 
	\readlist\eqtwo{x,y,y,y,->,g,double,,->,h,double,,\chi}
	\tikzsquare{\eqtwo}
	\end{aligned}
	\]
	Then the following picture specifies a map $\trunc 1 \obox 3 2 0 \to X$:
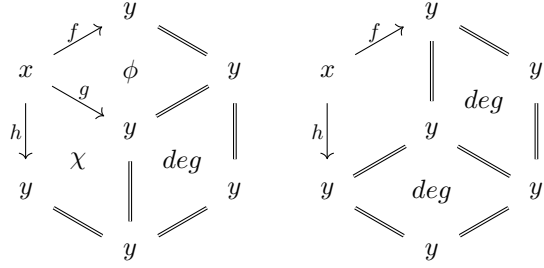
	\[
	\readlist\eqthree{x,y,y,y,y,y,y,y,->,f,double,,double,,->,h,double,,double,,->,g,double,,double,,double,,double,,double,,\phi,deg,\chi,deg,,deg,}
	\tikzcube{\eqthree}
	\]
	This map extends to $\admcubepp 3 2 0$, which in particular yields a homotopy $f \sim h$.
\end{proof}

Now consider a ``composable'' pair of 1-cubes $f\colon x \to y$ and $g \colon y \to z$ in a comical set $X$.
We may ``compose'' $f$ and $g$ by filling any one of the open boxes $\obox 2 1 0$, $\obox 2 1 1$, $\obox 2 2 0$, $\obox 2 2 1$ as follows:
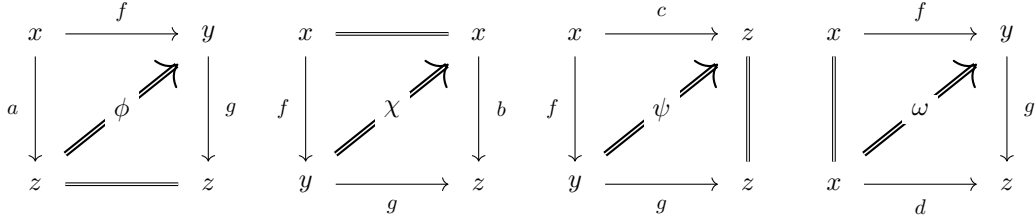
\[
\begin{aligned}
\readlist\compone{x,y,z,z,->,f,->,g,->,a,double,,\phi}
\tikzsquare{\compone} & &
\readlist\comptwo{x,x,y,z,double,,->,b,->,f,->,g,\chi}
\tikzsquare{\comptwo} & & 
\readlist\compthree{x,z,y,z,->,c,double,,->,f,->,g,\psi}
\tikzsquare{\compthree} & & 
\readlist\compfour{x,y,x,z,->,f,->,g,double,,->,d,\omega}
\tikzsquare{\compfour}
\end{aligned}
\]
We will temporarily call such $a$ a \emph{$(1,0)$-composite} of $f$ and $g$, and similarly call $b$, $c$ and $d$ $(1,1)$-, $(2,0)$- and $(2,1)$-composites of $f$ and $g$ respectively.

\begin{prop}\label{composite-is-unique}
	Any two composites of $f$ and $g$ are homotopic to each other.
\end{prop}

\begin{proof}
	First, consider two $(1,0)$-composites $a$ and $a'$, witnessed by $2$-cubes $\phi$ and $\phi'$ respectively.
	Then the following picture specifies a map $\trunc 1 \obox 3 1 0 \to X$:
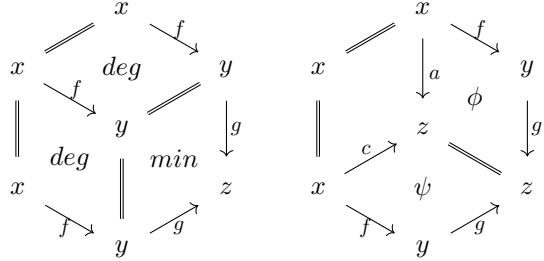
	\[
	\readlist\compositezero{x,y,z,z,x,z,z,y,->,f,->,g,double,,double,,->,a',double,,->,a,double,,double,,double,,->,f,->,g,\phi,deg,,deg,deg,\phi',}
	\tikzcube{\compositezero}
	\]
	This map extends to $\admcubepp 3 2 0$, which yields a homotopy $a \sim a'$.
	Similarly, for any given $a,b,c,d$ as above, a homotopy $a \sim c$ can be obtained by filling the following open box:
	\[
	\readlist\compositeone{x,x,y,z,x,y,y,z,double,,->,f,->,g,double,,->,f,->,g,->,f,double,,double,,->,a,->,c,double,,deg,min,deg,\phi,,\psi,}
	\tikzcube{\compositeone}
	\]
	and a homotopy $b \sim d$ can be obtained using:
	\[
	\readlist\compositetwo{x,x,x,z,x,x,x,y,double,,double,,->,b,double,,double,,->,d,double,,double,,double,,->,f,->,f,->,g,deg,,deg,\chi,min,\omega,}
	\tikzcube{\compositetwo}
	\]
	Finally, we can turn the $(2,1)$-composite $d$ into a $(1,0)$-composite using the following open box:
	\[
	\readlist\compositethree{x,x,y,z,x,z,x,x,double,,->,f,->,g,double,,->,d,double,,double,,->,f,->,d,double,,double,,->,d,min,,min,\omega,deg,deg,}
	\tikzcube{\compositethree}
	\]
	This completes the proof.
\end{proof}

\begin{prop}\label{homotopy-is-congruence}
	Let $f,f'\colon x \to y$ and $g,g'\colon y \to z$ be 1-cubes in a comical set $X$ such that $f \sim f'$ and $g \sim g'$.
	Then any composite of $f$ and $g$ is homotopic to any composite of $f'$ and $g'$.
\end{prop}
\begin{proof}
	Choose witnesses of the following forms for compositions and a homotopy $f \sim f'$:
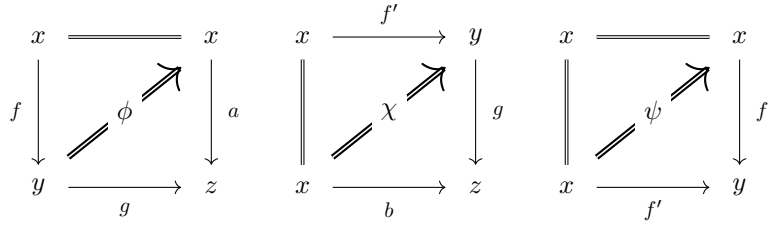
	\[
	\begin{aligned}
	\readlist\wellone{x,x,y,z,double,,->,a,->,f,->,g,\phi}
	\tikzsquare{\wellone} & &
	\readlist\welltwo{x,y,x,z,->,f',->,g,double,,->,b,\chi}
	\tikzsquare{\welltwo} & &
	\readlist\wellthree{x,x,x,y,double,,->,f,double,,->,f',\psi}
	\tikzsquare{\wellthree}
	\end{aligned}
	\]
	Then extending the following map $\trunc 1 \obox 3 2 1 \to X$ to $\admcubepp 3 2 1$ yields a homotopy $a \sim b$:
	\[
	\readlist\wellfour{x,x,x,z,x,x,x,y,double,,double,,->,a,double,,double,,->,b,double,,double,,double,,->,f,->,f',->,g,deg,,deg,\phi,\psi,\chi,}
	\tikzcube{\wellfour}
	\]
	
	Similarly, we may combine marked 2-cubes of the forms
	\[\begin{aligned}
	\readlist\wellfive{x,z,y,z,->,c,double,,->,f',->,g,\phi'}
	\tikzsquare{\wellfive} & &
	\readlist\wellsix{x,y,x,z,->,f',->,g',->,d,double,,\chi'}
	\tikzsquare{\wellsix} & &
	\readlist\wellseven{y,z,z,z,->,g,double,,->,g',double,,\psi'}
	\tikzsquare{\wellseven}
	\end{aligned}\]
	into a map $\trunc 1 \obox 3 2 0 \to X$ as follows:
	\[
	\readlist\welleight{x,z,z,z,z,z,y,z,->,c,double,,double,,->,d,double,,double,,->,f',->,g,->,g',double,,double,,double,,\phi',\psi',\chi',deg,,deg,}
	\tikzcube{\welleight}
	\]
	Extending this map to $\admcubepp 3 2 0$ yields a homotopy $c \sim d$.
	The desired result now follows by \cref{homotopy-is-eq-rel,composite-is-unique}.
\end{proof}

\begin{Def}
	We define the \emph{homotopy 1-category} $ho_1 X$ of a comical set $X$ to be the category of 0-cubes and homotopy classes of 1-cubes in $X$.
\end{Def}

\begin{prop}\label{homotopy-category-exists}
	For a comical set $X$, $ho_1 X$ is indeed a 1-category.
\end{prop}

\begin{proof}
	\cref{homotopy-is-congruence} implies that we have a well-defined composition operation on $ho_1 X$.
	For any 0-cube $x$ in $X$, we claim that the homotopy class containing $x \sigma_1$ is the identity at $x$.
	Indeed for any $f \colon x \to y$, the degenerate 2-cube
	\[
	\readlist\unital{x,x,y,y,double,x \sigma_1,->,f,->,f,double,y \sigma_1,f   \sigma_1}
	\tikzsquare{\unital}
	\]
	exhibits $f$ as a $(1,0)$-composite of $x \sigma_1$ and $f$, and also as a $(1,1)$-composite of $f$ and $y \sigma_1$.
	
	For associativity, suppose we are given 1-cubes $f\colon x \to y$, $g\colon y \to z$ and $h\colon z \to w$ in $X$.
	Compose these 1-cubes as follows:
	\[
	\readlist\assocone{x,y,x,z,->,f,->,g,double,,->,a,\phi}
	\readlist\assoctwo{x,z,x,w,->,a,->,h,double,,->,b,\chi}
	\readlist\assocthree{y,z,y,w,->,g,->,h,double,,->,c,\psi}
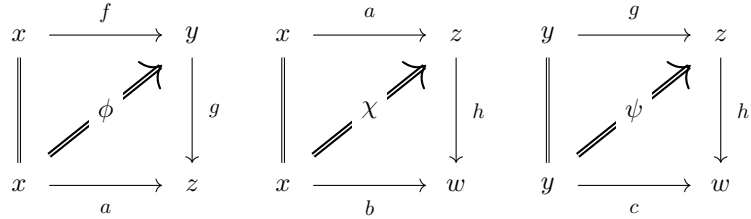
	\begin{aligned}
	\tikzsquare{\assocone} & & \tikzsquare{\assoctwo} & & \tikzsquare{\assocthree}
	\end{aligned}
	\]
	Then we may combine them into a map $\trunc 1 \obox 3 3 1 \to X$:
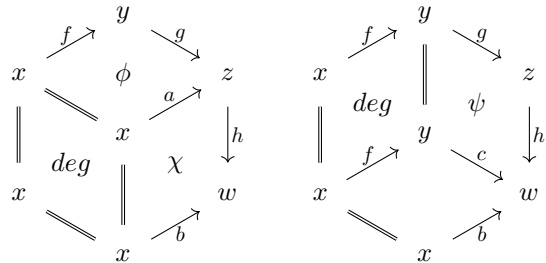
	\[
	\readlist\assocfour{x,y,z,w,x,x,x,y,->,f,->,g,->,h,double,,double,,->,b,double,,->,a,double,,double,,->,f,->,c,\phi,\chi,deg,\psi,deg,,}
	\tikzcube{\assocfour}
	\]
	Extending this map to $\admcubepp 3 3 1$ then yields a marked 2-cube that witnesses the desired associativity.
\end{proof}

The following proposition is straightforward to verify.

\begin{prop}
	The assignment $X \mapsto ho_1 X$ extends to a functor from the category of comical sets to $\Cat$.
	Moreover there is a natural isomorphism $ho_1(X^\op) \cong (ho_1 X)^\op$. \qed
\end{prop}


\section{Triangulation}\label{triangulation}
In this section, we upgrade the triangulation adjunction described in \cref{subsec:cubical} to a marked version.
We start by recalling the basic combinatorics of simplicial cubes, which can be found in \cite[\textsection 5]{verity:weak-complicial-2}.
(Note however that our indexing is reversed from Verity's.)

Given an $r$-simplex $\phi$ in the simplicial set $(\Delta^1)^n$, we can define a function
\[
\{1,\dots,n\} \to \{1,\dots,r,\pm\infty\}
\]
by declaring
\[
i \mapsto \left\{\begin{array}{cl}
+\infty, & \pi_i \circ \phi(r) = 0,\\
p, & \pi_i \circ \phi(p-1) = 0~\text{and}~\pi_i \circ \phi(p) = 1,\\
-\infty, & \pi_i \circ \phi(0) = 1.
\end{array}\right.
\]
If we regard $\phi$ as an $r$-step walk on the $n$-cube with the $p$-th step connecting $\phi(p-1)$ to $\phi(p)$, the above function takes $i$ to the unique $p$ such that the $p$-th step moves in the $i$-th direction; it takes the value $+\infty$ if we never move in the $i$-th direction, and the value $-\infty$ if we have already moved in that direction before we start.

Conversely, any function $\{1,\dots,n\} \to \{1,\dots,r,\pm\infty\}$ determines a unique $r$-simplex in $(\Delta^1)^n$, so we will identify the $r$-simplices and these functions.

\begin{rmk}
	In what follows, we sometimes write such expressions as $p \pm k$ for $p \in \{1,\dots,r,\pm\infty\}$ and finite $k$.
	These expressions are to be interpreted as $p$ when $p \in \{\pm\infty\}$.
	We will never consider expressions involving more than one $\pm\infty$.
\end{rmk}

\begin{Def}
	We will write $\iota=\iota_n \colon  \{1,\dots,n\} \to \{1,\dots,n,\pm\infty\}$ for the inclusion regarded as an $n$-simplex in $(\Delta^1)^n$.
\end{Def}

We will think of any set of the form $\{1,\dots,r,\pm\infty\}$ as a linearly ordered set
\[
-\infty < 1 < \dots < r < +\infty.
\]
Note however that simplices $\phi \colon  \{1,\dots,n\} \to \{1,\dots,r,\pm\infty\}$ are not necessarily order-preserving.

\begin{prop}\label{simplicial-action}
	Under this identification, a simplicial operator $\alpha \colon  [q] \to [r]$ sends an $r$-simplex $\phi$ to the $q$-simplex $\phi   \alpha$ given by
	\[
	(\phi   \alpha)(i) = \left\{\begin{array}{cl}
	+\infty, & \phi(i) > \alpha(q),\\
	p, & \alpha(p-1) < \phi(i) \le \alpha(p),\\
	-\infty, & \phi(i) \le \alpha(0).
	\end{array}\right.
	\]
\end{prop}

\begin{eg}\label{partition}
	Again, let us think of an $r$-simplex $\phi$ as an $r$-step walk.
	Then the last face $\phi\faces r$ of $\phi$ moves in the $i$-th direction at exactly the same step as $\phi$ does except that it does not have an $r$-th step.
	This agrees with the following formula obtained using \cref{simplicial-action}:
	\[
	(\phi\faces r)(i) = \left\{\begin{array}{cl}
		+\infty, & \phi(i) = r,\\
		\phi(i), & \text{otherwise}.
	\end{array}\right.
	\]
	On the other hand, taking the $0$-th face decreases the index of each step by $1$, so we have
	\[
	(\phi\faces 0)(i) = \left\{\begin{array}{cl}
		-\infty, & \phi(i) = 1,\\
		\phi(i)-1, & \text{otherwise.}
	\end{array}\right.
	\]
	For $0 < k < n$, taking the $k$-th step merges the $k$-th and the $(k+1)$-st steps, so it does not affect the endpoints of the whole walk.
	When regarded as a function, this means that $\phi\faces k$ takes the values $\pm\infty$ on exactly the same inputs as $\phi$ does.
	However, some of the indices are shifted:
	\[
	(\phi\faces k)(i) = 
	\left\{\begin{array}{cl}
		\phi(i)-1, & k < \phi(i) \le r,\\
		\phi(i), & \text{otherwise.}
	\end{array}\right.
	\]
	
	For any $0 \le m \le r$, taking the last face $(r-m)$ times yields $\frontface{m}{r-m}$, so we have:
	\[
	(\phi   \frontface{m}{r-m})(i) = \left\{\begin{array}{cl}
	+\infty, & \phi(i) > m,\\
	\phi(i), & \phi(i) \le m
	\end{array}\right.
	\]
	Similarly, since taking the $0$-th face $m$ times yields $\backface{m}{r-m}$, we have
	\[
	(\phi   \backface{m}{r-m})(i) = \left\{\begin{array}{cl}
	\phi(i)-m, & \phi(i) > m,\\
	-\infty, & \phi(i) \le m.
	\end{array}\right.
	\]
\end{eg}

It is easy to verify the following proposition using \cref{simplicial-action}.

\begin{prop}\label{non-degenerate}
	An $r$-simplex $\phi$ in $(\Delta^1)^n$ is non-degenerate if and only if $\phi^{-1}(p) \neq \varnothing$ for each $1 \le p \le r$.
	More precisely, $\phi$ is degenerate at $p-1$ if and only if $\phi^{-1}(p) = \varnothing$.
\end{prop}

Now we upgrade the \emph{codomain} of the triangulation functor to a marked version.
More precisely, we first consider the functor $\Cube \to \PreComp$ associated (in the sense of \cref{GrandisMauriMonoid}) to the cubical monoid $\Delta^1$, where $\PreComp$ is considered to be monoidal with respect to the Gray tensor product $\pretensor$.
Its object part is thus given by $[1]^n \mapsto (\Delta^1)^{\pretensor n}$.
This functor induces a strong monoidal left adjoint $T \colon \cSet \to \PreComp$ with right adjoint $U$.
We first show that $(\Delta^1)^{\pretensor n} = (\Delta^1)^{\otimes n}$.


\begin{prop}\label{tensor-power}
	An $r$-simplex $\phi \colon  \{1,\dots,n\} \to \{1,\dots,r,\pm\infty\}$ in $(\Delta^1)^{\otimes n}$ is unmarked if and only if there exist
	\[
	1 \le i_1 < \dots < i_r \le n
	\]
	such that $\phi(i_p) = p$ for all $1 \le p \le r$.
	In particular, the only unmarked $n$-simplex in $(\Delta^1)^{\otimes n}$ is $\iota_n$.
\end{prop}
\begin{proof}
	This is proved in \cite[Obs.~27]{verity:weak-complicial-2} (with opposite indexing from ours).
\end{proof}

Using the above characterisation, we can indeed prove the following.

\begin{prop}
	The marked simplicial set $(\Delta^1)^{\otimes n}$ is pre-complicial for any $n \ge 0$.
\end{prop}
\begin{proof}
	Suppose for contradiction that there exists a map $\admdeltap r k \to (\Delta^1)^{\otimes n}$ that cannot be extended to $\admdeltapp r k$.
	Regard this map as an $r$-simplex $\phi \colon \{1,\dots,n\} \to \{1,\dots,r,\pm\infty\}$.
	Then $\phi\faces k$ is unmarked, so \cref{tensor-power} implies that there exist
	\[
	1 \le i_1 \le \dots \le i_{r-1} \le n
	\]
	such that:
	\begin{itemize}
		\item $\phi(i_p) = p$ for $1 \le p \le k-1$;
		\item $\phi(i_k) = k$ or $\phi(i_k) = k+1$; and
		\item $\phi(i_p) = p+1$ for $k+1 \le p \le r-1$.
	\end{itemize}
	But if $\phi(i_k) = k$ then the same sequence $i_1,\dots,i_{r-1}$ witnesses that $\phi\faces{k+1}$ is unmarked, and similarly if $\phi(i_k) = k+1$ then $\phi\faces{k-1}$ is unmarked.
	In either case, it contradicts with our assumption that $\phi$ corresponds to a map $\admdeltap r k \to (\Delta^1)^{\otimes n}$.
\end{proof}

Next we would like to upgrade the \emph{domain} of $T$ to a marked version too, by sending the marked $n$-cube to $\trunc{n-1}\bigl((\Delta^1)^{\otimes n}\bigr)$.
\cref{tensor-power} implies that this marked simplicial set may be obtained from $(\Delta^1)^{\otimes n}$ by marking $\iota_n$.
The following lemma shows that it is indeed an object in $\PreComp$.

\begin{lem}
	The marked simplicial set $\trunc{n-1} \bigl((\Delta^1)^{\otimes n}\bigr)$ is pre-complicial for any $n \ge 1$.
\end{lem}
\begin{proof}
	Suppose for contradiction that we are given a map $\phi \colon  \admdeltap{m}{k} \to \trunc{n-1}\bigl((\Delta^1)^{\otimes n}\bigr)$ that cannot be extended to $\admdeltapp{m}{k}$.
	Then $\phi$ must not factor through the pre-complicial set $(\Delta^1)^{\otimes n}$, so $\phi$ sends at least one of the marked, non-degenerate simplices in $\admdeltap{m}{k}$ to $\iota_n$.
	Since all simplices in $\admdeltap{m}{k}$ of dimension $>m$ are degenerate, it follows that $m \ge n$.
	On the other hand, we cannot have $m>n$ since $\phi   \faces k$ is unmarked in the $(n-1)$-trivial marked simplicial set $\trunc{n-1}\bigl((\Delta^1)^{\otimes n}\bigr)$.
	Thus we must have $m = n$ and $\phi = \iota_n$.
	But at least one of $\faces{k-1}$ and $\faces{k+1}$ is a well-defined face of $\admdeltap{m}{k}$, and it can be easily checked using \cref{tensor-power} that $\phi$ sends this face to an unmarked simplex.
	This is the desired contradiction.
\end{proof}

Thus we have defined the object part of $T:\Cube^+ \to \PreComp$, but we still need to define its value on the generating morphisms $\varphi^n$, $\zeta^n_i$, and $\xi^n_{i, \varepsilon}$, and verify the co-cubical identities.
The maps $T \varphi^n \colon T [1]^n \to T [1]^n_e$ are identity on the underlying simplicial sets and add the additional marking on $\iota_n$.
To define $T \zeta^n_i$ (resp.~$T \xi^n_{i, \varepsilon}$), notice that we must have $T\zeta^n_i T\varphi^n = T\sigma^n_i$ (resp.~$T\xi^n_{i\varepsilon} T\varphi^n = T\gamma^n_{i, \varepsilon}$).
Since $T\sigma^n_i$ (resp.~$T\gamma^n_{i, \varepsilon}$) sends the $n$-simplex $\iota_n$ to a degenerate (and hence marked) one, it must factor through $T\varphi^n$.
Moreover, since $T\varphi^n$ is (entire and hence) an epimorphism, this factorisation is unique, yielding a unique possible choice for $T \zeta^n_i$ (resp.~$T \xi^n_{i, \varepsilon}$).
Finally, to see that this definition satisfies the additional identities, we note that these involving $\varphi$ are clear, whereas the remaining ones can be reduced to the usual cubical identities by precomposing with $\varphi$ and using the fact that it is an epimorphism.

Hence we obtain a left adjoint functor $T$ from structurally marked cubical sets to pre-complicial sets. Moreover, the right adjoint $U$ takes values in marked cubical sets, because the map $\cube n \to \mcube n$ is carried by $T$ to an epimorphism.
Thus by restricting the domain of $T$, we have constructed an adjunction $T \dashv U$ between marked cubical sets and pre-complicial sets.
In the remainder of the paper, we show that $T$ is strong monoidal with respect to either version of the Gray tensor products and moreover left Quillen with respect to suitable model structures.
We will make use of the following observation.

\begin{prop}
	There are isomorphisms $T(X^\op) \cong (T X)^\op$ natural in $X \in \mcSet$.
\end{prop}
\begin{proof}
	Since both $X \mapsto T(X^\op)$ and $X \mapsto T(X)^\op$ are cocontinuous, it suffices to verify the assertion for $X = \cube n$ for $n \ge 0$ and $X = \mcube n$ for $n \ge 1$.
	The component at each $\cube n$ is simply given by
	\[
	T\bigl((\cube n)^\op\bigr) = T(\cube n) = (\Delta^1)^{\otimes n} = \bigl((\Delta^1)^\op\bigr)^{\otimes n} \cong \bigl((\Delta^1)^{\otimes n}\bigr)^\op = T(\cube n)^\op
	\]
	(where the isomorphism is induced by the anti-monoidality of $(-)^\op \colon \mSet \to \mSet$ \cite[Lemma 131]{verity:complicial}) and the component at each $\mcube n$ is then obtained by applying $\trunc{n-1}$.
	
	It remains to check that these components are natural.
	Since the forgetful functor $\mSet \to \sSet$ is faithful, we may instead check the naturality of the whiskering:
	\[
	\begin{tikzpicture}
		\node at (0,0) {$\cube +$};
		\node at (3,0) {$\mSet$};
		\node at (5,0) {$\sSet$};
		\draw[->] (0.2,0.2) .. controls (0.7,0.7) and (2,0.7) .. (2.5, 0.2);
		\draw[->] (0.2,-0.2) .. controls (0.7,-0.7) and (2,-0.7) .. (2.5, -0.2);
		\node[scale = 0.8] at (1.35,0.9) {$T\bigl((-)^\op\bigr)$};
		\node[scale = 0.8] at (1.35,-0.9) {$T(-)^\op$};
		\draw[->, double] (1.35, 0.4) -- (1.35, -0.4);
		\draw[->] (3.5,0) -- (4.5,0);
	\end{tikzpicture}
	\]
	Now it is straightforward to check that this (potentially un-natural) transformation may also be written as:
	\[
	\begin{tikzpicture}
		\node at (0,0) {$\cube +$};
		\node at (2,0) {$\cube {}$};
		\node at (5,0) {$\sSet$};
		\draw[->] (0.3,0) -- (1.7,0);
		\draw[->] (2.2,0.2) .. controls (2.7,0.7) and (4,0.7) .. (4.5, 0.2);
		\draw[->] (2.2,-0.2) .. controls (2.7,-0.7) and (4,-0.7) .. (4.5, -0.2);
		\node[scale = 0.8] at (3.35,0.9) {$T\bigl((-)^\op\bigr)$};
		\node[scale = 0.8] at (3.35,-0.9) {$T(-)^\op$};
		\draw[->, double] (3.35, 0.4) -- (3.35, -0.4);
	\end{tikzpicture}
	\]
	where the first factor forgets the marking and the second factor (potentially un-natural transformation) is the unmarked analogue of our desired natural isomorphism.
	But in the unmarked case, we know that both $X \mapsto T(X^\op)$ and $X \mapsto T(X)^\op$ are anti-monoidal, and moreover it is straightforward to manually check that the naturality of its restriction to the full subcategory spanned by $\cube 0$, $\cube 1$ and $\cube 2$.
	Thus the desired naturality follows from \cref{GrandisMauriMonoid}.
\end{proof}

\section{Triangulating Gray tensor product} \label{sec:triangulating-gray}

We now prove that the triangulation functor is strong monoidal with respect to either version of the Gray tensor product.
We begin by describing a proof strategy that will be used in both the lax and the pseudo cases.

\subsection{Proof strategy}\label{strategy}

The proof typically reduces to showing an entire inclusion $A \incl B$ to be a complicial marking extension where $A,B$ are certain entire supersets of $(\Delta^1)^{\otimes N}$.
(The integer $N$ will be of the form $N = m+n$ in the actual proofs, but this is irrelevant in this subsection.)
There are three kinds of simplices of interest, namely those that are:
\begingroup
\renewcommand{\theenumi}{\roman{enumi}}
\begin{enumerate}
	\item marked in $(\Delta^1)^{\otimes N}$;
	\item marked in $A$ but not in $(\Delta^1)^{\otimes N}$; and
	\item marked in $B$ but not in $A$.
\end{enumerate}
\endgroup
The simplices of type (i) are characterized by \cref{tensor-power}.
The first step of the proof will be to (define suitable $A,B$ and) characterize simplices of type (ii) and (iii).

Before proceeding, we need the following definitions.
\begin{Def}
	For any $r$-simplex $\phi : \{1,\dots,N\} \to \{1,\dots,r,\pm\infty\}$ in $(\Delta^1)^N$, define:
	\begin{align*}
	\calD(\phi) &= \bigl|\phi^{-1}(\{1,\dots,r\})\bigr|-r,\\
	\calO(\phi) &= \bigl\{(i,j) \in \{1,\dots,N\}^2 : i<j, \phi(i)<\phi(j)\bigr\}.
	\end{align*}
\end{Def}
The integer $\calD(\phi)$ measures how ``diagonal'' $\phi$ is, and the set $\calO(\phi)$ measures how ``in order'' $\phi$ is.

We complicially extend the marking on $A$ to those simplices $\phi$ of type (iii) by nested induction on $\calD(\phi)$ and $|\calO(\phi)|$.
More precisely, consider the lexicographical ordering on $\mathbb{Z}\times\mathbb{N}$ so that $(u_1,v_1) \le (u_2,v_2)$ if and only if:
\begin{itemize}
	\item $u_1 < u_2$; or
	\item $u_1=u_2$ and $v_1 \le v_2$.
\end{itemize}
For each $(u,v) \in \mathbb{Z}\times\mathbb{N}$, let $A(u,v)$ denote the marked simplicial set obtained from $A$ by marking those simplices $\phi$ such that $\phi$ is marked in $B$ and $\bigl(\calD(\phi),|\calO(\phi)|\bigr) < (u,v)$.
Then:
\begin{itemize}
	\item $A(u_1,v_1)$ is an entire subset of $A(u_2,v_2)$ for any $(u_1,v_1) \le (u_2,v_2)$;
	\item $\colim_vA(u,v) = A(u+1,0)$ for any $u \in \mathbb{Z}$;
	\item $\colim_{u,v}A(u,v) = B$; and
	\item $A(0,0) = A$ (by \cref{non-degenerate}).
\end{itemize}
Now we assume the following.
\begin{assume}\label{diagonality-zero}
	Any marked simplex $\phi$ in $B$ with $\calD(\phi) = 0$ is marked in $A$.
\end{assume}
Then we may upgrade the last bulleted item to $A(1,0) = A$.
Thus to prove that $A \to B$ is a complicial marking extension, it suffices to exhibit the map $A(u,v) \to A(u,v+1)$ as a complicial marking extension for each $(u,v) \ge (1,0)$.

So fix $(u,v) \ge (1,0)$ and suppose that we are given an $r$-simplex $\phi$ of type (iii) with $\bigl(\calD(\phi),|\calO(\phi)|\bigr) = (u,v)$.
Then in particular $\calD(\phi) \ge 1$.
So by the pigeon hole principle, we can choose $1 \le p_\phi \le r$ such that $|\phi^{-1}(p_\phi)| \ge 2$.
Let $i_\phi = \min \phi^{-1}(p_\phi)$.
Let $\tilde \phi$ be the $(r+1)$-simplex given by
\[
\tilde \phi(i) = \left\{\begin{array}{cl}
\phi(i), & \phi(i) \le p_\phi\;\text{and}\;i \neq i_\phi,\\
\phi(i)+1, & \phi(i)>p_\phi\;\text{or}\;i = i_\phi.
\end{array}\right.
\]
Observe that we have $\tilde \phi   \faces {p_\phi} = \phi$.
We wish to show that this simplex $\tilde \phi$ extends to $\admdeltap{r+1}{p_\phi}$:
\[
\begin{tikzcd}
{\Delta^{r+1}}
\arrow [r, "\tilde \phi"]
\arrow [d, hook] &
A(u,v) \\
{\admdeltap{r+1}{p_\phi}}
\arrow [ru, dashed, "\exists", swap] &
\end{tikzcd}
\]
Assuming this fact, we can deduce that we have a pushout square
\[
\begin{tikzcd}[row sep = large]
\coprod \admdeltap{r+1}{p_\phi}
\arrow [r]
\arrow [d, hook] &
A(u,v)
\arrow [d] \\
\coprod \admdeltapp{r+1}{p_\phi}
\arrow [r] &
A(u,v+1)
\end{tikzcd}
\]
where the coproducts are taken over all $r$-simplices $\phi$ of type (iii) with $\bigl(\calD(\phi),|\calO(\phi)|\bigr) = (u,v)$ (for various $r$) and the horizontal maps are induced by $\tilde \phi$.

The following lemma implies that $\tilde \phi$ at least extends to $\admdelta{r+1}{p_\phi}$.
\begin{lem}
	Let $\alpha \colon  [q] \to [r+1]$ be a face operator with $\{p_\phi,p_\phi\pm1\} \subset \im\alpha$.
	Then $\tilde \phi   \alpha$ is marked in $(\Delta^1)^{\otimes N}$.
\end{lem}
\begin{proof}
	Let $p \in [q]$ be the necessarily unique element with  $\alpha(p) = p_\phi$.
	Then we must have $\alpha(p-1) = p_\phi-1$ and $\alpha(p+1) = p_\phi+1$.
	Now one can check using \cref{simplicial-action} and the minimality of $i_\phi$ that $(\tilde \phi   \alpha)^{-1}(p+1) = \{i_\phi\}$ and moreover any $1 \le i \le N$ satisfying $(\tilde \phi   \alpha)(i) = p$ must also satisfy $i > i_\phi$.
	Thus $\tilde \phi   \alpha$ is marked in $(\Delta^1)^{\otimes N}$ by \cref{tensor-power}.
\end{proof}

Therefore it remains to prove that the faces $\chi = \tilde \phi   \faces{p_\phi-1}$ and $\psi = \tilde \phi   \faces{p_\phi+1}$ are marked in $A(u,v)$.
First, we describe these simplices explicitly.

\begin{lem}\label{neighbours-description}
	The simplex $\chi$ is given by
	\[
	\chi(i) = \left\{\begin{array}{cl}
	-\infty, & \phi(i) = p_\phi\;\text{and}\;i \neq i_\phi,\\
	\phi(i), & \text{otherwise}.
	\end{array}\right.
	\]
	if $p_\phi = 1$ and
	\[
	\chi(i) = \left\{\begin{array}{cl}
	p_\phi-1, & \phi(i) = p_\phi\;\text{and}\;i \neq i_\phi,\\
	\phi(i), & \text{otherwise}
	\end{array}\right.
	\]
	if $p_\phi \ge 2$.
	The simplex $\psi$ is given by
	\[
	\psi(i) = \left\{\begin{array}{cl}
	+\infty, & i = i_\phi\\
	\phi(i), & \text{otherwise}
	\end{array}\right.
	\]
	if $p_\phi = r$ and
	\[
	\psi(i) = \left\{\begin{array}{cl}
	p_\phi+1, & i = i_\phi\\
	\phi(i), & \text{otherwise}
	\end{array}\right.
	\]
	if $p_\phi < r$.
\end{lem}
\begin{proof}
	This is a routine application of \cref{simplicial-action}.
\end{proof}

These explicit descriptions allow us to prove the following.

\begin{lem}
	The simplices $\chi$ and $\psi$ satisfy
	\begin{gather*}
	\bigl(\calD(\chi),|\calO(\chi)|\bigr) < \bigl(\calD(\phi),|\calO(\phi)|\bigr),\\
	\bigl(\calD(\psi),|\calO(\psi)|\bigr) < \bigl(\calD(\phi),|\calO(\phi)|\bigr).
	\end{gather*}
\end{lem}
\begin{proof}
	If $p_\phi=1$ then clearly $\calD(\chi) < \calD(\phi)$.
	
	Suppose $p_\phi \ge 2$.
	Then we have $\calD(\chi) =\calD(\phi)$.
	We claim that $\calO(\chi)$ is a proper subset of $\calO(\phi)$.
	Indeed, it can be seen from \cref{neighbours-description} that if a pair $(i,j)$ satisfies $\phi(i) \ge \phi(j)$ and $\chi(i) < \chi(j)$ then we must have $\phi(i)=\phi(j)=p_\phi$ and $i \neq i_\phi = j$.
	But then the minimality of $i_\phi$ implies $i > j$, and this shows that there is no pair $(i,j)$ in $\calO(\chi) \setminus \calO(\phi)$.
	On the other hand, since $\phi$ is unmarked in $(\Delta^1)^{\otimes N}$, there exist $1 \le i_1 \le \dots \le i_r \le N$ such that $\phi(i_p) = p$ for $1 \le p \le r$.
	It is straightforward to check that the pair $\bigl(i_{p_\phi-1},\max\phi^{-1}(p_\phi)\bigr)$ is then in $\calO(\phi) \setminus \calO(\chi)$.
	Therefore $\calO(\chi)$ is a proper subset of $\calO(\phi)$, and this proves the lexicographical inequality concerning $\chi$.
	
	The simplex $\psi$ can be treated dually.
\end{proof}

The last missing piece of the proof (that $A \to B$ is a complicial marking extension) is the following.

\begin{assume}\label{neighbours-are-marked}
	The simplices $\chi$ and $\psi$ are marked in $B$.
\end{assume}

This completes the proof strategy.
(Whether \cref{diagonality-zero,neighbours-are-marked} hold depends on the exact definitions of $A$ and $B$, so there is no general strategy for verifying them.)

\subsection{Triangulating the lax Gray tensor product}
The goal of this subsection is to prove the following theorem.
\begin{thm}\label{T-strong-monoidal-lax}
	The adjunction $T \dashv U$ is monoidal with respect to the lax Gray tensor products.
	Equivalently, $T \colon  (\mcSet,\otimes) \to (\PreComp,\pretensor)$ is strong monoidal.
\end{thm}

Fix $m\ge 1$ and $n \ge 0$.
Observe that
\[
\begin{tikzcd}[row sep = large]
\coprod \cube{m+k}
\arrow [r]
\arrow [d, swap, "\coprod \varphi"] &
\cube{m+n}
\arrow [d]\\
\coprod \mcube{m+k}
\arrow [r] &
\mcube m \otimes \cube n
\end{tikzcd}\]
is a pushout square in $\mcSet$ where the coproducts are taken over all face maps $[1]^k \to [1]^n$.
This pushout is preserved by $T$, so the right square in
\[
\begin{tikzcd}[row sep = large]
\coprod \Delta^{m+k}
\arrow [r, "\coprod \iota_{m+k}"]
\arrow [d] &
\coprod (\Delta^1)^{\otimes (m+k)}
\arrow [r]
\arrow [d] &
(\Delta^1)^{\otimes(m+n)}
\arrow [d]\\
\coprod \mDelta{m+k}
\arrow [r] &
\coprod \trunc{m+k-1}\bigl((\Delta^1)^{\otimes (m+k)}\bigr)
\arrow [r] &
T(\mcube m \otimes \cube n)
\end{tikzcd}\]
is a pushout square in $\PreComp$ where the upper horizontal map is induced by $\id_{(\Delta^1)^{\otimes m}} \otimes T(\phi)$ for various face maps $\phi \colon  [1]^k \to [1]^n$.
The left square is also a pushout by \cref{tensor-power}, so the pasted square is a pushout too.
In this subsection, we define $A$ to be the corresponding pushout in $\mSet$ (and not in $\PreComp$):
\[
\begin{tikzcd}[row sep = large]
\coprod \Delta^{m+k}
\arrow [r]
\arrow [d]
\arrow [dr, phantom, "\mathrm{p.o.}"] &
(\Delta^1)^{\otimes(m+n)}
\arrow [d]\\
\coprod \mDelta{m+k}
\arrow [r] &
A
\end{tikzcd}
\]
so that its pre-complicial reflection $A^\precomp$ is precisely $T(\mcube m \otimes \cube n)$.

Now we give combinatorial characterisations of marked simplices in $A$ and in $T(\mcube m)\otimes T(\cube n)$.

\begin{lem}\label{marked-in-A}
	An $r$-simplex $\phi$ is marked in $A$ but not in $(\Delta^1)^{\otimes(m+n)}$ if and only if $r \ge m$, $\phi(i) = i$ for all $1 \le i \le m$ and the restriction
	\[
	\phi^{-1}(\{1,\dots,r\}) \to \{1,\dots,r\}
	\]
	of $\phi$ is an isomorphism of linearly ordered sets.
\end{lem}
\begin{proof}
	Compute the colimit.
\end{proof}

\begin{lem}\label{marked-in-TT}
	Let $\phi$ be an unmarked $r$-simplex in $(\Delta^1)^{\otimes (m+n)}$.
	Then $\phi$ is marked in $T(\mcube m)\otimes T(\cube n)$ if and only if:
	\begin{enumerate}
		\item $r \ge m$;
		\item $\phi(i) = i$ for all $1 \le i \le m$; and
		\item there does NOT exist a sequence $m < i_m < i_{m+1} < \dots < i_r \le m+n$ such that $\phi(i_p) = p$ for all $m \le p \le r$.
	\end{enumerate}
\end{lem}
\begin{proof}
	Write $\phi = (\phi_1,\phi_2)$ for $\phi$ regarded as a simplex in the product simplicial set $(\Delta^1)^m \times (\Delta^1)^n$.
	That is, $\phi_1 \colon  \{1,\dots,m\} \to \{1,\dots,r,\pm\infty\}$ and $\phi_2 \colon  \{1,\dots,n\} \to \{1,\dots,r,\pm\infty\}$ are respectively given by $\phi_1(i) = \phi(i)$ and $\phi_2(i) = \phi(m+i)$.
	
	It follows from the definitions of $\otimes$ and $T(\mcube m)$ that $\phi$ (which we are assuming to be unmarked in $(\Delta^1)^{\otimes (m+n)}$) is marked in $T(\mcube m) \otimes T(\cube n)$ if and only if:
	\begin{itemize}
		\item[(a)] $r \ge m$;
		\item[(b)] $(\phi_1,\phi_2)$ is $q$-cloven for all $q$ except for $q = m$;
		\item[(c)] $\phi_1   \frontface{m}{r-m} = \iota_m$; and
		\item[(d)] $\phi_2   \backface{m}{r-m}$ is unmarked in $(\Delta^1)^{\otimes n}$.
	\end{itemize}
	The clauses (a) and (c) here clearly correspond respectively to (1) and (2) in the lemma.
	Since we are assuming $\phi$ to be unmarked in $(\Delta^1)^{\otimes (m+n)}$, \cref{tensor-power} implies that there exists a sequence $1 \le j_1 < \dots < j_r \le m+n$ such that $\phi(j_p) = p$ for all $1 \le p \le r$.
	Note that the strict inequalities imply $j_{m+1} > m$.
	One can now check using \cref{partition} that $(\phi_2   \backface{m}{r-m})(j_p-m) = p-m$ for all $m+1 \le p \le r$.
	Thus the clause (d) is in fact redundant by \cref{tensor-power}.
	
	It remains to check that, assuming (a), (c) and (d), the clauses (3) and (b) are equivalent.
	Note that $(\phi_1,\phi_2)$ is $q$-cloven for any $m < q \le r$ since the $q$-simplex $\phi_1   \frontface{q}{r-q}$ in the simplicial set $(\Delta^1)^m$ must be degenerate. 
	For $0 \le q <m$, since we are assuming (c), $\phi_1 \frontface{q}{r-q}=\iota_{q}$ is unmarked in $T(\mcube m)$.
	So (b) is equivalent to $\phi_2 \backface{q}{r-q}$ being marked in $(\Delta^1)^{\otimes n}$ for all $0 \le q < m$.
	By \cref{partition,tensor-power}, this latter condition for fixed $q$ is equivalent to the NON-existence of a sequence
	\[
	m<i_{q+1}<\dots<i_r\le m+n
	\]
	such that $\phi(i_p) = p$ for all $q+1 \le p \le r$.
	Clearly the non-existence for $q=m-1$, which is precisely (3), implies the non-existence for all other values of $q$.
	This completes the proof.
\end{proof}

By combining \cref{tensor-power} and \cref{marked-in-TT}, we obtain the following.
\begin{lem}\label{unmarked-in-TT}
	An $r$-simplex $\phi$ in $T(\mcube m) \otimes T(\cube n)$ with $r \ge m$ is unmarked if and only if there exist
	\[
	1 \le i_1 < \dots < i_r \le m+n
	\]
	such that $\phi(i_p) = p$ for all $1 \le p \le r$ and moreover $i_m > m$.
\end{lem}
\begin{proof}
	Let $\phi$ be an unmarked $r$-simplex in $(\Delta^1)^{\otimes (m+n)}$ with $r \ge m$.
	Note that $\phi$ is unmarked in $T(\mcube m) \otimes T(\cube n)$ if and only if it violates either \cref{marked-in-TT}(2) or (3).
	
	The ``if'' direction is easy since the existence of a sequence satisfying the condition stated in the lemma would immediately contradict \cref{marked-in-TT}(3).
	
	For the ``only if'' direction, assume that $\phi$ is unmarked in $T(\mcube m) \otimes T(\cube n)$.
	Recall that by \cref{tensor-power} there exist
	\[
	1 \le j_1 < \dots < j_r \le m+n
	\]
	such that $\phi(j_p) = p$ for all $1 \le p \le r$.
	If $j_m > m$, then simply taking $i_p = j_p$ for all $p$ would yield the desired sequence.
	So assume $j_m = m$.
	Then since the inequalities $j_1<\dots<j_m$ are strict, we must have $j_p = p$ for all $1 \le p \le m$.
	Thus $\phi$ cannot violate \cref{marked-in-TT}(2), so it must violate (3).
	That is, there exist $m < i_m < i_{m+1} < \dots < i_r \le m+n$ such that $\phi(i_p) = p$ for all $m \le p \le r$.
	We then obtain the desired sequence by taking $i_p = j_p$ for $1 \le p \le m-1$.
\end{proof}

\begin{lem}\label{A-to-TT}
	There is a complicial marking extension $A \to T(\mcube m)\otimes T(\cube n)$ that commutes with the evident inclusions of $(\Delta^1)^{\otimes(m+n)}$.
\end{lem}

\begin{proof}
	We apply the proof strategy from \cref{strategy} with $B = T(\mcube m)\otimes T(\cube n)$.
	
	One can easily check using \cref{marked-in-TT,marked-in-A} that any marked simplex $\phi$ in
	$T(\mcube m) \otimes T(\cube n)$ with $\calD(\phi) = 0$ must also be marked in $A$.
	This verifies \cref{diagonality-zero}.
	
	To verify \cref{neighbours-are-marked}, let $\phi$ be an $r$-simplex that is marked in $T(\mcube m) \otimes T(\cube n)$ but not in $A$.
	Then we necessarily have $r \ge m$ by \cref{marked-in-TT}.
	
	Consider the simplex $\chi = \tilde \phi   \faces{p_\phi-1}$.
	Suppose for contradiction that $\chi$ is unmarked in $T(\mcube m) \otimes T(\cube n)$.
	Then by \cref{unmarked-in-TT} there exist $1 \le i_1 < \dots < i_r \le m+n$ such that $\chi(i_p) = p$ for all $1 \le p \le r$ and $i_m > m$.
	\begin{itemize}
		\item If $p_\phi = 1$, then we also have $\phi(i_p) = p$ for all $p$ by \cref{neighbours-description}, thus $\phi$ is unmarked in $T(\mcube m) \otimes T(\cube n)$.
		This is the desired contradiction.
		\item Suppose $p_\phi \ge 2$.
		We claim that $\phi(i_p) = p$ holds for all $p$ in this case too.
		According to \cref{neighbours-description}, the only thing we must check is that $\phi(i_{p_\phi-1}) = p_\phi-1$ holds (as opposed to $\phi(i_{p_\phi-1}) = p_\phi$).
		To see that this is indeed the case, observe that $\chi^{-1}(p_\phi) = \{i_\phi\}$ by \cref{neighbours-description}.
		Thus we must have $i_{p_\phi} = i_\phi$.
		Since $i_{p_\phi-1} < i_{p_\phi}$, the minimality of $i_\phi$ implies that $i_{p_\phi-1} \notin \phi^{-1}(p_\phi)$.
		Therefore we have obtained the desired contradiction.
	\end{itemize}
	
	The simplex $\psi = \tilde \phi   \faces{p_\phi+1}$ can be similarly checked to be marked in $T(\mcube m) \otimes T(\cube n)$.
	This completes the proof.
\end{proof}

\begin{proof}[Proof of \cref{T-strong-monoidal-lax}]
	Since both $T(-\otimes-)$ and $T(-)\pretensor T(-)$ preserve colimits in each variable, it suffices to check the existence of natural isomorphisms $T(X \otimes Y)\cong T(X)\pretensor T(Y)$ for $X,Y$ generic (possibly marked) cubes.
	
	By construction of $T$, we have $T(\cube m \otimes \cube n) \cong T(\cube m) \pretensor T(\cube n)$ for any $m,n \ge 0$.
	
	For any $m \ge 1$ and $n \ge 0$, we may obtain an isomorphism $T(\mcube m \otimes \cube n) \cong T(\mcube m)\pretensor T(\cube n)$ by reflecting the complicial marking extension of \cref{A-to-TT} into $\PreComp$.
	Dually, we have $T(\cube m \otimes \mcube n) \cong T(\cube m) \pretensor T(\mcube n)$ for any $m \ge 0$ and $n \ge 1$.
	
	Let $m,n \ge 1$.
	Observe that the left square below is a pushout in $\mcSet$ by \cref{Gray-tensor-of-monos}(3):
	\[
	\begin{tikzcd}[row sep = large]
	\cube m \otimes \cube n
	\arrow [r]
	\arrow [d] &
	\mcube m \otimes \cube n
	\arrow [d] &
	T(\cube m \otimes \cube n)
	\arrow [r]
	\arrow [d] &
	T(\mcube m \otimes \cube n)
	\arrow [d] \\
	\cube m \otimes \mcube n
	\arrow [r] &
	\mcube m \otimes \mcube n &
	T(\cube m \otimes \mcube n)
	\arrow [r] &
	T(\mcube m \otimes \mcube n)
	\end{tikzcd}
	\]
	Since $T$ is cocontinuous, it follows that the right square is a pushout in $\PreComp$.
	On the other hand, since both $T(\cube m) \to T(\mcube m)$ and $T(\cube n) \to T(\mcube n)$ are entire, the square below is a pushout in $\PreComp$ by \cite[Lem.~140]{verity:complicial}:
	\[
	\begin{tikzcd}[row sep = large]
	T(\cube m) \pretensor T(\cube n)
	\arrow [r]
	\arrow [d] &
	T(\mcube m) \pretensor T(\cube n)
	\arrow [d] \\
	T(\cube m) \pretensor T(\mcube n)
	\arrow [r] &
	T(\mcube m) \pretensor T(\mcube n)
	\end{tikzcd}
	\]
	Thus by comparing the two pushout squares in $\PreComp$, we obtain $T(\mcube m \otimes \mcube n) \cong T(\mcube m) \pretensor T(\mcube n)$.
	The naturality of these isomorphisms is evident, and this completes the proof.
\end{proof}

\subsection{Triangulating the pseudo Gray tensor product}
The goal of this subsection is to prove the following theorem.
\begin{thm}\label{T-strong-monoidal-pseudo}
	The adjunction $T \dashv U$ is monoidal with respect to the pseudo Gray tensor products.
	Equivalently, $T \colon  (\mcSet,\pseudo) \to (\PreComp,\pseudo)$ is strong monoidal.
\end{thm}

Fix $m,n \ge 1$.
By \cref{comparing-Gray-tensors}, the square
\[
\begin{tikzcd}
\coprod
\cube {k+\ell}
\arrow [r]
\arrow [d] &
\cube {m+n}
\arrow [d]\\
\coprod \mcube{k+\ell}
\arrow [r] &
\cube m \pseudo \cube n
\end{tikzcd}
\]
is a pushout in $\mcSet$ where the coproducts are taken over all pairs of face maps $\cube k \to \cube m$ and $\cube \ell \to \cube n$ such that $k,\ell \ge 1$.
This pushout is preserved by $T$, so the right square in
\[
\begin{tikzcd}[row sep = large]
\coprod
\Delta^{k+\ell}
\arrow [r, "\coprod \iota_{k+\ell}"]
\arrow [d] &
\coprod
(\Delta^1)^{\otimes (k+\ell)}
\arrow [r]
\arrow [d] &
(\Delta^1)^{\otimes (m+n)}
\arrow [d]\\
\coprod
\mDelta{k+\ell}
\arrow [r] &
\coprod
\trunc{k+\ell-1}\bigl((\Delta^1)^{\otimes(k+\ell)}\bigr)
\arrow [r] &
T(\cube m \pseudo \cube n)
\end{tikzcd}
\]
is a pushout square in $\PreComp$.
The left square is also a pushout by \cref{tensor-power}, so the pasted square is a pushout too.
In this subsection, we define $A$ to be the corresponding pushout in $\mSet$ (and not in $\PreComp$):
\[
\begin{tikzcd}[row sep = large]
\coprod \Delta^{k+\ell}
\arrow [r]
\arrow [d]
\arrow [dr, phantom, "\mathrm{p.o.}"] &
(\Delta^1)^{\otimes(m+n)}
\arrow [d]\\
\coprod \mDelta{k+\ell}
\arrow [r] &
A
\end{tikzcd}
\]
so that its pre-complicial reflection $A^\precomp$ is precisely $T(\cube m \pseudo \cube n)$.

\begin{lem}\label{marked-in-Ap}
	An $r$-simplex $\phi$ is marked in $A$ but not in $(\Delta^1)^{\otimes(m+n)}$ if and only if the restriction
	\[
	\phi^{-1}(\{1,\dots,r\}) \to \{1,\dots,r\}
	\]
	of $\phi$ is an isomorphism of linearly ordered sets and moreover $\phi^{-1}(\{1,\dots,r\})$ intersects both $\{1,\dots,m\}$ and $\{m+1,\dots,m+n\}$.
\end{lem}
\begin{proof}
	Compute the colimit.
\end{proof}

\begin{lem}\label{unmarked-in-TpT}
	An $r$-simplex $\phi$ in $T(\cube m) \pseudo T(\cube n)$ is unmarked if and only if there exist either
	\[
	1 \le i_1 < \dots < i_r \le m
	\]
	or
	\[
	m+1 \le i_1 < \dots < i_r \le m+n
	\]
	such that $\phi(i_p) = p$ for all $1 \le p \le r$.
\end{lem}
\begin{proof}
	Since $\pseudo$ is the categorical product on $\PreComp$, $\phi$ is marked if and only if both $\pi_1(\phi)$ and $\pi_2(\phi)$ are marked.
	Equivalently, $\phi$ is unmarked if and only if either $\pi_1(\phi)$ or $\pi_2(\phi)$ is unmarked.
	Thus the assertion follows from \cref{tensor-power}.
\end{proof}

\begin{lem}\label{Ap-to-TpT}
	There is a complicial marking extension $A \to T(\cube m)\pseudo T(\cube n)$ that commutes with the evident inclusions of $(\Delta^1)^{\otimes(m+n)}$.
\end{lem}
\begin{proof}
	We apply the proof strategy from \cref{strategy} with $B = T(\cube m)\pseudo T(\cube n)$.
	
	One can easily check using \cref{tensor-power,marked-in-Ap,unmarked-in-TpT} that any marked simplex $\phi$ in $T(\cube m)\pseudo T(\cube n)$ with $\calD(\phi) = 0$ must also be marked in $A$.
	This verifies \cref{diagonality-zero}.
	
	To see that \cref{neighbours-are-marked} holds for $\chi$, suppose for contradiction that $\chi$ is unmarked in $T(\cube m)\pseudo T(\cube n)$.
	By \cref{unmarked-in-TpT}, this unmarked-ness is witnessed by a sequence $i_1,\dots,i_r$, but then the same sequence can be checked to witness that $\phi$ is unmarked in $T(\cube m)\pseudo T(\cube n)$.
	The details are similar to the corresponding part in the proof of \cref{A-to-TT}.
	\cref{neighbours-are-marked} for $\psi$ can be checked similarly.
\end{proof}

Let $m \ge 1$ and $n \ge 0$.
Observe that the square below is a pushout in $\mcSet$:
\[
\begin{tikzcd}
\coprod \cube m
\arrow [r]
\arrow [d] &
\cube m \pseudo \cube n
\arrow [d] \\
\coprod \mcube m
\arrow [r] &
\mcube m \pseudo \cube n
\end{tikzcd}
\]
where the coproducts are taken over all $[1]^0 \to [1]^n$.
This pushout is preserved by $T$, so the right square in 
\[
\begin{tikzcd}[row sep = large]
\coprod \Delta^m
\arrow [d]
\arrow [r, "\coprod \iota_m"] &
\coprod T(\cube m)
\arrow [r]
\arrow [d] &
T(\cube m \pseudo \cube n)
\arrow [d]\\
\coprod \mDelta{m}
\arrow [r] &
\coprod T(\mcube m)
\arrow [r] &
T(\mcube m \pseudo \cube n)
\end{tikzcd}
\]
is a pushout in $\PreComp$.
The left square is also a pushout by \cref{tensor-power}, so the pasted square is a pushout too.
Let $A'$ denote the ``corresponding'' pushout in $\mSet$ (and not in $\PreComp$):
\[
\begin{tikzcd}[row sep = large]
\coprod \Delta^m
\arrow [r]
\arrow [d]
\arrow [dr, phantom, "\mathrm{p.o.}"] &
A
\arrow [d]\\
\coprod \mDelta m
\arrow [r] &
A'
\end{tikzcd}
\]
so that its pre-complicial reflection $(A')^\precomp$ is precisely $T(\mcube m \pseudo \cube n)$.
\begin{lem}\label{marked-in-Ap'}
	The marked simplicial set $A'$ is obtained from $A$ by marking those $m$-simplices $\phi$ such that $\phi(i) = i$ for $1 \le i \le m$ and $\phi(i) \in \{\pm\infty\}$ for $i > m$.	
\end{lem}

The unmarked simplices in $B' = T(\mcube m)\pseudo T(\cube n)$ admit a characterization similar to \cref{unmarked-in-TpT}.

\begin{lem}\label{unmarked-in-TpT'}
	An $r$-simplex in $T(\mcube m) \pseudo T(\cube n)$ with $r \neq m$ is unmarked if and only if it is unmarked in $T(\cube m) \pseudo T(\cube n)$.
	An $m$-simplex $\phi$ in $T(\mcube m) \pseudo T(\cube n)$ is unmarked if and only if there exist
	\[
	m+1 \le i_1 < \dots < i_m \le m+n
	\]
	such that $\phi(i_p) = p$ for all $1 \le p \le m$.
\end{lem}

\begin{lem}\label{Ap'-to-TPT'}
	There is a complicial marking extension $A' \to T(\mcube m) \pseudo T(\cube n)$ that commutes with the evident inclusions of $(\Delta^1)^{\otimes(m+n)}$.
\end{lem}
\begin{proof}
	We first check that \cref{diagonality-zero} holds.
	Let $\phi$ be a marked $r$-simplex in $T(\mcube m) \pseudo T(\cube n)$ with $\mathcal{D}(\phi) = 0$.
	Note that, if $\phi$ is marked in $T(\cube m) \pseudo T(\cube n)$ then we already know that it is marked in $A$ (and so in $A'$ too).
	So suppose otherwise.
	Then we can see from \cref{unmarked-in-TpT,unmarked-in-TpT'} that we must have $r=m$ and a sequence
	\[
	1 \le i_1 < \dots < i_m \le m
	\]
	such that $\phi(i_p) = p$ for all $1 \le p \le m$.
	It is then easy to check that $\phi$ is one of the extra marked simplices described in \cref{marked-in-Ap'}.
	
	For \cref{neighbours-are-marked}, we assume for contradiction that $\chi$ or $\psi$ is unmarked, obtain a sequence using \cref{unmarked-in-TpT'}, and deduce that $\phi$ is unmarked.
	The details are similar to those in the proofs of \cref{A-to-TT,Ap-to-TpT}.
\end{proof}

\begin{proof}[Proof of \cref{T-strong-monoidal-pseudo}]
	Since both $T(-\pseudo-)$ and $T(-)\pseudo T(-)$ preserve colimits in each variable, it suffices to check the existence of natural isomorphisms $T(X \pseudo Y)\cong T(X)\pseudo T(Y)$ for $X,Y$ generic (possibly marked) cubes.
	
	For the appropriate values of $m$ and $n$, we may obtain isomorphisms
	\begin{align*}
	T(\cube m \pseudo \cube n) &\cong T(\cube m) \pseudo T(\cube n),\\
	T(\mcube m \pseudo \cube n) &\cong T(\mcube m) \pseudo T(\cube n),\\
	T(\cube m \pseudo \mcube n) &\cong T(\cube m) \pseudo T(\mcube n)
	\end{align*}
	by reflecting to $\PreComp$ the complicial marking extensions of \cref{Ap-to-TpT,Ap'-to-TPT'} and the dual of the latter respectively.
	
	Let $m,n \ge 1$.
	Observe that the left square below is a pushout in $\mcSet$ by \cref{Gray-tensor-of-monos}(3):
	\[
	\begin{tikzcd}[row sep = large]
	\cube m \pseudo \cube n
	\arrow [r]
	\arrow [d] &
	\mcube m \pseudo \cube n
	\arrow [d] &
	T(\cube m \pseudo \cube n)
	\arrow [r]
	\arrow [d] &
	T(\mcube m \pseudo \cube n)
	\arrow [d] \\
	\cube m \pseudo \mcube n
	\arrow [r] &
	\mcube m \pseudo \mcube n &
	T(\cube m \pseudo \mcube n)
	\arrow [r] &
	T(\mcube m \pseudo \mcube n)
	\end{tikzcd}
	\]
	Since $T$ is cocontinuous, it follows that the right square is a pushout in $\PreComp$.
	On the other hand, since both $T(\cube m) \to T(\mcube m)$ and $T(\cube n) \to T(\mcube n)$ are entire, the square below is a pushout in $\PreComp$ by \cref{pseudo-Gray-of-entire-maps}:
	\[
	\begin{tikzcd}[row sep = large]
	T(\cube m) \pseudo T(\cube n)
	\arrow [r]
	\arrow [d] &
	T(\mcube m) \pseudo T(\cube n)
	\arrow [d] \\
	T(\cube m) \pseudo T(\mcube n)
	\arrow [r] &
	T(\mcube m) \pseudo T(\mcube n)
	\end{tikzcd}
	\]
	Thus by comparing the two pushout squares in $\PreComp$, we obtain $T(\mcube m \pseudo \mcube n) \cong T(\mcube m) \pseudo T(\mcube n)$.
	The naturality of these isomorphisms is evident, and this completes the proof.
\end{proof}

\section{Triangulating model structures} \label{sec:triangulating-quillen}
The main theorem of our final section is the following.

\begin{thm}\label{left-Quillen}
	The adjunction $T \dashv U$ is a Quillen adjunction with respect to the comical model structure on $\mcSet$ and the complicial model structure on $\PreComp$.
\end{thm}
In the following proof, we denote a non-degenerate $r$-simplex $\phi \colon  \{1,2,3\} \to \{1,\dots,r,\pm\infty\}$ in the simplicial set $(\Delta^1)^{\times 3}$ by the sequence $\phi(1)\phi(2)\phi(3)$ and omitting the letter $\infty$.
For instance, $21-$ denotes the 2-simplex $\phi$ given by $\phi(1) = 2$, $\phi(2) = 1$ and $\phi(3) = -\infty$.
Note that since $\phi$ is assumed to be non-degenerate, the dimension of $\phi$ can be recovered as the maximum integer appearing in the sequence.
\begin{proof}
	We first show that $T$ preserves cofibrations.
	It suffices to prove that $T$ sends the boundary inclusions and the markers in $\mcSet$ to monomorphisms in $\mSet$.
	Clearly $T$ sends the boundary inclusions $\partial \cube 0 \hookrightarrow \cube 0$ and $\partial \cube 1 \hookrightarrow \cube 1$ to (maps that are isomorphic to) the boundary inclusions $\partial\Delta^0 \hookrightarrow \Delta^0$ and $\partial \Delta^1 \hookrightarrow \Delta^1$ respectively.
	For any $n \ge 2$, we have
	\begin{align*}
	T\bigl(\partial \cube n \hookrightarrow \cube n\bigr)  &\cong T\left(\bigl(\partial \cube 1 \hookrightarrow \cube 1\bigr)^{\hat \otimes n}\right)\\
	&\cong \left(T\bigl(\partial \cube 1 \hookrightarrow \cube 1\bigr)\right)^{\hat \otimes n}\\
	&\cong \bigl(\partial \Delta^1 \hookrightarrow \Delta^1 \bigr)^{\hat \otimes n}
	\end{align*}
	by \cref{geom-prod-of-bdry,T-strong-monoidal-lax}, and the last map is clearly a monomorphism.
	Also, $T$ sends the marker $\cube n \hookrightarrow \mcube n$ to the monomorphism $(\Delta^1)^{\otimes n} \to \trunc{n-1}\bigl((\Delta^1)^{\otimes n}\bigr)$ by definition.
	This shows that $T$ preserves cofibrations.
	
	Next we show that $T$ sends the open box inclusions to trivial cofibrations.
	We will check this ``by hand'' on the boxes of dimension $\le 3$.
	This will imply the general case since the higher dimensional box inclusions are generated by these low dimensional ones in the sense of \cref{elementary-boxes}, $T$ is strong monoidal with respect to the lax Gray tensor products (\cref{T-strong-monoidal-lax}), and the Leibniz Gray tensor product of a complicial horn inclusion and a monomorphism (in $\PreComp$) may be obtained as a composite of pushouts of complicial horn inclusions \cite[Lem.~72]{verity:weak-complicial-1}.
	
	Clearly $T$ sends $\obox{1}{1}{\epsilon} \hookrightarrow \admcube{1}{1}{\epsilon}$ to the trivial cofibration $\Lambda^1_{1-\epsilon} \incl \admdelta{1}{1-\epsilon}$.
	Consider the open box inclusion $\obox{2}{1}{0} \incl \admcube{2}{1}{0}$.
	Its image under $T$ may be written as a pushout of the horn inclusion $\Lambda^2_1 \incl \Delta^2_1$ followed by a pushout of $\Lambda^2_2 \incl \Delta^2_1$.
	The following pictures (in which thick arrows indicate marked simplices) depict this factorization:
	\[
	\left\{\begin{tikzpicture}[baseline = 15, scale = 1.3]
	\filldraw
	(0,0) circle [radius = 1pt]
	(1,0) circle [radius = 1pt]
	(0,1) circle [radius = 1pt]
	(1,1) circle [radius = 1pt];
	\draw[->] (0.2,1) -- (0.8,1);
	\draw[->] (1,0.8) -- (1,0.2);
	\draw[->, line width = 1.3pt] (0.2,0) -- (0.8,0);
	\end{tikzpicture}\right\}
	\hspace{10pt}\incl\hspace{10pt}
	\left\{\begin{tikzpicture}[baseline = 15, scale = 1.3]
	\filldraw
	(0,0) circle [radius = 1pt]
	(1,0) circle [radius = 1pt]
	(0,1) circle [radius = 1pt]
	(1,1) circle [radius = 1pt];
	\draw[->] (0.2,1) -- (0.8,1);
	\draw[->] (1,0.8) -- (1,0.2);
	\draw[->] (0.2,0.8) -- (0.8,0.2);
	\draw[->, line width = 1.3pt] (0.2,0) -- (0.8,0);
	\draw[->, line width = 1.3pt, double] (0.6,0.6) -- (0.8,0.8);
	\end{tikzpicture}\right\}
	\hspace{10pt}\incl\hspace{10pt}
	\left\{\begin{tikzpicture}[baseline = 15, scale = 1.3]
	\filldraw
	(0,0) circle [radius = 1pt]
	(1,0) circle [radius = 1pt]
	(0,1) circle [radius = 1pt]
	(1,1) circle [radius = 1pt];
	\draw[->] (0.2,1) -- (0.8,1);
	\draw[->] (1,0.8) -- (1,0.2);
	\draw[->] (0.2,0.8) -- (0.8,0.2);
	\draw[->] (0,0.8) -- (0,0.2);
	\draw[->, line width = 1.3pt] (0.2,0) -- (0.8,0);
	\draw[->, line width = 1.3pt, double] (0.6,0.6) -- (0.8,0.8);
	\draw[->, line width = 1.3pt, double] (0.4,0.4) -- (0.2,0.2);
	\end{tikzpicture}\right\}
	\]
	The box inclusions $\obox{2}{k}{\epsilon} \incl \admcube{2}{k}{\epsilon}$ for other values of $k$ and/or $\epsilon$ can be treated similarly.
	
	Now consider the open box inclusion $\obox 3 2 0 \incl \admcube 3 2 0$.
	Observe that the only marked, non-degenerate cubes in $\admcube 3 2 0$ are $\id$, $\face 1 1$, $\face 3 1$, and $\face 3 1 \face 1 1$:
	
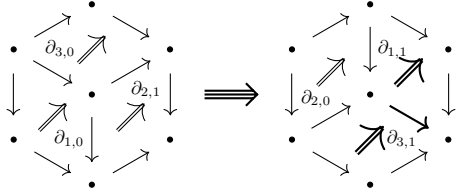
\[
	\begin{tikzpicture}[baseline = -2, scale = 0.7]
	\filldraw
	(150:2) circle [radius = 1.6pt]
	(90:2) circle [radius = 1.6pt]
	(30:2) circle [radius = 1.6pt]
	(-30:2) circle [radius = 1.6pt]
	(-90:2) circle [radius = 1.6pt]
	(-150:2) circle [radius = 1.6pt]
	(0,0) circle [radius = 1.6pt];
	
	\draw[{->}] (135:1.8) -- (105:1.8);
	\draw[{->}] (75:1.8) -- (45:1.8);
	\draw[->] (15:1.8) -- (-15:1.8);
	\draw[{->}] (165:1.8) -- (-165:1.8);	
	\draw[{->}] (-135:1.8) --(-105:1.8);
	\draw[{->}] (-75:1.8) -- (-45:1.8);
	
	\draw[->] (150:1.5) -- (150:0.5);
	\draw[->] (30:0.5) -- (30:1.5);
	\draw[->] (-90:0.5) -- (-90:1.5);
	
	\node[scale=0.7] at (-0.5,-1) {$\partial_{1,0}$};
	\node[scale=0.7] at (-0.7,1) {$\partial_{3,0}$};
	\node[scale=0.7] at (1.2,0.1) {$\partial_{2,1}$};
	
	\draw[->,double] (-150:1) + (-0.3,-0.3) --+ (0.3,0.3);
	\draw[->,double] (90:1) + (-0.3,-0.3) --+ (0.3,0.3);
	\draw[->,double] (-30:1) + (-0.3,-0.3) --+ (0.3,0.3);
	\end{tikzpicture}
	\hspace{10pt}
	\begin{tikzpicture}[baseline = -2]
	\draw[line width = 1.3pt]
	(0,0) -- (0.7,0)
	(0,0.1)--(0.6,0.1)
	(0,-0.1)--(0.6,-0.1)
	(0.5,0.2)--(0.7,0)--(0.5,-0.2);
\end{tikzpicture}
	\hspace{10pt}
	\begin{tikzpicture}[baseline = -2, scale = 0.7]
	\filldraw
	(150:2) circle [radius = 1.6pt]
	(90:2) circle [radius = 1.6pt]
	(30:2) circle [radius = 1.6pt]
	(-30:2) circle [radius = 1.6pt]
	(-90:2) circle [radius = 1.6pt]
	(-150:2) circle [radius = 1.6pt]
	(0,0) circle [radius = 1.6pt];
	
	\draw[{->}] (135:1.8) -- (105:1.8);
	\draw[{->}] (75:1.8) -- (45:1.8);
	\draw[->] (15:1.8) -- (-15:1.8);
	\draw[{->}] (165:1.8) -- (-165:1.8);	
	\draw[{->}] (-135:1.8) --(-105:1.8);
	\draw[{->}] (-75:1.8) -- (-45:1.8);

	\draw[->] (90:1.5) -- (90:0.5);
	\draw[->] (-150:1.5) -- (-150:0.5);
	\draw[->,line width = 1.3pt] (-30:0.5) -- (-30:1.5);
	
	\node[scale=0.7] at (0.5,1) {$\partial_{1,1}$};
	\node[scale=0.7] at (0.7,-1) {$\partial_{3,1}$};
	\node[scale=0.7] at (-1.2,-0.1) {$\partial_{2,0}$};
	
	\draw[->,double] (150:1) + (-0.3,-0.3) --+ (0.3,0.3);
	\draw[->,double,line width = 1.3pt] (-90:1) + (-0.3,-0.3) --+ (0.3,0.3);
	\draw[->,double,line width = 1.3pt] (30:1) + (-0.3,-0.3) --+ (0.3,0.3);
	\end{tikzpicture}
	\]
	Let $B'$ denote the marked simplicial set obtained from $(\Delta^1)^{\otimes 3}$ by marking the 3-simplex $\iota_3 =123$, the 2-simplices $12-$ and $-12$, and the 1-simplex $-1-$.
	Then the pre-complicial reflection $(B')^\precomp$ is precisely $T(\admcube 3 2 0)$.
	Observe that the $3$-simplex $231$ specifies a map $\admdelta 3 1 \to B'$.
	So if we mark its $1$st face $121$, the resulting object $B$ has the same pre-complicial reflection as $B'$.
	We will adopt $B$ (rather than $B'$) as our ``model'' for $T(\admcube 3 2 0)$.
	For the open box, we define $A$ to be the regular subset of $B$ (or $B'$) consisting of those simplices $\phi$ such that:
	\begin{itemize}
		\item $\phi(1) \in \{\pm\infty\}$;
		\item $\phi(2) = -\infty$; or
		\item $\phi(3) \in \{\pm\infty\}$
	\end{itemize}
	so that the pre-complicial reflection $A^\precomp$ is $T(\obox 3 2 0)$.
	Then we have a sequence of inclusions
	\[
	A = A^0 \incl A^1 \incl \dots \incl A^8 = B
	\]
	where $A^{s-1} \incl A^s$ is the pushout of a suitable trivial cofibration as indicated in \cref{table}.
	\begin{table}
		\begin{tabular}{|c||c|c|c|}\hline
			$s$ & pushout of & interior & missing face\\\hline
			1 & $\Lambda^2_1 \incl \Delta^2_1$ & $211$ & $111$\\
			2 & $\Lambda^2_1 \incl \Delta^2_1$ & $2\hspace{-2pt}+\hspace{-2pt}1$ & $1\hspace{-2pt}+\hspace{-2pt}1$\\
			3 & $\Lambda^3_2 \incl \Delta^3_2$ & $312$ & $212$\\
			4 & $\Lambda^3_1 \incl \Delta^3_1$ & $213$ & $112$\\
			5 & $\Lambda^3_2 \incl \Delta^3_2$ & $123$ & $122$\\
			6 & $\trunc 1 \Lambda^3_2 \incl \admdeltapp 3 2$ & $321$ & $221$\\
			7 & $\trunc 1 \Lambda^3_1 \incl \admdeltapp 3 1$ & $231$ & $121$\\
			8 & $\Lambda^3_3 \incl \Delta^3_3$ & $132$ & $1\hspace{-2pt}+\hspace{-2pt}2$\\\hline
		\end{tabular}
	\caption{Inclusions $A^{s-1} \incl A^s$}\label{table}
\end{table}
This table is to be interpreted as saying, for example, that the inclusion $A^0 \incl A^1$ fits into the pushout square
\[
\begin{tikzcd}
\Lambda^2_1
\arrow [r]
\arrow [d, hook] &
A^0
\arrow [d,hook] \\
\Delta^2_1
\arrow [r] &
A^1
\end{tikzcd}
\]
in $\mSet$ where the composite $\Delta^2_1 \to A^1 \incl B$ corresponds to the simplex $\phi = 211$, and the face $\phi   \faces 1$ corresponding to the missing face in the horn is $111$.
One can check that every non-degenerate face in $B \setminus A$ appears exactly once in \cref{table}, and moreover it is marked if and only if it appears either in the ``interior'' column or in the sixth or seventh row.
It is also straightforward to verify using \cref{tensor-power} that, for each $1 \le s \le 8$, the marked simplicial set $A^{s-1}$ indeed contains enough (marked) simplices to support a map from the domain in the ``pushout of'' column.
By reflecting everything to $\PreComp$, we can deduce that $T$ sends the box inclusion $\obox 3 2 0 \incl \admcube 3 2 0$ to a trivial cofibration.
The case $\obox 3 2 1 \incl \admcube 3 2 1$ is dual, and the other $3$-dimensional boxes can be treated using \cref{elementary-boxes}.
	
	It remains to prove that $T\bigl(\admcubep{n}{k}{\epsilon} \hookrightarrow \admcubepp{n}{k}{\epsilon}\bigr)$ is a trivial cofibration for any $n,k,\epsilon$.
	We show that this map is in fact invertible.
	Consider the following commutative diagram in $\PreComp$:
	\[
	\begin{tikzcd}
		\Delta^{n-1}
		\arrow [r, "\iota_{n-1}"]
		\arrow [d, hook] &
		T(\cube{n-1})
		\arrow [d,hook]
		\arrow [r, "T(\face k \epsilon)"] &
		T\bigl(\admcubep n k \epsilon\bigr)
		\arrow [d,hook] \\
		\mDelta{n-1}
		\arrow [r, "\iota_{n-1}", swap] &
		T(\mcube{n-1})
		\arrow [r, "T(\face k \epsilon)", swap] &
		T\bigl(\admcubepp n k \epsilon\bigr)
	\end{tikzcd}
	\]
	The left square is a pushout by the definition of $T$, and the right square is a pushout because it is the image of a pushout square under $T$.
	Thus, to show that $T\bigl(\admcubep{n}{k}{\epsilon} \hookrightarrow \admcubepp{n}{k}{\epsilon}\bigr)$ is invertible, it suffices to prove that the $(n-1)$-simplex in $T\bigl(\admcubep{n}{k}{\epsilon}\bigr)$ corresponding to the top row above is marked.
	Note that, by unwinding the above argument, one can express $T\bigl(\obox{n}{k}{\epsilon} \hookrightarrow \admcube{n}{k}{\epsilon}\bigr)$ as a composite 
	\[
	T(\obox n k \epsilon) = X^0 \incl X^1 \incl \dots \incl X^N = T(\admcube n k \epsilon)
	\]
	where each map $X^{s-1} \incl X^s$ is a pushout (in $\PreComp$) of the pre-complicial reflection of a complicial horn inclusion.
	We show by induction on $s$ that all $(n-1)$-simplices contained in $X^s$ are marked in $T\bigl(\admcubep n k \epsilon\bigr)$.
	
	For the base case, write $\admcubep n k \epsilon$ as a pushout
	\[
	\begin{tikzcd}
	\coprod	\cube m
	\arrow [r]
	\arrow [d] &
	\cube n
	\arrow [d]\\
	\coprod \mcube m
	\arrow [r] &
	\admcubep n k \epsilon
	\end{tikzcd}
	\]
	where the coproducts are taken over all marked faces of $\admcubep n k \epsilon$, which in particular include all faces $\face \ell \eta$ of codimension 1 with $(\ell,\eta) \neq (k,\epsilon)$.
	By applying $T$ to this pushout square, we can deduce that any $(n-1)$-simplex of the form
	\[
	\begin{tikzcd}[column sep = large]
	\Delta^{n-1}
	\arrow [r, "\iota_{n-1}"] &
	(\Delta^1)^{\otimes (n-1)} = T(\cube{n-1})
	\arrow [r, "T(\face \ell \eta)"] &
	T(\cube n)
	\end{tikzcd}
	\]
	is marked in $T\bigl(\admcubep n k \epsilon\bigr)$.
	By combining this observation with \cref{tensor-power}, one can deduce that any $(n-1)$-simplex contained in $X^0$ is marked in $T\bigl(\admcubep n k \epsilon\bigr)$.
	
	For the inductive step, suppose that all $(n-1)$-simplices contained in $X^{s-1}$ are marked in $T\bigl(\admcubep n k \epsilon\bigr)$.
	Suppose further that $X^s$ contains a non-degenerate $(n-1)$-simplex $\phi$ that $X^{s-1}$ does not contain (for otherwise we are done).
	Then $X^{s-1} \incl X^s$ fits into either a pushout square of the form
	\[
	\begin{tikzcd}
	(\Lambda^{n-1}_\ell)^\precomp
	\arrow [r]
	\arrow [d,hook] &
	X^{s-1}
	\arrow [d,hook]\\
	(\Delta^{n-1}_\ell)^\precomp
	\arrow [r, "\phi"] &
	X^s
	\end{tikzcd}
	\]
	or one of the form
	\[
	\begin{tikzcd}
	(\Lambda^n_\ell)^\precomp
	\arrow [r]
	\arrow [d,hook] &
	X^{s-1}
	\arrow [d,hook]\\
	(\Delta^n_\ell)^\precomp
	\arrow [r, "\chi"] &
	X^s
	\end{tikzcd}
	\]
	with $\chi   \faces \ell = \phi$.
	In the former case, $\phi$ is marked in $X^s$ and hence in $T\bigl(\admcubep n k \epsilon \bigr)$ since the unique non-degenerate $(n-1)$-simplex in $\Delta^{n-1}_\ell$ is marked.
	In the latter case, the inductive hypothesis implies that $\chi$ extends to the marked simplicial set $\admdeltap n \ell$.
	Since $X^s$ is a pre-complicial set, it follows that $\phi = \chi   \faces \ell$ is marked in $X^s$ and hence in $T\bigl(\admcubep n k \epsilon \bigr)$.
	This completes the proof.
\end{proof}

The saturated and $n$-trivial versions can be proved analogously.

\begin{thm}\label{left-Quillen-n}
	The adjunction $T \dashv U$ is a Quillen adjunction when $\mcSet$ and $\PreComp$ are respectively equipped with:
	\begin{itemize}
		\item the saturated comical model structure and the saturated complicial model structure,
		\item the $n$-trivial comical model structure and the $n$-trivial complicial model structure for some $0 \le n < \infty$, or
		\item the saturated $n$-trivial comical model structure and the saturated $n$-trivial complicial model structure for some $0 \le n < \infty$.
	\end{itemize}
\end{thm}
\begin{proof}
	The proof is analogous to that of \cref{left-Quillen}.
	For the $n$-trivial versions, observe that $T$ sends the (cubical) $m$-marker to a pushout of the (simplicial) $m$-marker.
	
	For the saturated versions, we only check that $T$ sends the basic Rezk maps to trivial cofibrations.
	(That the higher Rezk maps are also sent to trivial cofibrations then follows from \cref{PreComp-monoidal-model,T-strong-monoidal-lax}.)
	Note that $T$ sends all four basic Rezk maps to the same map (up to isomorphism).
	This unique image, which we denote $TL \incl TL'$, may be visualised as follows:
	\[
	TL = 
	\left\{\begin{tikzpicture}[baseline = 18, scale = 1.5]
		\foreach \x in {0,1,2}
		\filldraw (\x,0) circle [radius = 1pt] (\x,1) circle [radius = 1pt];
		\draw[->] (0.2,1) -- (0.8,1);
		\draw[->] (1,0.8) -- (1,0.2);
		\draw[->] (1.2,0) -- (1.8,0);
		\draw[->, line width = 1.3pt] (0,0.8) -- (0,0.2);
		\draw[->, line width = 1.3pt] (0.2,0) -- (0.8,0);
		\draw[->, line width = 1.3pt] (1.2,1) -- (1.8,1);
		\draw[->, line width = 1.3pt] (2,0.8) -- (2,0.2);
		\draw[->, line width = 1.3pt] (0.2,0.8) -- (0.8,0.2);
		\draw[->, line width = 1.3pt] (1.2,0.8) -- (1.8,0.2);
		\draw[->, double, line width = 1.3pt] (0.6,0.6) -- (0.8,0.8);
		\draw[->, double, line width = 1.3pt] (0.4,0.4) -- (0.2,0.2);
		\draw[->, double, line width = 1.3pt] (1.6,0.6) -- (1.8,0.8);
		\draw[->, double, line width = 1.3pt] (1.4,0.4) -- (1.2,0.2);
	\end{tikzpicture}\right\},
\hspace{20pt}
	TL' = 
	\left\{\begin{tikzpicture}[baseline = 18, scale = 1.5]
		\foreach \x in {0,1,2}
		\filldraw (\x,0) circle [radius = 1pt] (\x,1) circle [radius = 1pt];
		\draw[->, line width = 1.3pt] (0.2,1) -- (0.8,1);
		\draw[->, line width = 1.3pt] (1,0.8) -- (1,0.2);
		\draw[->, line width = 1.3pt] (1.2,0) -- (1.8,0);
		\draw[->, line width = 1.3pt] (0,0.8) -- (0,0.2);
		\draw[->, line width = 1.3pt] (0.2,0) -- (0.8,0);
		\draw[->, line width = 1.3pt] (1.2,1) -- (1.8,1);
		\draw[->, line width = 1.3pt] (2,0.8) -- (2,0.2);
		\draw[->, line width = 1.3pt] (0.2,0.8) -- (0.8,0.2);
		\draw[->, line width = 1.3pt] (1.2,0.8) -- (1.8,0.2);
		\draw[->, double, line width = 1.3pt] (0.6,0.6) -- (0.8,0.8);
		\draw[->, double, line width = 1.3pt] (0.4,0.4) -- (0.2,0.2);
		\draw[->, double, line width = 1.3pt] (1.6,0.6) -- (1.8,0.8);
		\draw[->, double, line width = 1.3pt] (1.4,0.4) -- (1.2,0.2);
	\end{tikzpicture}\right\}
	\]
	Let $A$ (resp.~$A'$) be the regular subset of $TL$ (resp.~$TL'$) consisting of the middle two non-degenerate $2$-simplices so that they look like:
	\[
	A = 
	\left\{\begin{tikzpicture}[baseline = 18, scale = 1.5]
		\filldraw
		(1,0) circle [radius = 1pt]
		(2,0) circle [radius = 1pt]
		(0,1) circle [radius = 1pt]
		(1,1) circle [radius = 1pt];
		\draw[->] (0.2,1) -- (0.8,1);
		\draw[->] (1,0.8) -- (1,0.2);
		\draw[->] (1.2,0) -- (1.8,0);
		\draw[->, line width = 1.3pt] (0.2,0.8) -- (0.8,0.2);
		\draw[->, line width = 1.3pt] (1.2,0.8) -- (1.8,0.2);
		\draw[->, double, line width = 1.3pt] (0.6,0.6) -- (0.8,0.8);
		\draw[->, double, line width = 1.3pt] (1.4,0.4) -- (1.2,0.2);
	\end{tikzpicture}\right\},
	\hspace{20pt}
	A' = 
	\left\{\begin{tikzpicture}[baseline = 18, scale = 1.5]
		\filldraw
		(1,0) circle [radius = 1pt]
		(2,0) circle [radius = 1pt]
		(0,1) circle [radius = 1pt]
		(1,1) circle [radius = 1pt];
		\draw[->, line width = 1.3pt] (0.2,1) -- (0.8,1);
		\draw[->, line width = 1.3pt] (1,0.8) -- (1,0.2);
		\draw[->, line width = 1.3pt] (1.2,0) -- (1.8,0);
		\draw[->, line width = 1.3pt] (0.2,0.8) -- (0.8,0.2);
		\draw[->, line width = 1.3pt] (1.2,0.8) -- (1.8,0.2);
		\draw[->, double, line width = 1.3pt] (0.6,0.6) -- (0.8,0.8);
		\draw[->, double, line width = 1.3pt] (1.4,0.4) -- (1.2,0.2);
	\end{tikzpicture}\right\}
	\]
	Clearly $TL \incl TL'$ is a pushout of its restriction $A \incl A'$, so it suffices to show that the latter is a trivial cofibration.
	
	Observe that $A$ is isomorphic to the regular subset of $\presat$ consisting of $\faces 0$ and $\faces 3$.
	One can check that the inclusion $A \incl \presat$ may be written as the composite of a pushout of $\horn 1 2 \incl \cDelta 1 2$ (attaching $\faces 1$) and a pushout of $\horn 2 3 \incl \cDelta 2 3$.
	Hence $A \incl \presat$ is complicial.
	Similarly, $A' \incl (\Delta^3)^\sharp$ is the composite of pushouts of two complicial horn inclusions and one elementary complicial marking extension (marking the $1$-simplex $\{0,3\}$), so it is complicial too.
	Since the square
	\[
	\begin{tikzcd}
		A
		\arrow [r, hook]
		\arrow [d, hook] &
		A'
		\arrow [d, hook] \\
		\presat
		\arrow [r, hook] &
		(\Delta^3)^\sharp
	\end{tikzcd}
	\]
	commutes, the desired conclusion now follows by the 2-out-of-3 property.
\end{proof}

%
%


\bibliographystyle{amsalphaurlmod}
\bibliography{all-refs}

\end{document}